%% file: main.tex
\documentclass{article}
\usepackage[top=1in, bottom=1.25in, left=1.3in, right=1.3in]{geometry}
\usepackage{xr}
\usepackage{subfiles}
\usepackage{amsmath,amssymb,amsthm}
\usepackage{microtype} 
\usepackage[dvipsnames]{xcolor}
\definecolor{links}{HTML}{006E9A}
\definecolor{urls}{HTML}{9B7A00}
\usepackage[colorlinks=true,
            citecolor=links,
            linkcolor=links,
            urlcolor=urls,
            pdfusetitle]{hyperref}
\usepackage[style=alphabetic,natbib=true,maxalphanames=7,maxcitenames=2,maxbibnames=99]{biblatex}
\usepackage[capitalize,noabbrev]{cleveref}
\usepackage[greek,english]{babel}
\usepackage{bm}
\usepackage{datetime}
\usepackage[mathscr]{eucal}
\usepackage[inline]{enumitem}
\usepackage{environ}
\usepackage{etoolbox}
\usepackage[T2A]{fontenc}
\usepackage{manfnt}
\usepackage{mathtools}
\usepackage{mathpartir}
\usepackage{pict2e}
\usepackage{pifont}
\usepackage{scalerel}
\usepackage{stmaryrd}
\usepackage{tikz}
\usepackage{tikz-cd}
\usepackage{upgreek}
\usepackage{xparse}
\usepackage{xspace}

\input{macros}

\hypersetup{
    pdftitle={The algebraic small object argument as a saturation},
    pdfauthor={Evan Cavallo, Christian Sattler}
}
\title{The algebraic small object argument as a saturation}

\author{
  Evan Cavallo\footremember{cgu}{University of Gothenburg and Chalmers University of Technology, Gothenburg, Sweden} \\
  \href{mailto://evan.cavallo@gu.se}{\texttt{evan.cavallo@gu.se}}
  \and
  Christian Sattler\footrecall{cgu} \\
  \href{mailto://sattler@chalmers.se}{\texttt{sattler@chalmers.se}}
}

\date{October 2025}

\begin{document}

\maketitle

\begin{abstract}
  We analyze the structure of left maps in algebraic weak factorization systems constructed using Garner's algebraic small object argument.
  We find that any left map can be constructed from generators in Bourke and Garner's double category of left maps by operations that parallel the classical cell-complex-forming operations of Quillen's small object argument (coproducts, cobase changes, transfinite composites, and retracts).
  Our main theorems are phrased as ``saturation'' principles, which express the closure conditions necessary for a given property or structure to extend from generators to all left maps.
  The core of the argument is an analysis of the construction of the free monad on a pointed endofunctor.
\end{abstract}

\setcounter{tocdepth}{2}
\tableofcontents

\subfile{introduction}

\subfile{free-monad-sequence}

\subfile{soa}

\subfile{applications}

\subfile{related}

\section{Acknowledgments}

We thank Dani\"el Apol and Ivan Di Liberti for helpful conversations.
We thank Benno van den Berg for calling our attention to problems with our first draft's description of constructive transfinite recursion and for sharing his own approach with us.

The first author was supported by the Knut and Alice Wallenberg Foundation (KAW) under Grant No.\ 2019.0116, and the second author was supported by the US Air Force Office of Scientific Research under award number FA9550-24-1-0302.

\appendix

\subfile{metatheory}

\printbibliography[heading=bibintoc]

\end{document}

%% file: macros.tex
\usepackage{tikz}
\usetikzlibrary{arrows,matrix,decorations.pathmorphing,
  decorations.markings, calc, backgrounds,patterns}

\setlist[enumerate,0]{label=(\alph*)}

\newcommand{\footremember}[2]{%
  \footnote{#2}
  \newcounter{#1}
  \setcounter{#1}{\value{footnote}}%
}
\newcommand{\footrecall}[1]{%
  \footnotemark[\value{#1}]%
}

\DeclarePairedDelimiter\braces\lbrace\rbrace
\DeclarePairedDelimiterX\set[2]\lbrace\rbrace{#1 \mathrel{\delimsize\vert} #2}

\makeatletter
\newcommand{\DeclareAbbrevation}[2]{\newcommand{#1}{\@ifnextchar{.}{\emph{#2}}{\emph{#2}.\@\xspace}}}
\makeatother

\DeclareAbbrevation{\ie}{i.e}
\DeclareAbbrevation{\eg}{e.g}
\DeclareAbbrevation{\cf}{cf}
\DeclareAbbrevation{\etc}{etc}
\DeclareAbbrevation{\resp}{resp}
\DeclareAbbrevation{\etal}{et al}
\DeclareAbbrevation{\ibid}{ibid}
\DeclareAbbrevation{\opcit}{op.~cit}
\DeclareAbbrevation{\ca}{ca}

\theoremstyle{definition}
\newtheorem{definition}{Definition}[subsection]
\newtheorem{lemma}[definition]{Lemma}
\newtheorem{theorem}[definition]{Theorem}
\newtheorem{proposition}[definition]{Proposition}
\newtheorem{corollary}[definition]{Corollary}

\newtheorem{example}[definition]{Example}

\newtheorem{terminology}[definition]{Terminology}
\newtheorem{remark}[definition]{Remark}
\newtheorem{notation}[definition]{Notation}

\newtheorem{manualtheoreminner}{Theorem}
\newenvironment{manualtheorem}[1]{%
  \renewcommand\themanualtheoreminner{#1}%
  \manualtheoreminner
}{\endmanualtheoreminner}

\numberwithin{equation}{section}

\mathchardef\mh="2D

\newcommand{\cat}[1]{\mathscr{#1}}
\newcommand{\dcat}[1]{\mathbb{#1}}
\newcommand{\class}[1]{\mathcal{#1}}
\newcommand{\ptdendo}[1]{\mathsf{#1}}
\newcommand{\monad}[1]{\mathsf{#1}}
\newcommand{\comonad}[1]{\mathsf{#1}}

\newcommand{\Ga}{\ensuremath{\alpha}}
\newcommand{\Gb}{\ensuremath{\beta}}

\newcommand{\Gg}{\ensuremath{\gamma}}

\newcommand{\Ge}{\ensuremath{\epsilon}}
\newcommand{\Gh}{\ensuremath{\eta}}

\newcommand{\Gm}{\ensuremath{\mu}}

\newcommand{\Gp}{\ensuremath{\pi}}

\newcommand{\Gt}{\ensuremath{\tau}}

\newcommand{\Gth}{\ensuremath{\theta}}

\renewcommand{\qoppa}{\textgreek{\textqoppa}}

\DeclareMathAlphabet{\mathbbe}{U}{bbold}{m}{n}

\newcommand{\2}{\ensuremath{\mathbbe{2}}}

\NewDocumentCommand\YoSymScaled{m}{
  \raisebox{#1em * \real{-.02}}{%
    \includegraphics[quiet=true,draft=false,height=#1em * \real{.67},keepaspectratio]{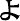}%
    \hspace{0.1em}%
  }
}

\NewDocumentCommand{\yosym}{}{
  \mathord{%
    \mathchoice{\YoSymScaled{1}}{\YoSymScaled{1}}{\YoSymScaled{\defaultscriptratio}}{\YoSymScaled{\defaultscriptscriptratio}}
  }%
}

\DeclareFontFamily{U}{mathx}{}
\DeclareFontShape{U}{mathx}{m}{n}{<-> mathx10}{}
\DeclareSymbolFont{mathx}{U}{mathx}{m}{n}
\DeclareMathAccent{\widehat}{0}{mathx}{"70}
\DeclareMathAccent{\widecheck}{0}{mathx}{"71}

\newcommand{\defeq}{\coloneqq}
\newcommand{\eqdef}{\eqqcolon}

\newcommand{\co}{\colon}
\newcommand{\card}[1]{\#{#1}}

\newcommand{\Hom}[3]{{#1}(#2,#3)}
\NewDocumentCommand{\id}{o}{\mathrm{id}\IfValueT{#1}{_{#1}}}
\NewDocumentCommand{\idv}{o}{\mathbf{id}}
\NewDocumentCommand{\Id}{o}{\mathrm{Id}\IfValueT{#1}{_{#1}}}

\DeclareMathOperator{\dom}{\mathrm{dom}}
\DeclareMathOperator{\cod}{\mathrm{cod}}

\newcommand{\op}[1]{#1^{\mathrm{op}}}

\newcommand{\pair}[2]{\langle #1, #2 \rangle}
\newcommand{\proj}[1]{\Gp_{#1}}
\newcommand{\projl}{\proj{0}}
\newcommand{\projr}{\proj{1}}

\newcommand{\copair}[2]{[ #1, #2 ]}

\NewDocumentCommand{\injsym}{}{\upsilon}

\newcommand{\inj}[1]{\injsym_{#1}}
\newcommand{\inl}{\inj{0}}
\newcommand{\inr}{\inj{1}}

\newcommand{\Func}[2]{[#1,#2]}

\newcommand{\leibpush}[1]{\mathbin{\widehat{#1}}}

\newcommand{\leibpushapp}[1]{\widehat{#1}}
\newcommand{\leibpullapp}[1]{\widecheck{#1}}

\NewDocumentCommand{\Set}{o}{\mathbf{Set}\IfValueT{#1}{_{#1}}}
\newcommand{\FinSet}{\mathbf{FinSet}}
\newcommand{\Cat}{\mathbf{Cat}}

\newcommand{\Comma}[2]{{#1} \mathbin{\downarrow} {#2}}
\newcommand{\Slice}[2]{{#1} / {#2}}
\newcommand{\Coslice}[2]{{#1} / {#2}}
\DeclareMathOperator{\catEl}{\int}

\newcommand{\Arr}[1]{#1^\shortrightarrow}
\newcommand{\ArrFull}[2]{#1^\shortrightarrow_{#2}}

\newcommand{\CartArr}[1]{#1^\shortrightarrow_{\mathrm{cart}}}
\newcommand{\CartMono}[1]{#1^{\mono}_{\mathrm{cart}}}

\newcommand{\FibOb}[2]{#2_{#1\mh\mathrm{cart}}}
\newcommand{\Adju}{\mathbf{Adj}}

\DeclareMathOperator*{\colim}{colim}

\newcommand{\Lan}{\operatorname{Lan}}

\newcommand{\Den}[1]{\mathrm{D}_{#1}} 

\NewDocumentCommand{\mono}{s}{\IfBooleanTF{#1}{\leftarrowtail}{\rightarrowtail}}
\NewDocumentCommand{\epi}{s}{\IfBooleanTF{#1}{\twoheadleftarrow}{\twoheadrightarrow}}

\makeatletter
\newcommand{\adjunction}[4]{%
  #1\colon #2%
  \mathrel{\vcenter{%
    \offinterlineskip\m@th
    \ialign{%
      \hfil$##$\hfil\cr
      \longrightarrow\cr
      \noalign{\kern-.3ex}
      \smallbot\cr
      \noalign{\kern.1ex}
      \longleftarrow\cr
    }%
  }}%
  #3 \noloc #4%
}
\newcommand{\noloc}{%
  \nobreak
  \mspace{6mu plus 1mu}
  {:}
  \nonscript\mkern-\thinmuskip
  \mathpunct{}
  \mspace{2mu}
}
\newcommand{\smallbot}{%
  \begingroup\setlength\unitlength{.25em}%
  \begin{picture}(1,1)
  \roundcap
  \polyline(0,0)(1,0)
  \polyline(0.5,0)(0.5,1)
  \end{picture}%
  \endgroup
}
\makeatother

\tikzset{
  epi/.style = {commutative diagrams/two heads},
  mono/.style = {commutative diagrams/tail},
  cof/.style = {{Triangle[reversed,scale=0.9]}->},
  fib/.style = {-{Triangle[open,scale=0.9]}},
  weq/.style = {"\sim"'{sloped,font=\tiny,#1}},
  tcof/.style = {cof,weq={#1}},
  tcof/.default = {above},
  tfib/.style = {fib,weq={#1}},
  tfib/.default = {above},
  weq/.default = {above},
  bifib/.style = {{Triangle[reversed,scale=0.9]}-{Triangle[open,scale=0.9]}},
  verto/.style = {"\bullet"'{marking,scale=0.7,#1}},
}

\tikzcdset{
  arrow style=tikz,
  every label/.append style = {font = \small},
  phantomcenter/.style = {phantom,start anchor=center,end anchor=center}
}

\NewDocumentCommand{\pushout}{D<>{0} O{6ex}}{%
  \arrow[phantom,start anchor=center,to path={-- ++({#1+135}:#2) \tikztonodes}, "\ulcorner"]
}
\NewDocumentCommand{\pullback}{D<>{0} O{6ex}}{%
  \arrow[phantom,start anchor=center,to path={-- ++({#1-45}:#2) \tikztonodes}, "\lrcorner"]
}

\newcommand{\placeholder}{\scalerel*{\tikz[baseline=0pt,color=black!60!white]{\draw (0,0) circle (3pt);}}{\bigcirc}}
\newcommand{\object}{\scalerel*{\tikz[baseline=0pt]{\draw[dashed, dash pattern=on 1 off 1] (0,0) circle (3pt);}}{\bigcirc}}

\NewDocumentCommand{\PSh}{m}{\mathrm{PSh}({#1})}
\NewDocumentCommand{\yo}{g}{\yosym\IfValueT{#1}{#1}}

\newcommand{\subst}[1]{#1^*}
\newcommand{\lan}[1]{#1_!}
\newcommand{\ran}[1]{#1_*}

\NewDocumentCommand{\arrowofstyle}{m m m}{
  \NewDocumentCommand#1{o}{
    \mathrel{\hspace{-0.63ex}\tikz[baseline=-0.7ex, shorten <=3pt, shorten >=2pt]
      \draw[#2,line width=0.08ex](0,0) -- (1.6em,0)
      \IfValueT{##1}{node [midway,scale=0.75,yshift=0.5ex,xshift=#3] {$##1$}};
    \hspace{-0.41ex}}
  }
}

\arrowofstyle{\ordinary}{->}{-0.2ex}
\arrowofstyle{\ordinaryleft}{<-}{0.4ex}
\arrowofstyle{\cof}{cof}{0.0ex}
\arrowofstyle{\fib}{fib}{-0.3ex}
\NewDocumentCommand{\weq}{s}{\IfBooleanTF{#1}{\ordinaryleft[\sim]}{\ordinary[\sim]}}
\newcommand{\tcof}{\cof[\sim]}

\newcommand{\ConfWP}[1]{\operatorname{ConfMnd}_{\text{wp}}^{{#1}}}
\newcommand{\ConfP}[1]{\operatorname{ConfMnd}_{\text{p}}^{{#1}}}
\newcommand{\AlgChain}[2]{{#1}\mh\mathrm{AlgChain}_{#2}}

\newcommand{\Fun}{\mathbf{Fun}}
\newcommand{\Mnd}{\mathbf{Mnd}}
\newcommand{\PtdEndo}{\mathbf{PtdEndo}}

\newcommand{\IdPtdEndo}[1]{\mathsf{Id}_{#1}}
\newcommand{\IdMonad}[1]{\mathsf{Id}_{#1}}

\newcommand{\Alg}[1]{{#1}\mh\mathrm{Alg}}
\newcommand{\Coalg}[1]{{#1}\mh\mathrm{Coalg}}

\newcommand{\dpitchfork}{\pitchfork\mskip-9.8mu\pitchfork}
\newcommand{\drlp}[1]{#1^{\dpitchfork}}

\newcommand{\rlp}[1]{#1^{\pitchfork}}
\newcommand{\llp}[1]{{}^{\pitchfork}#1}

\NewDocumentCommand{\SplPtdEndo}{o}{\operatorname{\mathsf{Spl}}\IfValueT{#1}{^{#1}}}
\NewDocumentCommand{\SplEndo}{o}{\operatorname{\mathrm{Spl}}\IfValueT{#1}{^{#1}}}
\NewDocumentCommand{\splpt}{o}{\mathsf{spl}\IfValueT{#1}{^{#1}}}

\newcommand{\CodPtdEndo}{\mathsf{Tgt}}
\newcommand{\CodEndo}{\operatorname{Tgt}}
\newcommand{\codpt}{\mathsf{tgt}}

\newcommand{\DAlg}[1]{{#1}\mh\mathrm{\dcat{A}lg}}
\newcommand{\DCoalg}[1]{{#1}\mh\mathrm{\dcat{C}oalg}}

\newcommand{\UnderPtd}[1]{#1_{\mathrm{p}}}

\newcommand{\conjugate}[1]{{#1}_\dagger}

\newcommand{\oneStep}[1]{\mathrm{step}_{#1}}

\newcommand{\vlabel}{\mathord{\ooalign{\hidewidth$\bm{\cdot}$\hidewidth\cr$\scriptstyle\downarrow$}}}

\newcommand{\Sq}[1]{\mathrm{\mathbb{S}q}({#1})}
\newcommand{\VertArr}[1]{#1^{\vlabel}}

\newcommand{\gluelabel}{\mathrm{gl}}
\newcommand{\GlueDbl}[1]{\mathrm{\mathbb{G}l}({#1})}
\newcommand{\glueconc}[1]{{#1}_{\gluelabel}}
\newcommand{\gluedom}{\dom_{\gluelabel}}

\newcommand{\DComma}[2]{{#1} \mathbin{\downarrow} {#2}}

\NewDocumentCommand{\verto}{s}{
  \mathrel{\hspace{-0.63ex}\tikz[baseline=-0.7ex, shorten <=3pt, shorten >=2pt, decoration={markings, mark=at position \IfBooleanTF{#1}{0.53}{0.47} with {\fill circle [radius=1.2pt];}}]{
    \draw[\IfBooleanTF{#1}{<-}{->},postaction={decorate},line width=0.08ex](0,0) -- (1.6em,0);}
    \hspace{-0.41ex}}
}

\newcommand{\vcirc}{\mathbin{\star}}

\newcommand{\vdom}{\mathrm{dom}_{\vlabel}}
\newcommand{\vcod}{\mathrm{cod}_{\vlabel}}

\DeclareMathOperator{\Retract}{\mathrm{Ret}}
\newcommand{\projret}{e_{\mathsf{0}}}
\newcommand{\projobj}{e_{\mathsf{1}}}

\DeclareMathOperator{\CodRet}{\mathrm{CodRet}}

\DeclareFontEncoding{LS2}{}{}
\DeclareFontSubstitution{LS2}{stix}{m}{n}
\DeclareSymbolFont{arrows3}{LS2}{stixtt}{m}{n}
\DeclareMathSymbol{\lrtriangle}{\mathord}{arrows3}{"9E}
\DeclareMathSymbol{\lltriangle}{\mathord}{arrows3}{"99}
\NewDocumentCommand{\leftcnx}{s}{\scalerel*{\IfBooleanTF{#1}{\lltriangle}{\lrtriangle}}{\Delta}}

\NewDocumentCommand{\DomClass}{m m}{\Class{#1}{#2}{\cong}}
\NewDocumentCommand{\DClass}{m m m}{\VertArr{#1}(\tfrac{#2}{#3})}
\NewDocumentCommand{\Class}{m m m}{\Arr{#1}(\tfrac{#2}{#3})}
\NewDocumentCommand{\CodClass}{m m}{\Class{#1}{\cong}{#2}}
\NewDocumentCommand{\LevelClass}{m m}{\Class{#1}{#2}{#2}}
\NewDocumentCommand{\DCodClass}{m m}{\DClass{#1}{\cong}{#2}}

\newcommand{\Cof}{\mathbb{F}}
\newcommand{\CofT}{1}
\newcommand{\TCof}[2]{\mathrm{TCof_p}(#1,#2)}
\newcommand{\DTCof}[2]{\mathrm{\mathbb{T}Cof_p}(#1,#2)}
\newcommand{\Fib}[2]{\mathrm{Fib}(#1,#2)}

\newcommand{\boxDiagram}[2]{\mathsf{box}^{#1}_{#2}}

\newcommand{\ival}{\mathbb{I}}

\newcommand{\LCMono}[1]{\class{M}_{\mathrm{lc}}}
\newcommand{\CartLCMono}[1]{#1^{\mathrm{lc}\mono}_{\mathrm{cart}}}

\NewDocumentCommand{\Ext}{m}{\mathrm{Ext}_{#1}}
\NewDocumentCommand{\DExt}{m}{\mathrm{\mathbb{E}xt}_{#1}}

\newcommand{\corr}[2]{\mathrm{pull}_{#1}{#2}}
\newcommand{\lift}[2]{\mathrm{lift}_{#1}{#2}}

\newcommand{\ranob}[2]{\prod_{#1}{#2}}

\newcommand{\rancob}[2]{\overline{\prod}_{#1}{#2}}
\newcommand{\ranc}[1]{#1_{\overline{*}}}

\newcommand{\FullCartArr}[1]{{#1}^{\dashrightarrow}}


%% file: introduction.tex
\section{Introduction}
\label{sec:intro}

Quillen's \emph{small object argument} \cite{quillen:67} builds, under some conditions, a weak factorization system (\textsc{wfs}) in a category $\cat{E}$ from a generating set of left maps in $\cat{E}$.
Given such a set $S \subseteq \Arr{\cat{E}}$, the class $\class{R} \defeq \rlp{S}$ of maps right lifting against $S$ and the class $\class{L} \defeq \llp{\class{R}}$ of maps left lifting against $\class{R}$ constitute a \textsc{wfs} $(\class{L},\class{R})$ exactly if every map factors as one in $\class{L}$ followed by one in $\class{R}$, in which case one says $(\class{L},\class{R})$ is \emph{cofibrantly generated by $S$}.
The small object argument constructs these factorizations.

\Citeauthor{garner:09}'s \emph{algebraic small object argument} \cite{garner:09} is a refinement of Quillen's argument that generalizes it along two axes: it takes a diagram of generators $u \co \cat{J} \to \Arr{\cat{E}}$ as input---as opposed to a set, which is the special case where $\cat{J}$ is discrete---and produces an \emph{algebraic weak factorization system} (\textsc{awfs}) as output.
A \textsc{wfs} is \emph{generated by a diagram} $u \co \cat{J} \to \Arr{\cat{E}}$ if its right maps are those $f$ equipped with a choice of filler for each lifting problem $\alpha \co u_i \to f$ such that for every $t \co j \to i$ in $\cat{J}$ and $\alpha \co u_i \to f$, the triangle
\begin{equation}
  \label{generating-diagram-coherence}
  \begin{tikzcd}[row sep=large]
    A_j \ar[phantomcenter]{dr}{u_t} \ar{d} \ar{r} & A_i \ar[phantomcenter]{dr}{\alpha} \ar{d} \ar{r} & X \ar{d}{f} \\
    B_j \ar[dashed,bend right=10]{urr} \ar{r} & B_i \ar[dashed,bend right=15]{ur} \ar{r} & Y
  \end{tikzcd}
\end{equation}
formed by the solutions for $\alpha \co u_i \to f$ and $\alpha u_t \co u_j \to f$ commutes.
This generalization captures more \textsc{wfs}'s of interest \cite[\S4.3]{riehl:11}~\cite[Example 2.6]{rosicky:17}, in particular \textsc{wfs}'s used in constructive interpretations of homotopy type theory \cite{swan:16,gambino-sattler:17,cavallo-mortberg-swan:20,awodey:23,accrs:24,cavallo-sattler:23}.
With an \emph{algebraic} \textsc{wfs}, one has \emph{categories} of left and right maps with useful properties such as closure under small colimits and limits respectively \cite{grandis-tholen:06}.

The left factor of a factorization produced by Quillen's argument is, by construction, a \emph{cell complex}: a transfinite composite of cobase changes (\ie, pushouts along arbitrary maps) of small coproducts of generating maps.
Defining a \emph{cellular class of maps} to be one closed under transfinite composition, cobase change, small coproducts, it follows that any cellular class containing the generators contains the left factor of every chosen factorization.
Furthermore, any left map can be written as a codomain retract of the left factor of a factorization \cite[Lemma 1.1.9]{hovey:99}.
Thus the class of left maps has a useful universal property: defining a \emph{saturated class of maps} to be a cellular class closed under codomain retracts, the class of left maps is the least saturated class containing the generators.

The algebraic small object argument does not come with an obvious equivalent for this story.
Even if we start from a discrete diagram of generators, it is not clear how to extract a cell complex decomposition from Garner's construction \cite[\S6.4]{garner:07}.
For a non-discrete diagram, it is not even clear what the notion of cell complex should be.

We contribute an analogue to the concept of saturated class of maps for \textsc{awfs}'s and show that the category of left maps of a cofibrantly generated \textsc{awfs} is ``least'' among the ``saturated classes'' containing the diagram that generates it.

\subsection{Results}
\label{sec:intro:results}

A \emph{functorial factorization} on a category $\cat{E}$ is a pair of functors $L,R \co \Arr{\cat{E}} \to \Arr{\cat{E}}$ such that $\dom L = \dom$, $\cod L = \dom R$, $\cod R = \cod$, and $Rf \circ Lf = f$ for each $f \in \Arr{\cat{E}}$.
An \emph{algebraic weak factorization system} or \emph{\textsc{awfs}} on $\cat{E}$ \cite{grandis-tholen:06,garner:09} is a functorial factorization $(L,R)$ equipped with extensions of the copointed and pointed endofunctors $(L,\Phi)$ and $(R,\Lambda)$ defined by the squares
\[
  \begin{tikzcd}
    X \ar[phantom]{dr}{\Phi_f} \ar{d}[left]{Lf} \ar[equals]{r} & X \ar{d}{f} \\
    Ef \ar{r}[below]{Rf} & Y
  \end{tikzcd}
  \qquad
  \text{and}
  \qquad
  \begin{tikzcd}
    X \ar[phantom]{dr}{\Lambda_f} \ar{d}[left]{f} \ar{r}{Lf} & Ef \ar{d}{Rf} \\
    Y \ar[equals]{r} & Y
  \end{tikzcd}
\]
to a comonad $\comonad{L} = (L,\Phi,\Sigma)$ and monad $\monad{R} = (R,\Lambda,\Pi)$ related by a distributive law \cite{beck:69}.
An \textsc{awfs} has an underlying \textsc{wfs} whose left maps are the coalgebras for the copointed endofunctor $\UnderPtd{\comonad{L}} = (L,\Phi)$ and whose right maps are the algebras for the pointed endofunctor $\UnderPtd{\monad{R}} = (R,\Lambda)$ {\cite[Remark 2.17]{garner:09}~\cite[Remark 2.11]{riehl:11}}.

Given an \textsc{awfs} on $\cat{E}$, \textcite{bourke-garner:16} describe a double category $\DCoalg{\UnderPtd{\comonad{L}}}$ whose objects and horizontal morphisms are those of $\cat{E}$ and whose category of vertical morphisms is the category $\Coalg{\UnderPtd{\comonad{L}}}$ of left maps. 
It lives over $\cat{E}$ via the double functor $U \co \DCoalg{\UnderPtd{\comonad{L}}} \to \Sq{\cat{E}}$ (to the ``double category of squares'' on $\cat{E}$ whose horizontal and vertical morphisms are both the morphisms of $\cat{E}$) that forgets the left structure on vertical morphisms.
Abstracting from this situation, we define a \emph{notion of composable structure on $\cat{E}$} (\cref{notion-of-composable-structure}) to be a pseudo double category $U \co \dcat{A} \to \Sq{\cat{E}}$ over $\Sq{\cat{E}}$ whose horizontal morphisms are those of $\cat{E}$ and whose vertical morphism projection $\VertArr{U} \co \VertArr{\dcat{A}} \to \Arr{\cat{E}}$ is a conservative isofibration.

A notion of composable structure $U \co \dcat{A} \to \Sq{\cat{E}}$ will be \emph{cellular} when its category of vertical morphisms $\VertArr{\dcat{A}}$ is closed under certain pushouts and sequential colimits.
Given a diagram of generators $u \co \cat{J} \to \Arr{\cat{E}}$, its \emph{density comonad} $\Den{u} \co \Arr{\cat{E}} \to \Arr{\cat{E}}$ \cite{kock:66} is the pointwise left Kan extension of $u$ along itself, defined at $f \in \Arr{\cat{E}}$ by
\begin{equation}
  \label{density-comonad-colimit}
  \Den{u}(f) \defeq \colim \left(\begin{tikzcd}[cramped] \Comma{u}{f} \ar{r}{\proj{}} & \cat{J} \ar{r}{u} & \Arr{\cat{E}}\end{tikzcd} \right) \rlap{.}
\end{equation}
When $\cat{J}$ is discrete, $\Den{u}(f)$ is the coproduct $\coprod_{i \in \cat{J}, \alpha \co u_i \to f} u_i$, and in general it plays an analogous role to the coproducts in Quillen's argument.
Importantly, it can be effectively characterized in our examples of interest (\cref{sec:uniform-fibrations}).

Our first main theorem turns any lift of $\Den{u}$ through a cellular notion of composable structure $\dcat{A}$ into a corresponding of the left factor functor $L$ of the \textsc{awfs} generated from $u$.
To be precise:

\begin{manualtheorem}{{\ref{soa-unit-vertical-point}}}
  \input{soa-unit-vertical-point}
\end{manualtheorem}

\Cref{soa-unit-vertical-point} is a corollary of a more technical result, \cref{soa-unit-vertical-point-abstract}, that expresses a universal property of the outputs.
The backdrop parameter $\class{M}$, a wide subcategory of $\cat{E}$, is new relative to the classical story and controls the kind of colimits required in $\cat{E}$ and $\VertArr{\dcat{A}}$.
We expand on the intuition and motivation for this setup in \cref{sec:introduction:approach}.
For now, we summarize the definitions:
\begin{itemize}
\item a wide subcategory $\class{M} \hookrightarrow \cat{E}$ is a \emph{$\kappa$-backdrop} if
  $\class{M}$ has colimits of $(1 + \alpha)$-chains for $\alpha < \kappa$ preserved by $\class{M} \hookrightarrow \cat{E}$, $\kappa$-chains in $\class{M}$ have colimits in $\cat{E}$, and
  $\class{M}$ is closed under cobase change (\cref{backdrop});
\item $A \in \cat{E}$ is \emph{$(\kappa,\class{M})$-small} if $\Hom{\cat{E}}{A}{-}$ preserves colimits of $\kappa$-chains in $\class{M}$ (\cref{small-object});
\item $u$ is \emph{compatible with $\class{M}$} if $\Den{u} \co \Arr{\cat{E}} \to \Arr{\cat{E}}$ is valued in $\class{M}$ and its pushout application $\leibpushapp{\Den{u}} \co \Arr{(\Arr{\cat{E}})} \to \Arr{\cat{E}}$ sends squares levelwise in $\class{M}$ to morphisms in $\class{M}$ (\cref{compatible});
\item $U \co \dcat{A} \to \Sq{\cat{E}}$ is \emph{$(\kappa,\class{M})$-cellular} if the wide subcategory $\DCodClass{\dcat{A}}{\class{M}} \hookrightarrow \VertArr{\dcat{A}}$ of squares with upper horizontal map an isomorphism and lower horizontal map in $\class{M}$ is closed under cobase change in $\VertArr{\dcat{A}}$, and $(1 + \alpha)$-chains in $\DCodClass{\dcat{A}}{\class{M}}$ have colimits in $\VertArr{\dcat{A}}$ for $\alpha \preceq \kappa$ (\cref{cellular}).
\end{itemize}
If $\cat{E}$ is cocomplete, one can always take $\class{M} = \cat{E}$, but choosing a smaller backdrop gives a stronger result; often we take $\class{M}$ to be a class of monomorphisms.

Our second main theorem, derived from the first, shows that if $U \co \dcat{A} \to \Sq{\cat{E}}$ also functorially lifts codomain retracts (\cref{lifts-codomain-retracts}) then we can assign an $\dcat{A}$-structure to every left map:

\begin{manualtheorem}{{\ref{soa-copointed-coalgebras-functor}}}
  \input{soa-copointed-coalgebras-functor}
\end{manualtheorem}

If the retract lifting operator interacts suitably with composition of left maps, we can even upgrade $j$ to a pseudo double functor from the double category of left maps to $\dcat{A}$ (\cref{soa-copointed-coalgebras-double-functor}), although we prove this only in the case that $\dcat{A}$ is a \emph{thin} pseudo double category.

Implicit in the above is the existence of the cofibrantly generated \textsc{awfs} under the given hypotheses, which we prove first:

\begin{manualtheorem}{{\ref{soa-backdrop}}}
  \input{soa-backdrop}
\end{manualtheorem}

To prove \cref{soa-backdrop}, we use \citeauthor{bourke-garner:16}'s proof \cite[Theorem 6]{bourke-garner:16} that to cofibrantly generate an \textsc{awfs} $(\comonad{L},\monad{R})$ on a diagram $u$, it suffices to show that a certain pointed endofunctor on $\Arr{\cat{E}}$ (\cref{one-step-endofunctor}) admits a free and algebraically free monad, which then becomes the monad $\monad{R}$ of the \textsc{awfs}.
We instantiate this result with a free monad construction due to \textcite[\S5 and \S14]{kelly:80}.
In \cref{sec:free-monad-sequence} we analyze the functoriality of this construction, formulating it as a functor from a category $\ConfP{\kappa}$ of configurations (\cref{pointed-configurations}), whose objects $(\cat{E},\class{M},\ptdendo{T}) \in \ConfP{\kappa}$ are suitable pointed endofunctors $\ptdendo{T} = (T,\tau)$ on a category $\cat{E}$ equipped with a $\kappa$-backdrop $\class{M}$, to a category $\Mnd_s$ whose objects $(\cat{E},\monad{M}) \in \Mnd_s$ are monads $\monad{M}$ on a category $\cat{E}$:

\begin{manualtheorem}{{\ref{pointed-free-monad}}}
  \input{pointed-free-monad}
\end{manualtheorem}

We do not assume $\cat{E}$ is cocomplete, showing that the construction only requires the colimits provided by the $\kappa$-backdrop $\class{M}$ under the assumptions that $\tau$ is valued in $\class{M}$ and that its pushout application $\leibpushapp{\tau}$ preserves $\class{M}$.
The functoriality in \cref{pointed-free-monad} tells us that any strong morphism between pointed endofunctors preserving these colimits induces a strong morphism between their free monads.
This observation comes from previously circulated but unpublished work of the second author \cite{sattler-free-monad}.
This functoriality is key to our proof of \cref{soa-unit-vertical-point}.

We describe two applications of our saturation theorems in \cref{sec:applications}.
To motivate these, we first recall a concrete family of \textsc{awfs}'s generated by diagrams: \textsc{awfs}'s for \emph{uniform box-filling fibrations}, which are used in the semantics of homotopy type theory \cite{gambino-sattler:17}.
We show here how the density comonad for a given diagram of generators can be calculated in practice.
Fixing a notion of configuration for a uniform fibration \textsc{awfs} (\cref{uniform-fibration-configuration}), we use \cref{soa-backdrop} to prove:

\begin{manualtheorem}{{\ref{uniform-fibrations}}}
  \input{uniform-fibrations}
\end{manualtheorem}

The first condition \ref{uniform-fibrations:monos} is always satisfied in classical set theory with choice.
In the field of semantics of homotopy type theory, there is an interest in obtaining such \textsc{awfs}'s constructively; condition \ref{uniform-fibrations:finitary} (\cref{finitary-uniform-fibration-configuration}) offers a route for a more limited class of configurations that can be applied in constructive metatheories, even those without general quotients.
This is made possible by our refinement of the algebraic small object argument to non-cocomplete settings.
For the interested reader, we comment on the constructivity of our results in \cref{sec:metatheory}.

Our first application is a relativization of \cref{soa-copointed-coalgebras-functor,soa-copointed-coalgebras-double-functor}:

\begin{manualtheorem}{{\ref{preimage-coalgebras-functor}}}
  \input{preimage-coalgebras-functor}
\end{manualtheorem}

This corresponds to a typical use of the classical saturation principle, namely showing that a functor $F \co \cat{E} \to \cat{F}$ sends left maps in $\cat{E}$ into some class of maps in $\cat{F}$.
Often $U \co \dcat{A} \to \Sq{\cat{F}}$ is the double category of left maps for a second \textsc{awfs} on $\cat{F}$; such a double category always satisfies the left-connectedness and cellularity hypotheses.

For our second application, we look at \emph{extension operations}, operations that take some structure on the domain of a left map and extend it to the codomain.
Concretely, we are interested in showing that certain Quillen premodel categories \cite{barton:19} satisfy the \emph{fibration extension property}, the property that any fibration $X \fib A$ can be extended along any trivial cofibration $A \tcof B$ to produce a fibration over $B$ fitting in a pullback square
\[
  \begin{tikzcd}
    X \pullback \ar[fib]{d} \ar[dashed]{r} & Y \ar[dashed,fib]{d} \\
    A \ar{r}[below]{m} & B \rlap{.}
  \end{tikzcd}
\]
This property can be used to show that a premodel category is in fact a Quillen model category (see, \eg, \cite[\S3.3]{cavallo-sattler:23}).
We consider more generally a category $\cat{E}$ equipped with a Grothendieck fibration $P \co \cat{F} \to \cat{E}$; for the motivating example, the fiber of $P$ at $A \in \cat{E}$ is the category of fibrations into $A$.
Given $P$, we define a notion of composable structure $U_P \co \Ext{P} \to \Arr{\cat{E}}$ (\cref{extension-double-category}) for which a structure on a morphism $m \co A \to B$ is a section of the reindexing functor $\subst{m} \co \cat{F}_B \to \cat{F}_A$, then use \cref{soa-copointed-coalgebras-functor} to show:

\begin{manualtheorem}{{\ref{lift-extension-operation}}}
  \input{lift-extension-operation}
\end{manualtheorem}

\subsection{Intuition}
\label{sec:introduction:approach}

The reduction of the algebraic small object argument to a free monad construction is useful for abstract reasoning, but it can obscure the intuition behind our saturation theorems \ref{soa-unit-vertical-point}, \ref{soa-copointed-coalgebras-functor}, and \ref{soa-copointed-coalgebras-double-functor}.
To motivate them, we first review Quillen's small object argument, then unfold Garner's and compare.

\subsubsection{Quillen's small object argument}

Given a set of generating left maps $S \subseteq \Arr{\cat{E}}$, the small object argument constructs a factorization of $f \co X \to Y$ by iteratively attaching new ``cells'' to $X$, eventually arriving at a map $f \co X' \to Y$ that lifts against $S$ together with a comparison left map $m \co X \to X'$ over $Y$. In the first stage of the iteration, one takes each lifting problem $\alpha \co u \to f$ against a map $u \co A \to B$ in $S$ and glues a copy of the codomain $B$ onto the existing copy of the domain $A$ in $X$, defining a new object $X_1$:

\begin{equation}
  \label{quillen-soa-one-step}
  \begin{tikzcd}[column sep=small]
    \coprod_{\substack{(u \co A \to B) \in S \\ \alpha \co u \to f}} A \ar{d} \ar{r} & \coprod_{\substack{(u \co A \to B) \in S \\ \alpha \co u \to f}} B \ar[dashed]{d} \ar[bend left]{dd} \\
    X \ar[bend right=15]{dr}[below left,pos=.2]{f} \ar[dashed]{r}{m_1} & \pushout X_1 \ar[dashed]{d}[left]{f_1} \\[1em]
    & [-3em] Y
  \end{tikzcd}
\end{equation}

Any class defined by left lifting is closed under small coproducts and cobase change, so $m_1$ is a left map.
While $f_1$ has a solution for every lifting problem $\alpha \co u \to f_1$ that factors through $f \to f_1$, it need not itself be a right map.
Thus one iterates, producing a transfinite sequence $f \to f_1 \to f_2 \to \cdots$.
The sequence does not typically converge, but the approximations do become right maps at some ordinal index $\kappa$ given compactness assumptions on the domains of the maps in $S$ (the titular ``small objects'').
The composite
\[
  \begin{tikzcd}
    X \ar{r}{m_1} & X_1 \ar{r}{m_2} & X_2 \ar{r} & \cdots \ar{r} & X_\kappa
  \end{tikzcd}
\]
is a left map, since any class defined by left lifting is closed under transfinite composition, and so $X \to X_\kappa \to Y$ is the desired factorization.
By construction, the left factors of the factorizations are cell complexes built from maps in $S$.

\subsubsection{Garner's algebraic small object argument}
\label{sec:introduction:algebraic-soa}

Garner's argument follows the same blueprint as Quillen's, but modifies both the one-step approximation construction and the process of iteration.

The one-step factorization \eqref{quillen-soa-one-step} of Quillen's argument has a natural generalization to the diagram case.
Rather than attaching a coproduct of all lifting problems against $f$, one attaches the colimit described by the density comonad for $u$ \eqref{density-comonad-colimit}:
\[
  \begin{tikzcd}[column sep=huge]
    \placeholder \ar{d} \ar{r}{\Den{u}(f)} & \placeholder \ar[dashed]{d} \ar[bend left]{dd} \\
    X \ar[bend right=15]{dr}[below left,pos=.2]{f} \ar[dashed]{r}{m_1} & \pushout X_1 \ar[dashed]{d}[left]{f_1} \\[1em]
    & [-3em] Y \rlap{.}
  \end{tikzcd}
\]
Here the outer square $\Den{u}(f) \to f$ is the counit of the comonad.
This is the same construction as Quillen's when $\cat{J}$ is discrete; in the general case, it ensures that the attached solutions satisfy the required coherences \eqref{generating-diagram-coherence}.

Naively iterating this one-step factorization would produce multiple solutions to each problem, and solutions added at different stages would not satisfy the required coherences with respect to each other.
One instead iterates \emph{while quotienting to identify solutions that are added multiple times}.
Besides being necessary to handle non-discrete generating diagrams, this change also conceptual benefits: it allows the sequence of approximations to actually converge, and the full construction can be described as applying the algebraically free monad on the pointed endofunctor on $\Arr{\cat{E}}$ that sends $f$ to its one-step factorization $f_1$.
Indeed, the monad $\monad{R}$ of the \textsc{awfs} $(\comonad{L},\monad{R})$ generated by Garner's argument is precisely this algebraically free monad.

We want to extract a notion of cell complex from this construction.
Something has to change: the left factor $Lf$ of a factorization is still a transfinite composite of step maps, as in Quillen's argument, but now the step maps need not be left maps.
The following contrived example illustrates this:

\begin{example}
  \label{backdrop-not-left-maps}
  Consider the diagram $u \co \cat{J} \to \Arr{\Set}$ where $\cat{J} = \braces{\mathsf{b} \to \mathsf{a} \leftarrow \mathsf{b'}}$ is the walking cospan and $u$ is the mapping
  \begin{equation}
    \label{set-bad-transition-diagram}
    \begin{tikzcd}
      \mathsf{b} \ar[mapsto,shorten=.7em]{d} \ar{r} & \mathsf{a} \ar[mapsto,shorten=.7em]{d} & \ar{l} \mathsf{b'} \ar[mapsto,shorten=.7em]{d} \\[.5em]
      0 \ar{d} \ar{r} & 1 \ar{d}{\inl} & \ar{l} 0 \ar{d} \\
      1 \ar{r}[below]{\inr} & 1 \sqcup 1 & \ar{l}{\inr} 1 \rlap{}
    \end{tikzcd}
  \end{equation}
  where $\inl,\inr \co 1 \to 1 \sqcup 1$ are the two coprojections.
  This diagram generates the same \textsc{wfs} as the single arrow $0 \to 1$, namely the (complemented mono, split epi) \textsc{wfs} on $\Set$.

  Now we examine Garner's small object argument applied to factorize $f \co 0 \to 1$.
  In the first step, we insert solutions to one lifting problem against $u(\mathsf{b})$ and one against $u(\mathsf{b'})$; there are no lifting problems against $u(\mathsf{a})$ since the domain of $f$ is empty.
  We thus have
  \[
    \begin{tikzcd}
      0 \ar{d}[left]{f} \ar[dashed]{r}{m_1} & 1 \sqcup 1 \ar[dashed]{d}[left]{f_1} \\
      1 \ar[equals]{r} & 1 \rlap{.}
    \end{tikzcd}
  \]
  at the first stage. There are four lifting problems for $u$ against $f_1$:
  \[
    \begin{tikzcd}[column sep=small]
      0 \ar{d}[left]{u(\mathsf{b})} \ar{r}{!} & 1 \sqcup 1 \ar{d}{f_1} \\
      1 \ar{r}[below]{!} & 1
    \end{tikzcd}
    \qquad
    \begin{tikzcd}[column sep=small]
      0 \ar{d}[left]{u(\mathsf{b'})} \ar{r}{!} & 1 \sqcup 1 \ar{d}{f_1} \\
      1 \ar{r}[below]{!} & 1
    \end{tikzcd}
    \qquad
    \begin{tikzcd}[column sep=small]
      1 \ar{d}[left]{u(\mathsf{a})} \ar{r}{\inl} & 1 \sqcup 1 \ar{d}{f_1} \\
      1 \sqcup 1 \ar{r}[below]{!} & 1
    \end{tikzcd}
    \qquad
    \begin{tikzcd}[column sep=small]
      1 \ar{d}[left]{u(\mathsf{a})} \ar{r}{\inr} & 1 \sqcup 1 \ar{d}{f_1} \\
      1 \sqcup 1 \ar{r}[below]{!} & 1
    \end{tikzcd}
  \]
  For each of these problems, we add a point to the two-step factorization.
  Solutions to the first two problems were already added in the first step, so the new solutions are equated with the existing ones in $f_1$.
  Moreover, the morphisms of \eqref{set-bad-transition-diagram} identify the point added for the third problem  with those added for the first and second problems; likewise for the fourth problem.
  The two-step factorization is hence a \emph{quotient} of the one-step factorization in which the points added in the first step are equated:
  \[
    \begin{tikzcd}
      0 \ar{d}[left]{f} \ar{r}{m_1} & 1 \sqcup 1 \ar{d}[left]{f_1} \ar[dashed]{r}{m_2} & 1 \ar[dashed]{d}[left]{f_2} \\
      1 \ar{r}[below]{!} & 1 \ar{r}[below]{!} & 1 \rlap{.}
    \end{tikzcd}
  \]
  Garner's argument converges at this stage.
  In contrast to Quillen's argument, the transition map between the $n$-step and $(n+1)$-step factorization need not belong to the left class: $m_2 \co 1 \sqcup 1 \to 1$ is not a monomorphism.
\end{example}

We thus put aside the step maps $X_\alpha \to X_{\alpha + 1}$ and instead look at the composites $m_{\le \alpha} \co X \to X_{\alpha}$, which \emph{do} belong to the left class.
We can write the left factor of Garner's factorization as their colimit in $\Arr{\cat{E}}$:
\begin{equation}
  \label{left-factor-transfinite-composite}
  \begin{tikzcd}
    X \ar[equals]{d} \ar[equals]{r} & X \ar{d}{m_{\le 1}} \ar[equals]{r} & X \ar{d}{m_{\le 2}} \ar[equals]{r} & \cdots \ar[equals]{r} & X \ar{d}{\colim_{\alpha < \kappa} m_{\le \alpha} \eqdef Lf} \\
    X \ar{r}[below]{m_1} & X_1 \ar{r}[below]{m_2} & X_2 \ar{r}[below]{m_3} & \cdots \ar{r} & X_{\kappa}
  \end{tikzcd}
\end{equation}
While the \emph{class} of left maps of an \textsc{awfs} $(\comonad{L},\monad{R})$ is not generally closed under such colimits, the \emph{category} of $\comonad{L}$-coalgebras is (indeed, it has all small colimits), and  in fact \eqref{left-factor-transfinite-composite} lifts to a diagram in $\Coalg{\comonad{L}}$: thanks to the quotienting, we solve lifts against $m_{\le \alpha}$ and $m_{\le \beta}$ in the same way on their overlap.
It turns out that the cobase change involved in defining $m_{\le \alpha+1}$ from $m_\alpha$ can similarly be described as a pushout in $\Coalg{\comonad{L}}$.
These are the two kinds of colimit we require of a cellular notion of composable structure.

Finally, while the step maps in Garner's argument are not always left maps of the generated \textsc{awfs}, we can often say \emph{something} about them.
In our target examples of \cref{sec:uniform-fibrations}, for instance, they are always monomorphisms.
In such cases, we can more precisely delimit the colimits we need in Garner's construction.
This is the role of the backdrop $\class{M}$.
Ultimately, the argument requires colimits chains of $\comonad{L}$-coalgebras of the form
\[
  \begin{tikzcd}[column sep=large]
    A \ar[verto]{d} \ar[equal]{r} & A \ar[verto]{d} \ar[equal]{r} & A \ar[verto]{d} \ar[equal]{r} & \cdots \\
    B_0 \ar{r}[below]{\in \class{M}} & B_1 \ar{r}[below]{\in \class{M}} & B_2 \ar{r}[below]{\in \class{M}} & \cdots
  \end{tikzcd}
\]
(up to some ordinal bound) and pushouts of spans of the form
\[
  \begin{tikzcd}[column sep=large]
    A_0 \ar[verto]{d} & \ar{l} A \ar[verto]{d} \ar[equal]{r} & A \ar[verto]{d} \\
    B_0 & \ar{l} B \ar{r}[below]{\in \class{M}} & B_1 \rlap{.}
  \end{tikzcd}
\]
For these colimits to suffice for some generating diagram $u \co \cat{J} \to \Arr{\cat{E}}$, we require (\cref{compatible}) that $\Den{u} \co \Arr{\cat{E}} \to \Arr{\cat{E}}$ is valued in $\class{M}$ and that given a square $(h,k) \co f \to g$ in $\cat{E}$ with $h,k \in \class{M}$, the pushout gap map
\[
  \begin{tikzcd}[column sep=large]
    \placeholder \ar{d}[left]{\dom \Den{u}(h,k)} \ar{r}{\Den{u}f} & \placeholder \ar{d} \ar[bend left]{ddr}{\cod \Den{u}(h,k)} \\
    \placeholder \ar[bend right=20]{drr}[below left,pos=0.2]{\Den{u}g} \ar{r} & \pushout \object \ar[dashed]{dr} \\[-1em]
    & &[-2em] \placeholder
  \end{tikzcd}
\]
is in $\class{M}$.
The first condition is always satisfied for $\class{M}$ the class of left maps, but the second can fail.

%% file: soa-unit-vertical-point.tex
Fix a category $\cat{E}$ equipped with a $\kappa$-backdrop $\class{M}$, a diagram $u \co \cat{J} \to \Arr{\cat{E}}$ compatible with $\class{M}$ such that $\dom u$ is levelwise $(\kappa,\class{M})$-small, and a left-connected, $(\kappa,\class{M})$-cellular notion of composable structure $U \co \dcat{A} \to \Sq{\cat{E}}$.
If $(\comonad{L},\monad{R})$ is the \textsc{awfs} cofibrantly generated by $u$, then any lift $\bm{D}_{\dcat{A}} \co \Arr{\cat{E}} \to \VertArr{\dcat{A}}$ of $\Den{u}$ through $\VertArr{U}$ induces lifts $\bm{L}_{\dcat{A}}$ of $L$ and $\Gg \co \bm{D}_{\dcat{A}} \to \bm{L}_{\dcat{A}}$ of $\oneStep{u} \co \Den{u} \to L$ through $\VertArr{U}$, as in the diagram
\[
  \begin{tikzcd}[sep=huge]
    \Arr{\cat{E}} \ar[equals]{d} \ar[bend left=20]{r}[name=vlan]{\bm{D}_{\dcat{A}}} \ar[dashed,bend right=20]{r}[below,name=vcomonad]{\bm{L}_{\dcat{A}}} & \VertArr{\dcat{A}} \ar{d}{\VertArr{U}} \arrow[Rightarrow,dashed,from=vlan,to=vcomonad,shorten=1mm,"\Gg"] \\
    \Arr{\cat{E}} \ar[bend left=20]{r}[name=lan]{\Den{u}} \ar[bend right=20]{r}[below,name=comonad]{L} & \Arr{\cat{E}}
    \arrow[Rightarrow,from=lan,to=comonad,shorten=1mm] \rlap{.}
  \end{tikzcd}
\]


%% file: soa-copointed-coalgebras-functor.tex
Fix a category $\cat{E}$ equipped with a $\kappa$-backdrop $\class{M}$, a diagram $u \co \cat{J} \to \Arr{\cat{E}}$ compatible with $\class{M}$ such that $\dom u$ is levelwise $(\kappa,\class{M})$-small, and a left-connected, $(\kappa,\class{M})$-cellular notion of composable structure $U \co \dcat{A} \to \Sq{\cat{E}}$ with a codomain retract lifting operator.
If $(\comonad{L},\monad{R})$ is the \textsc{awfs} cofibrantly generated by $u$, then any lift $\bm{D}_{\dcat{A}} \co \Arr{\cat{E}} \to \VertArr{\dcat{A}}$ of $\Den{u} \co \Arr{\cat{E}} \to \Arr{\cat{E}}$ through $\VertArr{U}$ induces a functor $j \co \Coalg{\UnderPtd{\comonad{L}}} \to \VertArr{\dcat{A}}$ fitting in the diagram
\[
  \begin{tikzcd}[sep=small]
    \Coalg{\UnderPtd{\comonad{L}}} \ar[dashed]{rr}{j} \ar{dr}[below left]{U_{\UnderPtd{\comonad{L}}}} && \VertArr{\dcat{A}} \ar{dl}{\VertArr{U}} \rlap{.} \\
    & \Arr{\cat{E}}
  \end{tikzcd}
\]


%% file: soa-backdrop.tex
Let $\cat{E}$ be a category with a $\kappa$-backdrop $\class{M}$ and let $u \co \cat{J} \to \Arr{\cat{E}}$ be a diagram compatible with $\class{M}$ such that $\dom u$ is levelwise $(\kappa,\class{M})$-small.
Then there is an \textsc{awfs} $(\comonad{L},\monad{R})$ cofibrantly generated by $u$.


%% file: pointed-free-monad.tex
The projection $\ConfP{\kappa} \to \Cat$ lifts to a functor $\ConfP{\kappa} \to \Mnd_s$ sending $(\cat{E}, \class{M}, \ptdendo{T})$ to the free and algebraically free monad on $\ptdendo{T}$.


%% file: uniform-fibrations.tex
Let $(t,I)$ be a uniform fibration configuration in a presheaf category $\cat{E} = \PSh{\cat{C}}$.
If either
\begin{enumerate*}[itemjoin=\ ]
\item\label{uniform-fibrations:monos}the monomorphisms form a $\kappa$-backdrop in $\cat{E}$ and $t$ is locally $(\kappa,\text{mono})$-small, or
\item\label{uniform-fibrations:finitary}$(t,I)$ is finitary,
\end{enumerate*}
then the uniform fibration \textsc{awfs} on $(t,I)$ exists.


%% file: preimage-coalgebras-functor.tex
Fix a $\kappa$-backdrop-preserving functor $F \co (\cat{E},\class{M}) \to (\cat{F},\class{N})$ and diagram $u \co \cat{J} \to \Arr{\cat{E}}$ compatible with $\class{M}$ such that $\dom u$ is levelwise $(\kappa,\class{M})$-small.
Write $(\comonad{L},\monad{R})$ for the \textsc{awfs} cofibrantly generated by $u$.
For a left-connected, $(\kappa,\class{N})$-cellular notion of composable structure $U \co \dcat{A} \to \Sq{\cat{F}}$ with a codomain retract lifting operator, any lift
\[
  \begin{tikzcd}[row sep=large]
    \Arr{\cat{E}} \ar[dashed]{r} \ar{d}[left]{\Den{u}} & \VertArr{\dcat{A}} \ar{d}{\VertArr{U}} \\
    \Arr{\cat{E}} \ar{r}[below]{\Arr{F}} & \Arr{\cat{F}}
  \end{tikzcd}
\qquad
\text{induces a functor}
\qquad
  \begin{tikzcd}[row sep=large]
    \Coalg{\UnderPtd{\comonad{L}}} \ar{d}[left]{U_{\UnderPtd{\comonad{L}}}} \ar[dashed]{r} & \VertArr{\dcat{A}} \ar{d}{\VertArr{U}} \\
    \Arr{\cat{E}} \ar{r}[below]{\Arr{F}} & \Arr{\cat{F}} \rlap{.}
  \end{tikzcd}
\]
If $\dcat{A}$ is thin and its codomain retract lifting operator is compositional, then said functor extends to a pseudo double functor $\DCoalg{\UnderPtd{\comonad{L}}} \to \dcat{A}$ over $\Sq{\cat{E}} \to \Sq{\cat{F}}$.


%% file: lift-extension-operation.tex
Let $\cat{E}$ be a category with a $\kappa$-backdrop $\class{M}$ and a diagram $u \co \cat{J} \to \Arr{\cat{E}}$ compatible with $\class{M}$ such that $\dom u$ is levelwise $(\kappa,\class{M})$-small.
Write $(\comonad{L},\monad{R})$ for the \textsc{awfs} cofibrantly generated by $u$.
Let $P \co \cat{F} \to \cat{E}$ be a Grothendieck fibration such that
\begin{enumerate}
\item colimits of $(1+\alpha)$-chains in $\class{M}$ are Van Kampen for $P$ for $\alpha \preceq \kappa$;
\item cobase changes of maps in $\class{M}$ are Van Kampen for $P$.
\end{enumerate}
Any lift of $\Den{u} \co \Arr{\cat{E}} \to \Arr{\cat{E}}$ through $U_P \co \Ext{P} \to \Arr{\cat{E}}$ induces a functor $\Coalg{\UnderPtd{\comonad{L}}} \to \Ext{P}$ over $\Arr{\cat{E}}$ assigning an extension operation to every left map of the \textsc{awfs}.


%% file: free-monad-sequence.tex
\section{Fine-grained functoriality of free algebras and monads}
\label{sec:free-monad-sequence}

The goal of this section is to review free algebra and free monad constructions described by \textcite{kelly:80} without the blanket assumption that all small colimits exist.
Instead, we make the needed colimits explicit, parametrizing them by a class of maps $\class{M}$ satisfying certain colimit closure properties.
We will construct free and algebraically free monads on pointed endofunctors:

\begin{notation}
  Given a (co)monad $\monad{M} = (M,\eta,\mu)$, write $\UnderPtd{\monad{M}} = (M,\eta)$ for its underlying (co)pointed endofunctor.
\end{notation}

\begin{definition}
  \label{algebraically-free-monad}
  Let $\ptdendo{T}$ be a pointed endofunctor on a category $\cat{E}$.
  A monad $\monad{M}$ on $\cat{E}$ equipped with a morphism of pointed endofunctors $\Gg \co \ptdendo{T} \to \UnderPtd{\monad{M}}$ on $\cat{E}$ defines the \emph{free monad on $\ptdendo{T}$} when every morphism $\Gg' \co \ptdendo{T} \to \UnderPtd{\monad{M}'}$ into some monad $\monad{M}'$ on $\cat{E}$ factors as $\Gg$ followed by a unique morphism of monads on $\cat{E}$, and the \emph{algebraically free monad on $\ptdendo{T}$} when the functor $\Alg{\monad{M}} \to \Alg{\ptdendo{T}}$ induced by $\gamma$ is an isomorphism of categories.
\end{definition}

Besides allowing us to apply the constructions in non-cocomplete settings, an essential consequence is a refined functoriality principle: a translation functor between two settings for the construction (which we call ``configurations'') need only preserve the class $\class{M}$ and associated colimits for free algebras to be preserved and free monads to be related.

Following Kelly's approach, we first treat well-pointed endofunctors and then reduce the case of an arbitrary pointed endofunctor to the well-pointed case.
We parameterize the categories of configurations by a limit ordinal $\kappa$ and impose a convergence condition in terms of the class $\class{M}$.
For simplicity of presentation, we prefer to fix $\kappa$ uniformly for all objects, only noting here that a more flexible treatment would be possible.
We refer to \cref{kelly-convergence} below for a comparison with Kelly's convergence criteria.

\subsection{Strong categories of functors, adjunctions, monads}

We express functoriality of the free algebra and free monad constructions by exhibiting functors from a category of configurations to strong (\ie, non-lax) variants of the categories $\Adju$ and $\Mnd$ of adjunctions and monads, respectively, which we define below.

\begin{definition} \label{cat-of-functors}
The \emph{category $\Fun_s$ of functors with strong morphisms} is defined as follows:
\begin{enumerate}[label=(\roman*),ref=\thedefinition(\roman*)]
\item
An object consists of categories $\cat{C}$ and $\cat{D}$ with a functor $F \co \cat{C} \to \cat{D}$.
\item
  A morphism $(U,V,\gamma) \co (\cat{C}_1, \cat{D}_1, F_1) \to (\cat{C}_2, \cat{D}_2, F_2)$ consists of functors $U \co \cat{C}_1 \to \cat{C}_2$ and $V \co \cat{D}_1 \to \cat{D}_2$ with an isomorphism $\gamma \co V F_1 \cong F_2 U$:
  \[
    \begin{tikzcd}
      \cat{C}_1 \ar{d}[left]{F_1} \ar{r}{U} & \cat{C}_2 \ar{d}{F_2} \\
      \cat{D}_1 \ar[phantom]{ur}[sloped]{\cong} \ar{r}[below]{V} & \cat{D}_2 \rlap{.}
    \end{tikzcd}
  \]
\end{enumerate}
\end{definition}

\begin{definition} \label{cat-of-adjunctions}
The \emph{category $\Adju_s$ of adjunctions with strong morphisms} is defined as follows:
\begin{enumerate}[label=(\roman*),ref=\thedefinition(\roman*)]
\item 
An object consists of categories $\cat{C}$ and $\cat{D}$ with functors $F \co \cat{C} \to \cat{D}$ and $G \co \cat{D} \to \cat{C}$ and natural transformations $\eta \co \Id \to GF$ and $\epsilon \co FG \to \Id$ such that $\epsilon F \circ F \eta = \id$ and $G \epsilon \circ \eta G = \id$.
\item \label{cat-of-adjunctions:morphism}
  A morphism $(U,V,\alpha,\beta) \co (\cat{C}_1, \cat{D}_1, F_1, G_1, \eta_1, \epsilon_1)\to (\cat{C}_2, \cat{D}_2, F_2, G_2, \eta_2, \epsilon_2)$ consists of functors $U \co \cat{C}_1 \to \cat{C}_2$ and $V \co \cat{D}_1 \to \cat{D}_2$ with isomorphisms $\alpha \co F_2 U \cong V F_1$ and $\beta \co U G_1 \cong G_2 V$ satisfying $\beta F_1 \circ U \eta = G_2 \alpha \circ \eta' U$ and $V \epsilon \circ \alpha G_1 = \epsilon' V \circ F_2 \beta$:
  \[
    \begin{tikzcd}[sep=small]
      & U \ar{dl}[above left]{U\eta_1} \ar{dr}{\eta_2U} \\
      UG_1 F_1 \ar[phantom]{rr}{=} \ar{dr}[below left]{\beta F_1} && G_2F_2U \ar{dl}{G_2\alpha} \\
      & G_2VF_1
    \end{tikzcd}
    \qquad
    \begin{tikzcd}[sep=small]
      & F_2UG_1 \ar{dl}[above left]{\alpha G_1} \ar{dr}{F_2 \beta} \\
      VF_1G_1 \ar[phantom]{rr}{=}  \ar{dr}[below left]{V \epsilon_1} && F_2G_2V \ar{dl}{\epsilon_2V} \\
      & V \rlap{.}
    \end{tikzcd}
  \]
\end{enumerate}
\end{definition}

Note that the two equations in \ref{cat-of-adjunctions:morphism} are interderivable.
There are functors from $\Adju_s$ to $\Fun_s$ selecting the left and right adjoint, respectively.
Observe that while there are lax versions of $\Fun_s$ and $\Adju_s$ for which these functors are fully faithful, this is not the case here.
Requiring natural isomorphisms instead of just natural transformations is how we encode preservation of free algebras in \cref{wellpointed-free-algebras,pointed-free-algebras}.

\begin{definition} \label{cat-of-ptd-endos}
The \emph{category $\PtdEndo_s$ of pointed endofunctors with strong morphisms} is defined as follows:
\begin{enumerate}[label=(\roman*)]
\item
An object $(\cat{E},\ptdendo{T})$ is a pointed endofunctor $\ptdendo{T}$ on a category $\cat{E}$.
\item
A morphism $(F,\gamma) \co (\cat{E}_1, (T_1, \tau_1)) \to (\cat{E}_2, (T_2, \tau_2))$ is a functor $F \co \cat{E}_1 \to \cat{E}_2$ and isomorphism $\gamma \co F T_1 \cong T_2 F$ such that $\gamma \circ F \tau_1 = \tau_2 F$.
\end{enumerate}
\end{definition}

\begin{definition} \label{cat-of-monads}
The \emph{category $\Mnd_s$ of monads with strong morphisms} is defined as follows:
\begin{enumerate}[label=(\roman*)]
\item
An object $(\cat{E},\monad{M})$ is a monad $\monad{M}$ on a category $\cat{E}$.
\item
A morphism $(F,\gamma) \co (\cat{E}_1, (M_1, \eta_1, \mu_1))$ to $(\cat{E}_2, (M_2, \eta_2, \mu_2))$ is a strong morphism of pointed endofunctors $(F, \gamma) \co (\cat{E}_1, (M_1, \eta_1)) \to (\cat{E}_2, (M_2, \eta_2))$ such that $\gamma \circ F \mu_1 = \mu_2 F \circ M_2 \gamma \circ \gamma M_1$.
\end{enumerate}
\end{definition}

\begin{terminology}
  A \emph{strict} morphism of pointed endofunctors is a strong morphism for which the isomorphism $\gamma \co FT \cong T'F$ is an identity; likewise for strict morphisms of monads.
\end{terminology}

There is a functor from $\Mnd_s$ to $\Fun_s$ forgetting the monad structure and the endofunctor aspect, and there is a functor from $\Adju_s$ to $\Mnd_s$ sending an adjunction to its associated monad.
Like $\Adju_s$, $\Mnd_s$ is a wide subcategory of its more common lax version.

\begin{remark}
  The above categories and the below categories of configurations are in fact strict 2-categories and the functors relating them are strict 2-functors, but we shall not need this here.
\end{remark}

\subsection{Well-pointed endofunctors}
\label{sec:well-pointed}

\begin{definition}
  A pointed endofunctor $\ptdendo{S} = (S,\sigma)$ is \emph{well-pointed} when $S\sigma = \sigma S$.
\end{definition}

We use transfinite iteration to build the free monad on a well-pointed endofunctor.
We aim to do so constructively (see \cref{sec:metatheory}), so this requires some care.
For concreteness, we use \citeauthor{powell:75}'s set-theoretic definition of ordinals \cite{powell:75}---an ordinal is a transitive set whose elements are transitive sets---but the constructive reader should be able to substitute the definition appropriate for their favorite metatheory.

\begin{notation}
  For ordinals $\alpha, \beta$, we write $\alpha < \beta$ to mean $\alpha \in \beta$.
  We write $\alpha \preceq \beta$ to mean that either $\alpha < \beta$ or $\alpha = \beta$.
  An \emph{$\alpha$-chain} is a diagram indexed by the poset $(\alpha,\preceq)$.

  We say $\alpha$ is a \emph{limit} ordinal when for every family $(\beta_i < \alpha)_{i \in I}$ indexed by an finite inhabited set $I$ we have some $\beta < \alpha$ with $\beta_i < \beta$ for all $i$ (\cf \textcite[Definition 3.2]{pitts-steenkamp:22}).
  Classically, this agrees with the usual definition.
\end{notation}

For a well-pointed endofunctor $\ptdendo{S}$, the forgetful functor $U_{\ptdendo{S}} \co \Alg{\ptdendo{S}} \to \cat{E}$ is a fully faithful inclusion: $\Alg{\ptdendo{S}}$ consists of those $A \in \cat{E}$ for which $\sigma_A$ is invertible, in which case the algebra map is $\sigma_A^{-1} \co SA \to A$ \cite[Proposition 1.5 and Corollary 1.6]{wolff:78}.
We want to construct the left adjoint to $U_{\ptdendo{S}}$ as the colimit in the category of pointed endofunctors on $\cat{E}$ of the $\kappa$-chain
\begin{equation} \label{wellpointed-free-algebras:0}
  \begin{tikzcd}
    ``\Id_{\cat{E}} \ar{r}{\sigma} & S \ar{r}{\sigma S} & S^2 \ar{r}{\sigma S^2} & \cdots" \rlap{.}
  \end{tikzcd}
\end{equation}
To define the diagram \eqref{wellpointed-free-algebras:0} in a constructively acceptable way, we cannot distinguish between successor and limit ordinal cases.
We adapt a construction for non-pointed endofunctors due to \textcite{pitts-steenkamp:22}.
Benno van den Berg explained a similar argument for pointed endofunctors to us in personal communication, and we follow him in using \citeauthor{koubek-reiterman:79}'s language of \emph{algebraized} (or \emph{algebraic}) chains \cite{koubek-reiterman:79}.
\Citeauthor{koubek-reiterman:79} use algebraized chains to obtain free algebras for non-pointed endofunctors, and \textcite[\S A]{bourke:19} describes an analogue for pointed endofunctors; both distinguish between limit and successor cases, so are not wholly constructive.

\begin{definition}
  \label{algebraized-chain}
  Given a pointed endofunctor $\ptdendo{T} = (T,\tau)$ on a category $\cat{E}$ and an ordinal $\kappa$, we define the \emph{category $\AlgChain{\ptdendo{T}}{\kappa}$ of $\ptdendo{T}$-algebraized $\kappa$-chains} in $\cat{E}$ as follows:
  \begin{enumerate}[label=(\roman*)]
  \item An object $(X,x)$ is a $\kappa$-chain $X \co (\kappa,\preceq) \to \cat{E}$ together with morphisms $x_{\beta < \alpha} : TX_\beta \to X_\alpha$ for each $\beta < \alpha < \kappa$, such that
    \begin{enumerate}[label=(\alph*), ref=\thedefinition(\alph*)]
    \item \label{algebraized-chain:unit} for all $\beta < \alpha < \kappa$ we have $x_{\beta < \alpha}\tau_\beta = X_{\beta \preceq \alpha}$;
    \item \label{algebraized-chain:natural}
      for all $\alpha,\alpha',\beta,\beta' < \kappa$ standing in the relations
      \[
        \begin{tikzcd}
          \beta \ar[phantom]{d}[sloped]{<} \ar[phantom]{r}{\preceq} & \beta' \ar[phantom]{d}[sloped]{<} \\
          \alpha \ar[phantom]{r}{\preceq} & \alpha' \rlap{,}
        \end{tikzcd}
      \]
      the square
      \[
        \begin{tikzcd}[column sep=large]
          TX_\beta \ar{d}[left]{x_{\beta < \alpha}} \ar{r}{TX_{\beta \preceq \beta'}} & TX_{\beta'} \ar{d}{x_{\beta' < \alpha'}} \\
          X_\alpha \ar{r}[below]{X_{\alpha \preceq \alpha'}} & X_{\alpha'}
        \end{tikzcd}
      \]
      commutes.
    \end{enumerate}
  \item A morphism $\varphi \co (X,x) \to (Y,y)$  is a natural transformation $\varphi \co X \to Y$ with the property that $y_{\beta < \alpha} \circ T\varphi_\beta = \varphi_\alpha \circ x_{\beta < \alpha}$ for $\beta < \alpha < \kappa$.
    \end{enumerate}
\end{definition}

\begin{definition}
  A $\ptdendo{T}$-algebraized $\kappa$-chain $(X,x)$ is \emph{colimiting} when for every $\beta < \kappa$ the cocone $(x_{\gamma<\beta} \co TX_\gamma \to X_\beta)_{\gamma < \beta}$ under the restricted diagram $X \upharpoonright \beta \co (\beta, \preceq) \to \cat{E}$ is colimiting.
\end{definition}

We now define the category of configurations for the free monad construction.

\begin{definition}
  Let $\class{M}$ be a wide subcategory of a category $\cat{E}$.
  For a given diagram shape $\cat{J}$, we say that $\class{M}$ \emph{has $\cat{J}$-indexed colimits in $\cat{E}$} when $\class{M}$ (as a category) has these colimits and the inclusion into $\cat{E}$ preserves them.
\end{definition}

\begin{definition} \label{wellpointed-configuration}
Let $\kappa > 0$ be a limit ordinal.
The (large) category $\ConfWP{\kappa}$ of \emph{configurations for the free monad sequence on a well-pointed endofunctor} is defined as follows:
\begin{enumerate}[label=(\roman*)]
\item
An object is a tuple $(\cat{E}, \class{M}, \ptdendo{S})$ of a category $\cat{E}$, a wide subcategory $\class{M} \hookrightarrow \cat{E}$, and a well-pointed endofunctor $\ptdendo{S} = (S,\sigma)$ on $\cat{E}$ such that:
\begin{enumerate}[label=(\alph*),ref=\thedefinition(\alph*)]
\item \label{wellpointed-configurations:unit}
  $\sigma$ is valued in $\class{M}$;
\item \label{wellpointed-configurations:composition}
  $\class{M}$ has colimits of $(1 + \alpha)$-chains in $\cat{E}$ for $\alpha < \kappa$;
\item \label{wellpointed-configurations:convergence}
  if $(X,x)$ is an $\ptdendo{S}$-algebraized $\kappa$-chain such that $X$ and $x$ factor through $\class{M}$, then $X$ admits a colimit in $\cat{E}$ and this colimit is preserved by $S$.
\end{enumerate}
\item
  A morphism $(\cat{E}_1, \class{M}_1, \ptdendo{S}_1) \to (\cat{E}_2, \class{M}_2, \ptdendo{S}_2)$ is a map $(F,\gamma) \co (\cat{E}_1,\ptdendo{S}_1) \to (\cat{E}_2,\ptdendo{S}_2)$ in $\PtdEndo_s$ such that $F$ sends $\class{M}_1$ into $\class{M}_2$, preserves colimits of $(1 + \alpha)$-chains in $\class{M}_1$ for $\alpha < \kappa$, and preserves colimits of $\ptdendo{S}$-algebraized $\kappa$-chains in $\class{M}_1$.
\end{enumerate}
\end{definition}

\begin{remark}
  \label{wellpointed-complemented-example}
  Conditions \ref{wellpointed-configurations:composition} asks that each $(1 + \alpha)$-chain in $\class{M}$ admits a cocone in $\class{M}$ that is colimiting both in $\class{M}$ and in $\cat{E}$.
  In contrast, \ref{wellpointed-configurations:convergence} requires only that each (algebraic) $\kappa$-chain has a colimit in $\cat{E}$.

  This distinction accommodates cases such as where $\cat{E}$ is a category with coproducts and $\class{M}$ is the class of complemented monomorphisms.
  These examples are important when working in constructive metatheories, as discussed in \cref{sec:metatheory}.
  Given an $\omega$-chain $A_0 \mono A_1 \mono \cdots$ of complemented monomorphisms, its colimit in can be computed without quotienting as the coproduct $\coprod_{n < \omega} A_{n} \setminus A_{n-1}$ of complements (setting $A_{-1} \defeq 0$).
  However, this need not be a colimit in the wide subcategory of complemented monomorphisms: given a family of compatible inclusions $A_n \mono B$ for $n < \omega$, we do get an inclusion $\colim_n A_n \mono B$, but the latter may not be complemented.

  We take care to only ask for colimits of \emph{algebraic} chains in \ref{wellpointed-configurations:convergence} because we will use this assumption in the reduction of the pointed to the well-pointed case; see \cref{pointed-free-algebras:end}.
\end{remark}

We now show that the forgetful functor $U_{\ptdendo{S}} \co  \Alg{\ptdendo{S}} \to \cat{E}$ associated to a configuration $(\cat{E}, \class{M}, \ptdendo{S})$ admits a left adjoint, functorially in the configuration.
There are two aspects to this theorem.
The first is to ensure that free algebras exist for any fixed configuration; we construct these explicitly following \citeauthor{kelly:80}.
This already ensures that the functor $\ConfWP{\kappa} \to \Fun_s$ lifts to through the category $\Adju_{l/s}$ defined as $\Adju_s$ in \cref{cat-of-adjunctions} but without requiring the $\alpha \co F_2 U \to VF_1$ in the definition of morphisms to be invertible, since the projection $\Adju_{l/s} \to \Fun_s$ of the right adjoint is fully faithful.

The second is to show that this $\ConfWP{\kappa} \to \Adju_s$ lifts through the inclusion $\Adju_s \to \Adju_{l/s}$, \ie, to show that the $\alpha$-components of the action of $\ConfWP{\kappa} \to \Adju$ on morphisms are invertible.
This corresponds to showing that the colimits invoked in constructing free algebras are preserved by the mediating functor; the morphisms in $\ConfWP{\kappa}$ are defined so as to ensure this.

\begin{lemma}
  \label{preserved-colimit-of-algebraized-chain-is-algebra}
  Let $\ptdendo{T}$ be a pointed endofunctor on a category $\cat{E}$, let $\kappa$ be a limit ordinal, and let $(X,x)$ be a $\ptdendo{T}$-algebraized $\kappa$-chain.
  If $X$ admits a colimit $X_\kappa \in \cat{E}$ and $\ptdendo{T}$ preserves this colimit, then $X_\kappa$ admits a $\ptdendo{T}$-algebra structure.
\end{lemma}
\begin{proof}
  Write $(\inj{\alpha} \co X_\alpha \to X_\kappa)_{\alpha < \kappa}$ for the colimit cocone under $X$.
  For each $\alpha < \kappa$, we have a choice of $\alpha < \alpha' < \kappa$ by assumption that $\kappa$ is a limit ordinal; write $t_\alpha \co TX_\alpha \to X_\kappa$ for the composite
  \[
    \begin{tikzcd}[column sep=huge]
      TX_\alpha \ar{r}{x_{\alpha < \alpha'}} & X_{\alpha'} \ar{r}{\inj{\alpha'}} & X_\kappa \rlap{.}
    \end{tikzcd}
  \]
  Furthermore, for any $\alpha \preceq \beta < \kappa$ we have a choice of $\alpha',\beta' < \gamma < \kappa$ and thus a commutative diagram
  \[
    \begin{tikzcd}[column sep=large,row sep=large]
      TX_\alpha \ar{dd}[left]{TX_{\alpha \preceq \beta}} \ar[bend right=10]{dr}[below left,pos=.6]{x_{\alpha<\gamma}} \ar{r}{x_{\alpha < \alpha'}} & X_{\alpha'} \ar{d}[pos=.3,left]{X_{\alpha'\preceq\gamma}} \ar{dr}{\inj{\alpha'}} & \\
      & X_\gamma \ar{r}[description]{\inj{\gamma}} & X_\kappa \\
      TX_\beta \ar[bend left=10]{ur}[pos=.6]{x_{\beta<\gamma}} \ar{r}[below]{x_{\beta < \beta'}} & X_{\beta'} \ar{u}[pos=.3,left]{X_{\beta'\preceq\gamma}} \ar{ur}[below right]{\inj{\beta'}}
    \end{tikzcd}
  \]
  showing that $t_{\beta} \circ TX_{\alpha \preceq \beta} = t_\alpha$.
  Thus $(t_\alpha)_{\alpha < \kappa}$ is a cocone under $TX$.
  By the assumption that $TX_\kappa$ is the colimit of $TX$, then, we have a unique induced map $t \co TX_\kappa \to X_\kappa$.
  For each $\alpha < \kappa$ we have
  \begin{align*}
    t \circ \tau_{X_\kappa} \circ \inj{\alpha}
    &= t \circ T\inj{\alpha} \circ \tau_{X_\alpha} \\
    &= \inj{\alpha'} \circ x_{\alpha < \alpha'} \circ \tau_{X_\alpha} \\
    &= \inj{\alpha'} \circ X_{\alpha \preceq \alpha'} \\
    &= \inj{\alpha} \rlap{,}
  \end{align*}
  whence $t \circ \tau_{X_\kappa} = \id$.
  Thus $t$ exhibits $X_\kappa$ as a $\ptdendo{T}$-algebra.
\end{proof}

\begin{proposition}
  \label{algebraized-chains-initial}
  Every colimiting $\ptdendo{T}$-algebraized $\kappa$-chain is an initial object of $\AlgChain{\ptdendo{T}}{\kappa}$.
\end{proposition}
\begin{proof}
  By transfinite induction on $\kappa$.
  Let $(X,x),(Y,y) \in \AlgChain{\ptdendo{T}}{\kappa}$ where $(X,x)$ is colimiting.
  By induction hypothesis, we have for $\alpha < \kappa$ a unique morphism $\varphi^\alpha \co (X,x) \upharpoonright \alpha \to (Y,y) \upharpoonright \alpha$ between the restrictions to $\AlgChain{\ptdendo{T}}{\alpha}$.
  Note that by uniqueness, the restriction of $\varphi^\alpha$ to any $\beta < \alpha$ is identical with $\varphi^\beta$.
  Because $(X,x)$ is colimiting, we have for each $\alpha < \kappa$ a unique $\varphi_\alpha \co X_\alpha \to Y_\alpha$ fitting into the diagram
  \begin{equation}
    \label{algebraized-chains-initial-definition}
    \begin{tikzcd}
      SX_\beta \ar{d}[left]{S\varphi^\alpha_\beta} \ar{r}{x_{\beta < \alpha}} & X_\alpha \ar[dashed]{d}{\varphi_\alpha} \\
      SY_\beta \ar{r}[below]{y_{\beta < \alpha}} & Y_\alpha
    \end{tikzcd}
  \end{equation}
  for all $\beta < \alpha$.
  From \cref{algebraized-chains-initial-definition}, we have that $\varphi_\beta = \varphi^\alpha_\beta$ for $\beta < \alpha$: for any $\gamma < \beta < \alpha$, we calculate that
  \[
    \varphi_\beta \circ x_{\gamma < \beta} \;=\; y_{\gamma < \beta} \circ S\varphi^\beta_\gamma \;=\; y_{\gamma < \beta} \circ S\varphi^\alpha_\gamma \;=\; \varphi^\alpha_\beta \circ x_{\gamma < \beta} \rlap{.}
  \]
  Thus \eqref{algebraized-chains-initial-definition} shows that $\varphi_\alpha$ is a morphism $\varphi \co (X,x) \to (Y,y)$.
  Uniqueness of $\varphi$ follows from uniqueness of the $\varphi^\alpha$: if $\varphi \co (X,x) \to (Y,y)$ is a morphism, then we know that $\varphi \upharpoonright \alpha = \varphi^\alpha$ for $\alpha < \kappa$ by induction hypothesis, so for any $\alpha < \kappa$ and $\beta < \alpha$ we have $\varphi_\alpha \circ x_{\beta < \alpha} = y_{\beta < \alpha} \circ S\varphi_\beta = y_{\beta < \alpha} \circ S\varphi^\alpha_\beta$.
\end{proof}

\begin{lemma}
  \label{generate-colimiting-algebraized-chain-local}
  Let $\cat{E}$ be a category, $\class{M} \hookrightarrow \cat{E}$ be a wide subcategory, and $\ptdendo{S} = (S,\sigma)$ be a well-pointed endofunctor on $\cat{E}$ such that $\sigma$ is valued in $\class{M}$ and $\class{M}$ has colimits of $\alpha$-chains in $\cat{E}$ for $\alpha < \kappa$.
  There is a colimiting $\ptdendo{S}$-algebraized $\kappa$-chain $(X,x)$ such that $X$ and $x$ factor through $\class{M}$.
\end{lemma}
\begin{proof}
  We go by transfinite induction on $\kappa$; suppose that we have such an algebraized $\alpha$-chain $(X^\alpha,x^\alpha)$ for all $\alpha < \kappa$.
  Define $X \co \kappa \to \cat{E}$ on objects by $X_{\alpha} \defeq \colim_{\alpha} {SX^\alpha}$ and write $\inj{\beta} \co SX^\alpha_\beta \to X_\alpha$ for the coprojections.
  For $\alpha \preceq \alpha'$, we have a unique isomorphism $X^{\alpha \preceq \alpha'} \co (X^\alpha,x^\alpha) \cong (X^{\alpha'},x^{\alpha'}) \upharpoonright \alpha$ by \cref{algebraized-chains-initial}, which induces a transition map $X_{\alpha \preceq \alpha'} \defeq \colim_{\alpha} {SX^{\alpha \preceq \alpha'}} \co X_{\alpha} \to X_{\alpha'}$ in $\class{M}$.
  For $\beta < \alpha < \kappa$, we have a unique $\theta_\beta^\alpha \co X_\beta \to X_\beta^\alpha$ in $\class{M}$ fitting in the diagram
  \[
    \begin{tikzcd}[sep=large]
      SX^\beta_\gamma \ar{d}[sloped]{\cong}[left]{SX^{\beta \preceq \alpha}_\gamma} \ar{r}{\inj{\gamma}} & X_\beta \ar[dashed]{d}{\Gth_\beta^\alpha} \\
      SX_\gamma^\alpha \ar{r}[below]{x^\alpha_{\gamma<\beta}} & X_\beta^\alpha
    \end{tikzcd}
  \]
  for $\gamma < \beta$.
  For $\beta < \alpha$, define $x_{\beta < \alpha} \co SX_\beta \to X_\alpha$ to be the composite
  \begin{equation}
    \label{generate-colimiting-algebraized-chain-local:bump}
    \begin{tikzcd}
      SX_\beta \ar{r}{S\theta_\beta^\alpha} &[2em] SX^\alpha_\beta \ar{r}{\inj{\beta}} & X_\alpha \rlap{.}
    \end{tikzcd}
  \end{equation}

  We check that $(X,x)$ is a $\ptdendo{S}$-algebraized $\kappa$-chain.
  For \ref{algebraized-chain:unit}, that $x_{\beta < \alpha} \circ \sigma_{X_\beta} = X_{\beta \preceq \alpha}$, first observe that for $\beta < \alpha$ the diagram
  \begin{equation}
    \label{generate-colimiting-algebraized-chain-local:trapezoid}
    \begin{tikzcd}
      &[-1em] X_{\beta}^\alpha \ar{r}{\sigma_{X_\beta^\alpha}} &[1em] SX_\beta^\alpha \ar{dr}{\inj{\alpha}} &[-1em] { } \\
      X_\beta \ar{ur}{\theta_\beta^\alpha} \ar{rrr}[below]{X_{\beta \preceq \alpha}} &&& X_\alpha
    \end{tikzcd}
  \end{equation}
  commutes by probing with $\inj{\gamma} \co SX_\gamma^\beta \to X_\beta$ for $\gamma < \beta$: we have
  \begin{align*}
    \inj{\alpha} \circ \sigma_{X_\beta^\alpha} \circ \theta_\beta^\alpha \circ \inj{\gamma}
    &= \inj{\alpha} \circ \sigma_{X_\beta^\alpha} \circ x^\alpha_{\gamma < \beta} \circ SX^{\beta \preceq \alpha}_\gamma \\
    &= \inj{\alpha} \circ Sx^\alpha_{\gamma < \beta} \circ \sigma_{SX_\gamma^\alpha} \circ SX^{\beta \preceq \alpha}_\gamma \\
    &= \inj{\alpha} \circ Sx^\alpha_{\gamma < \beta} \circ S\sigma_{X_\gamma^\alpha} \circ SX^{\beta \preceq \alpha}_\gamma \tag{well-pointedness} \\
    &= \inj{\alpha} \circ SX^\alpha_{\gamma \preceq \beta} \circ SX_\gamma^{\beta \preceq \alpha} \\
    &= \inj{\gamma} \circ SX^{\beta \preceq \alpha}_\gamma \\
    &= X_{\beta \preceq \alpha} \circ \inj{\gamma} \rlap{.}
  \end{align*}
  It follows that $x_{\beta < \alpha} \circ \sigma_{X_\beta} = \inj{\beta} \circ S\theta_\beta^\alpha \circ \sigma_{X_\beta} = \inj{\beta} \circ \sigma_{X_\beta^\alpha} \circ \theta_\beta^\alpha = X_{\beta \preceq \alpha}$.

  For \ref{algebraized-chain:natural}, let $\beta < \alpha < \kappa$ and $\beta' < \alpha' < \kappa$ with $\beta \preceq \beta'$ and $\alpha \preceq \alpha'$.
  We can check that the diagram
  \[
    \begin{tikzcd}[column sep=huge]
      X_\beta \ar{d}[left]{\theta^\alpha_\beta} \ar{rr}{X_{\beta \preceq \beta'}} && X_{\beta'} \ar{d}{\theta^{\alpha'}_{\beta'}} \\
      X^\alpha_\beta \ar{r}[below]{X^{\alpha \preceq \alpha'}_\beta} & X^{\alpha'}_\beta \ar{r}[below]{X^{\alpha'}_{\beta\preceq \beta'}} & X^{\alpha'}_{\beta'}
    \end{tikzcd}
  \]
  commutes by probing with $\inj{\gamma} \co SX_\gamma^\beta \to X_\beta$ for all $\gamma < \beta$.
  From this it follows that
  \[
    \begin{tikzcd}[column sep=huge]
      SX_\beta \ar[bend right=90]{dd}[left]{x_{\beta < \alpha}}\ar{d}[left]{S\theta^\alpha_\beta} \ar{rr}{SX_{\beta \preceq \beta'}} && SX_{\beta'} \ar{d}{S\theta^{\alpha'}_{\beta'}} \ar[bend left=90]{dd}{x_{\beta' < \alpha'}} \\
      SX^\alpha_\beta \ar{r}{SX^{\alpha \preceq \alpha'}_\beta} \ar{d}[left]{\inj{\beta}} & SX^{\alpha'}_\beta \ar{r}{SX^{\alpha'}_{\beta\preceq \beta'}}  \ar{dr}[below left]{\inj{\beta}} & SX^{\alpha'}_{\beta'} \ar{d}{\inj{\beta'}} \\
      X_\alpha \ar{rr}[below]{X_{\alpha \preceq \alpha'}} && X_{\alpha'}
    \end{tikzcd}
  \]
  commutes.

  Finally, to see that $(X,x)$ is colimiting, observe that each $\theta_\beta^\alpha$ is an isomorphism, with inverse $X_\beta^\alpha \to X_\beta$ induced by the cocone $(\inj{\gamma} \circ S(X_\gamma^{\alpha\preceq\alpha'})^{-1} \co SX^\alpha_\gamma \to X_\beta)_{\gamma<\beta}$, and these assemble to a natural isomorphism $\theta^\alpha \co X \upharpoonright \alpha \simeq X^\alpha$.
  It thus follows from the definition \eqref{generate-colimiting-algebraized-chain-local:bump} of $x$ that the cocones $(x_{\gamma < \beta})_{\gamma < \beta}$ are colimiting for all $\beta < \alpha$.
\end{proof}

The combination of \cref{algebraized-chains-initial} and \cref{generate-colimiting-algebraized-chain-local} would immediately give us initial algebras, \ie, free algebras on initial objects.
To construct free algebras on arbitrary objects, we pass to a coslice.

\begin{notation}
  When $\ptdendo{T} = (T,\tau)$ is a (well-)pointed endofunctor on $\cat{E}$ and $A$ is an object of $\cat{E}$, we write $\Coslice{A}{\ptdendo{T}}$ for the induced (well-)pointed endofunctor on the coslice $\Coslice{A}{\cat{E}}$, whose endofunctor $\Coslice{A}{T}$ sends $b \co A \to B$ to $\tau_Bb \co A \to TB$ and whose point at $b \co A \to B$ is $\tau_B \co (B,b) \to (TB,\tau_Bb)$.
\end{notation}

\begin{proposition}
  \label{generate-colimiting-algebraized-chain}
  Let $(\cat{E}, \class{M}, \ptdendo{S}) \in \ConfWP{\kappa}$ be a configuration.
  For every $A \in \cat{E}$, we have a colimiting $\Coslice{A}{\ptdendo{S}}$-algebraized $\kappa$-chain $(X,x)$ in $\Coslice{A}{\cat{E}}$ that factors through $\Coslice{A}{\class{M}}$.\footnote{By $\Coslice{A}{\class{M}}$ we mean the coslice of the category $\class{M}$ under $A$, not the full subcategory of $\Coslice{A}{\cat{E}}$ consisting of morphisms in $\class{M}$.}
\end{proposition}
\begin{proof}
  The assumption that $\class{M}$ has colimits of $(1 + \alpha)$-chains in $\cat{E}$ for $\alpha \preceq \kappa$ implies that $\Coslice{A}{\class{M}}$ has colimits of $\alpha$-chains in $\Coslice{A}{\cat{E}}$ for $\alpha \preceq \kappa$: we calculate the colimit of an $\alpha$-chain $X \co (\alpha,\preceq) \to \Coslice{A}{\class{M}}$ by calculating the colimit of the $(1+\alpha)$-chain in $\cat{E}$ that sends $0$ to $A$ and $1 + \beta$ to the object underlying $X_\beta$.
  Apply \cref{generate-colimiting-algebraized-chain-local} with the well-pointed endofunctor $\Coslice{A}{\ptdendo{S}}$.
\end{proof}

\begin{lemma}
  \label{colimiting-algebraized-chain-to-free-algebra}
  Let $(\cat{E}, \class{M}, \ptdendo{S}) \in \ConfWP{\kappa}$ be a configuration and $(X,x)$ be a colimiting $\Coslice{A}{\ptdendo{S}}$-algebraized $\kappa$-chain in $\Coslice{A}{\cat{E}}$ that factors through $\Coslice{A}{\class{M}}$.
  Then $X$ admits a colimit $(B,b) \in \Coslice{A}{\cat{E}}$, $B$ is an $\ptdendo{S}$-algebra, and $b \co A \to B$ exhibits $B$ as the free $\ptdendo{S}$-algebra on $A$.
\end{lemma}
\begin{proof}
  The colimit of a $\kappa$-chain in $\Coslice{A}{\cat{E}}$ is the colimit of a $(1 + \kappa)$-chain in $\cat{E}$.
  Since $\kappa > 0$, the coprojection $(\kappa,\preceq) \to (1 + \kappa,\preceq)$ is final; thus such a colimit can in fact be computed as the colimit of the $\kappa$-chain of codomains in $\cat{E}$.
  Furthermore, if $(X,x)$ is a $\Coslice{A}{\ptdendo{S}}$-algebraized $\kappa$-chain in $\Coslice{A}{\class{M}}$, then projecting codomains produces a $\ptdendo{S}$-algebraized $\kappa$-chain in $\class{M}$.
  By definition of $\ConfWP{\kappa}$, $\cat{E}$ has and $S$ preserves colimits of such chains.
  We conclude that $X$ has a colimit $(B,b) \in \Coslice{A}{\cat{E}}$.

  By \cref{preserved-colimit-of-algebraized-chain-is-algebra}, $(B,b)$ is an $\Coslice{A}{\ptdendo{S}}$-algebra, which implies that $B$ is an $\ptdendo{S}$-algebra.
  If $t : SC \to C$ is any other $\ptdendo{S}$-algebra with a morphism $c \co A \to C$, then we have an $\Coslice{A}{\ptdendo{S}}$-algebraized $\kappa$-chain $(Y,y)$ where $Y$ is the constant diagram at $(C, c) \in \Coslice{A}{\cat{E}}$ and $y_{\beta<\alpha} \co (\Coslice{A}{\ptdendo{S}})Y \to Y$ is $t$.
  By \cref{algebraized-chains-initial}, there is a unique morphism $\varphi \co (X,x) \to (Y,y)$, which induces a unique extension of $c \co A \to C$ along $b \co A \to B$.
\end{proof}

\begin{theorem} \label{wellpointed-free-algebras}
  The functor $\ConfWP{\kappa} \to \Fun_s$ sending $(\cat{E}, \class{M}, \ptdendo{S})$ to the forgetful functor $\Alg{\ptdendo{S}} \to \cat{E}$ lifts through the restriction $\Adju_s \to \Fun_s$ to the right adjoint.
\end{theorem}
\begin{proof}
  Let $(\cat{E}, \class{M}, \ptdendo{S}) \in \ConfWP{\kappa}$ be a configuration with $\ptdendo{S} = (S, \sigma)$ a well-pointed endofunctor.
  By \cref{generate-colimiting-algebraized-chain,colimiting-algebraized-chain-to-free-algebra}, we have a free $\ptdendo{S}$-algebra on every $A \in \cat{E}$.
  This gives us a left adjoint $S^\kappa$ to $U_{\ptdendo{S}}$, completing the construction of the action on objects of $\ConfWP{\kappa} \to \Adju_s$.
  
For functoriality, given a morphism $(F,\gamma) \co (\cat{E}_1, \class{M}_1, \ptdendo{S}_1) \to (\cat{E}_2, \class{M}_2, \ptdendo{S}_2)$, the canonical comparison map $S_2^\kappa  F \to F S_1^\kappa$ is invertible thanks to the colimit-preservation assumptions on $F$: for $A \in \cat{E}$, it sends the unique colimiting $\Coslice{A}{\ptdendo{S}_1}$-algebraized $\kappa$-chain in $\Coslice{A}{\cat{E}_1}$ provided by \cref{generate-colimiting-algebraized-chain} to the unique colimiting $\Coslice{FA}{\ptdendo{S}_2}$-algebraized $\kappa$-chain in $\Coslice{FA}{\cat{E}_2}$.
\end{proof}

As a consequence of \cref{wellpointed-free-algebras}, free monads exist functorially:

\begin{theorem} \label{wellpointed-free-monad}
  The projection $\ConfWP{\kappa} \to \Cat$ lifts to a functor $\ConfWP{\kappa} \to \Mnd_s$ sending $(\cat{E}, \class{M}, \ptdendo{S})$ to the free and algebraically free monad on $\ptdendo{S}$.
\end{theorem}
\begin{proof}
  This is an immediate consequence of \cref{wellpointed-free-algebras}, as the free and algebraically free monad on a pointed endofunctor is given by the monad of the free algebra adjunction \cite[Proposition 22.2 and Theorem 22.3]{kelly:80}.
\end{proof}

\begin{remark}
  \label{kelly-convergence}
  Kelly describes a number of convergence criteria for the free algebra construction on a well-pointed or pointed endofunctor $\ptdendo{S} = (S,\sigma)$, each requiring that \emph{$S$ preserves $\class{E}$-tightness of $(\class{M}',\cat{K})$-cones} for some orthogonal factorization systems $(\class{E},\class{M})$ and $(\class{E}',\class{M}')$ and diagram category (or categories) $\cat{K}$ \cite[\S2.3]{kelly:80}.
  Here an \emph{$(\class{M}',\cat{K})$-cone} is a cocone $\delta \co d \to \Delta e$ under a $\cat{K}$-indexed diagram $d$ for which each $\delta_i \co d_i \to e$ belongs to $\class{M}'$.
  A cocone $\delta \co d \to \Delta e$ is \emph{$\class{E}$-tight} when the induced map $[\delta] \co \colim d \to e$ is in $\class{E}$.

  Our convergence condition on our $S$ and $\class{M}$ is loosely analogous to asking that $S$ preserves isomorphism-tightness of $(\class{M},\kappa)$-cones.
  However, we do not assume that $\class{M}$ is the right class of an orthogonal factorization system, while conversely Kelly does not assume that $\class{M}$ is closed under transfinite compositions.
  We also use the property in a different way.
  Our conditions ensure that the diagram \eqref{wellpointed-free-algebras:0} factors through our $\class{M}$, allowing us to apply the convergence criterion to that diagram directly.
  By contrast, Kelly's argument (see, \eg, the proof of \cite[Proposition 4.1]{kelly:80}) does not require that \eqref{wellpointed-free-algebras:0} factors through his $\class{M}'$; rather, he derives auxiliary diagrams which do factor through $\class{M}'$ and uses the tightness-preservation property with \emph{these} diagrams to arrive indirectly at the convergence of \eqref{wellpointed-free-algebras:0}.
\end{remark}

\subsection{Pointed endofunctors}

For the construction of free algebras and monads on arbitrary pointed endofunctors, we need not only sequential colimits but also pushouts.
We now introduce the notion of \emph{backdrop} that we will use throughout this article:

\begin{definition}
  Let $\class{M}$ be a wide subcategory of a category $\cat{E}$.
  We say that $\class{M}$ is \emph{closed under cobase change in $\cat{E}$} when for every span consisting of morphisms $A \to B$ in $\class{M}$ and $A \to X$ in $\cat{E}$, there is a pushout square of the form
  \[
    \begin{tikzcd}
      A \ar{d}[left]{\class{M} \ni} \ar{r} & X \ar[dashed]{d}{\in \class{M}} \\
      B \ar[dashed]{r} & \pushout Y
    \end{tikzcd}
  \]
  in $\cat{E}$.
\end{definition}

\begin{definition}
  \label{backdrop}
  Let $\cat{E}$ be a category and $\kappa > 0$ be a limit ordinal.
  A wide subcategory $\class{M}$ of $\cat{E}$ is a \emph{$\kappa$-backdrop} when
  \begin{enumerate}[label=(\alph*),ref=\thedefinition(\alph*)]
  \item \label{backdrop:sequence} $\class{M}$ has colimits of $(1 + \alpha)$-chains in $\cat{E}$ for $\alpha < \kappa$;
  \item \label{backdrop:end} $\kappa$-chains in $\class{M}$ have colimits in $\cat{E}$;
  \item \label{backdrop:pushout} $\class{M}$ is closed under cobase change in $\cat{E}$.
  \end{enumerate}
  Given $F \co \cat{E}_1 \to \cat{E}_2$ and $\kappa$-backdrops $\class{M}_1$ and $\class{M}_2$ in $\cat{E}_1$ and $\cat{E}_2$ respectively, we say that $F$ is a \emph{$\kappa$-backdrop-preserving functor} $(\cat{E}_1,\class{M}_1) \to (\cat{E}_2,\class{M}_2)$ when $F$ sends $\class{M}_1$ into $\class{M}_2$, and preserves colimits of $(1+\alpha)$-chains in $\class{M}_1$ for $\alpha < \kappa$, colimits of $\kappa$-chains in $\class{M}_1$, and cobase changes of maps in $\class{M}_1$.
\end{definition}

\begin{example}
  \label{monomorphisms-backdrop}
  For any adhesive \cite{lack-sobocinski:05} and $\kappa$-exhaustive \cite[\S3]{shulman:15} category $\cat{E}$ (\eg, any topos), the class of monomorphisms in $\cat{E}$ is a $\kappa$-backdrop.
\end{example}

\begin{example}
  \label{complemented-backdrop}
  If $\cat{E}$ is a category with coproducts, then the class of complemented monomorphisms in $\cat{E}$ is an $\omega$-backdrop. As noted in \cref{wellpointed-complemented-example}, $\cat{E}$ has colimits of $\omega$-chains of complemented monomorphisms.
  Cobase changes of along complemented monomorphisms can likewise be computed using coproducts and complements:
  \[
    \begin{tikzcd}
      A \ar{d}[left]{f} \ar[tail]{r}{m} & C \ar{d} \\
      B \ar[tail]{r} & \pushout B \sqcup (C\setminus A) \rlap{.}
    \end{tikzcd}
  \]
  Given another category $\cat{C}$, the functor category $\Func{\cat{C}}{\cat{E}}$ has the class of levelwise complemented monomorphisms as an $\omega$-backdrop, since the relevant colimits are computed levelwise.
  In a constructive metatheory without quotients (see \cref{sec:metatheory}), we are interested in particular in the case where $\cat{E} = \Set$.
\end{example}

We recall \emph{Leibniz application} of a natural transformation, an instance of the Leibniz construction exposited by \textcite[\S7]{joyal-tierney:07} and \textcite[\S4]{riehl-verity:14}.

\begin{definition}
  \label{leibniz-application}
  Let $\alpha \co F \to G$ be a natural transformation between functors $F,G \co \cat{C} \to \cat{D}$.
  The \emph{(Leibniz) pushout application} $\leibpushapp{\alpha} \co \Arr{\cat{C}} \to \Arr{\cat{D}}$ and \emph{(Leibniz) pullback application} $\leibpullapp{\alpha} \co \Arr{\cat{C}} \to \Arr{\cat{D}}$, when they exist, are the functors sending $f \co A \to B$ in $\cat{C}$ to the pushout and pullback gap maps
  \[
    \begin{tikzcd}
      FA \ar{d}[left]{Ff} \ar{r}{\alpha_A} & GA \ar{d} \ar[bend left=20]{ddr}{Gf} \\
      FB \ar[bend right=20]{drr}[below]{\alpha_B} \ar{r} & \pushout \object \ar[dashed]{dr}[pos=0.2]{\leibpushapp{\alpha}f} \\[-2em]
      & & GB
    \end{tikzcd}
    \qquad \text{and} \qquad
    \begin{tikzcd}
      FA \ar[bend right=20]{ddr}[below left]{Ff} \ar[dashed]{dr}[below left,pos=0.7]{\leibpullapp{\alpha}f} \ar[bend left=20]{drr}{\alpha_A} && \\[-2em]
      & \object \pullback \ar{d} \ar{r} & GA \ar{d}{Gf} \\
      & FB \ar{r}[below]{\alpha_B} & GB
    \end{tikzcd}
  \]
  respectively.
\end{definition}

\begin{definition} \label{pointed-configurations}
  Let $\kappa > 0$ be a limit ordinal.
  The (large) category $\ConfP{\kappa}$ of \emph{configurations for the free monad sequence on a pointed endofunctor} is defined as follows.
  \begin{enumerate}[label=(\roman*)]
  \item
    An object is a tuple $(\cat{E}, \class{M}, \ptdendo{T})$ of a category $\cat{E}$, a wide subcategory $\class{M} \hookrightarrow \cat{E}$, and a pointed endofunctor $\ptdendo{T} = (T, \tau)$ on $\cat{E}$ such that:
    \begin{enumerate}[label=(\alph*),ref=\thedefinition(\alph*)]
    \item \label{pointed-configurations:backdrop}
      $\class{M}$ is a $\kappa$-backdrop in $\cat{E}$,
    \item \label{pointed-configurations:unit}
      $\tau$ is valued in $\class{M}$,
    \item \label{pointed-configurations:leibniz-application}
      $\class{M}$ is closed under pushout application of $\tau$,
    \item \label{pointed-configurations:convergence}
      $T$ preserves colimits of $\kappa$-chains in $\class{M}$.
    \end{enumerate}
  \item
    A morphism from $(\cat{E}_1, \class{M}_1, \ptdendo{T}_1)$ to $(\cat{E}_2, \class{M}_2, \ptdendo{T}_2)$ is a morphism $(F,\gamma) \co (\cat{E}_1,\ptdendo{T}_1) \to (\cat{E}_2,\ptdendo{T}_2)$ in $\PtdEndo_s$ such that $F$ defines a $\kappa$-backdrop-preserving functor $(\cat{E}_1,\class{M}_1) \to (\cat{E}_2,\class{M}_2)$.
  \end{enumerate}
\end{definition}

\begin{lemma}
  \label{pointed-configurations:functor}
  Let $(\cat{E}, \class{M}, (T,\tau)) \in \ConfP{\kappa}$ be a configuration.
  Then $T$ preserves $\class{M}$.
\end{lemma}
\begin{proof}
  For any $f \co A \to B$ in $\class{M}$, we have the diagram
  \[
    \begin{tikzcd}
      A \ar{d}[left]{f} \ar{r}{\tau_A} & TA \ar{d} \ar[bend left=20]{ddr}{Tf} \\
      B \ar[bend right=20]{drr}[below]{\tau_B} \ar{r} & \pushout \object \ar{dr}[pos=0.2]{\leibpushapp{\tau}f} \\[-2em]
      & & TB \rlap{.}
    \end{tikzcd}
  \]
  The morphism from $TA$ to the pushout object is a cobase change of $f$, so belongs to $\class{M}$ by \ref{pointed-configurations:backdrop}.
  The pushout gap map belongs to $\class{M}$ by \ref{pointed-configurations:leibniz-application}.
  Thus their composite $Tf$ belongs to $\class{M}$.
\end{proof}

\begin{remark}
  \label{pointed-configuration-simplification}
  Let $\cat{E}$ be a category with a wide subcategory $\class{M}$ and let $(T,\tau)$ be a pointed endofunctor  satisfying \ref{pointed-configurations:backdrop} and \ref{pointed-configurations:unit}.
  If $T$ preserves $\class{M}$ and $\tau$ is cartesian, then \ref{pointed-configurations:leibniz-application} holds whenever $\class{M}$ is closed under ``binary unions'' in the sense that for any pullback square of the form
  \[
    \begin{tikzcd}
      A \ar{d}[left]{\class{M} \ni} \pullback \ar{r}{\in \class{M}} & C \ar{d}{\in \class{M}} \\
      B \ar{r}[below]{\in \class{M}} & D
    \end{tikzcd}
  \]
  the pushout gap map is in $\class{M}$.
  In particular, this condition is satisfied when $\class{M}$ is the class of monomorphisms in an adhesive category \cite[Theorem 5.1]{lack-sobocinski:05}.
\end{remark}

We employ the following construction that transfers a pointed endofunctor along an adjunction.

\begin{definition}
  \label{transfer}
  Let $\adjunction{\cat{E}}{F}{G}{\cat{F}}$ be an adjoint pair of functors and let $(T,\Gt)$ be a pointed endofunctor on $\cat{E}$.
  When the pushout
  \begin{equation}
    \label{pointed-endofunctor-transfer-pushout}
    \begin{tikzcd}[column sep=large]
      FG \ar{d}[left]{\epsilon} \ar{r}{F\Gt G} & FTG \ar[dashed]{d} \\
      \Id_{\cat{F}} \ar[dashed]{r}[below]{\Gp} & \pushout P
    \end{tikzcd}
  \end{equation}
  in $\Func{\cat{F}}{\cat{F}}$ exists and is computed pointwise, we say that the pointed endofunctor $(P,\Gp)$ on $\cat{F}$ is the \emph{transfer of $(T,\Gt)$ along $F \dashv G$}.
\end{definition}

\begin{proposition}
  \label{transfer-well-pointed}
  Let $\adjunction{\cat{E}}{F}{G}{\cat{F}}$ be an adjoint pair of functors and let $\ptdendo{S}$ be a well-pointed endofunctor on $\cat{E}$.
  When it exists, the transfer of $\ptdendo{S}$ along $F \dashv G$ is also well-pointed.
\end{proposition}
\begin{proof}
  See \textcite[Proposition 9.2]{kelly:80}.
\end{proof}

\begin{proposition}
  \label{transfer-algebras}
  Let $\adjunction{\cat{E}}{F}{G}{\cat{F}}$ be an adjoint pair of functors and let $\ptdendo{T} = (T,\tau)$ be a pointed endofunctor on $\cat{E}$.
  When the transfer $\ptdendo{P} = (P,\pi)$ of $\ptdendo{T}$ along $F \dashv G$ exists, the square
  \begin{equation}
    \label{transfer-algebras-pullback}
    \begin{tikzcd}[row sep=large, column sep=small]
      \Alg{\ptdendo{P}} \ar{d}[left]{U_{\ptdendo{P}}} \pullback \ar{r} & \Alg{\ptdendo{T}} \ar{d}{U_{\ptdendo{T}}} \\
      \cat{F} \ar{r}[below]{G} & \cat{E}
    \end{tikzcd}
  \end{equation}
  is a pullback (and thus also a 2-pullback \cite[Theorem 1]{joyal-street:93}).
  Here the functor $\Alg{\ptdendo{P}} \to \Alg{\ptdendo{T}}$ sends an algebra $PA \to A$ to the transpose of the composite $FTGA \to PA \to A$.
\end{proposition}
\begin{proof}
  An explicit proof can be found in \textcite[Lemma 10]{seip:24}.
  The case where $\ptdendo{T}$ is well-pointed is in \textcite[Proposition 9.2]{kelly:80}, and the general case is used by \textcite[Proposition 16]{bourke-garner:16} with reference to the proof for idempotent monads in \textcite[\S2]{wolff:78}.
\end{proof}

In other words, the category of objects $A \in \cat{F}$ paired with a $\ptdendo{T}$-algebra structure on $GA$ is itself the category of algebras for a pointed endofunctor.

\begin{definition}
  \label{transfer-chain}
  Let $\adjunction{\cat{E}}{F}{G}{\cat{F}}$ be an adjoint pair of functors and let $\ptdendo{T} = (T,\tau)$ be a pointed endofunctor on $\cat{E}$ admitting a transfer $\ptdendo{P} = (P,\pi)$ along $F \dashv G$.
  Given a $\ptdendo{P}$-algebraized $\alpha$-chain $(X,x)$, write $G(X,x)$ for the $\ptdendo{T}$-algebraized $\alpha$-chain consisting of the $\alpha$-chain $GX \colon (\alpha,\preceq) \to \cat{E}$ and maps
  \[
    \begin{tikzcd}
      TGX_\beta \ar{r}{\gamma_{X_\beta}} & GPX_\beta \ar{r}{Gx_{\beta<\alpha}} &[2em] GX_\beta
    \end{tikzcd}
  \]
  for $\beta < \alpha$, where $\gamma \colon TG \to GP$ is the transpose of the transformation $FTG \to P$ in the pushout \eqref{pointed-endofunctor-transfer-pushout} defining $\ptdendo{P}$.
\end{definition}

\begin{proposition}
  \label{transfer-functoriality}
  Let a strong morphism of adjunctions
  \[
    (U,V,\alpha,\beta) \co (\cat{C}_1, \cat{D}_1, F_1, G_1, \eta_1, \epsilon_1) \to (\cat{C}_2, \cat{D}_2, F_2, G_2, \eta_2, \epsilon_2)
  \]
  in $\Adju_s$ be given together with an extension of $U$ to a strong morphism of pointed endofunctors $(U,\gamma) \co (\cat{C}_1,\ptdendo{T}_1) \to (\cat{C}_2,\ptdendo{T}_2)$ in $\PtdEndo_s$.
  If the transfers $\ptdendo{P}_1$ of $\ptdendo{T}_1$ along $F_1 \dashv G_1$ and $\ptdendo{P}_2$ of $\ptdendo{T}_2$ along $F_2 \dashv G_2$ respectively exist and $V$ preserves the pushouts defining $\ptdendo{P}_1$, then $V$ extends to a morphism $(\cat{D}_1,\ptdendo{P}_1) \to (\cat{D}_2,\ptdendo{P}_2)$ in $\PtdEndo_s$.
\end{proposition}
\begin{proof}
  We have an isomorphism of spans
  \[
    \begin{tikzcd}[row sep=large, column sep=huge]
      V \ar[equal]{dddd} & VF_1G_1 \ar{l}[above]{V\epsilon_1} \ar{r}{VF_1\tau_1 G_1} & VF_1T_1G_1 \\
      & & F_2UT_1G_1 \ar{u}[sloped]{\cong}[right]{\alpha T_1G_1} \ar{dd}[sloped,below]{\cong}{F_2\gamma G_1} \\[-2em]
      & F_2UG_1 \ar{uu}{\alpha G_1}[sloped,below]{\cong}  \ar{ur}[near end]{F_2U\tau_1G_1} \ar{dr}[below left,near end]{F_2\tau_2UG_1} \ar{dd}[sloped]{\cong}[left]{F_2\beta} \\[-2em]
      & & F_2T_2UG_1 \ar{d}[sloped,below]{\cong}{F_2T_2\beta} \\
      V & F_2G_2V \ar{l}{\epsilon_2V} \ar{r}[below]{F_2\tau_2G_2V} & F_2T_2G_2V
    \end{tikzcd}
  \]
  and thus an induced isomorphism of pushouts
  \[
    \begin{tikzcd}
      V \ar[equal]{d} \ar{r}{V\pi_1} & VP_1 \ar[dashed]{d} \\
      V \ar{r}[below]{\pi_2V} & P_2V
    \end{tikzcd}
  \]
  as required.
\end{proof}

Following Kelly, we reduce the construction of a free monad on a pointed endofunctor $\ptdendo{T}$ on $\cat{E}$ to the construction of a free monad on a derived well-pointed endofunctor $\ptdendo{T}^\qoppa$ on a category $\cat{E}^\qoppa$.

\begin{notation}
  Let $\cat{E}$ be a category and $\class{M}$ be a wide subcategory of $\cat{E}$ closed under cobase change.
  Given a pointed endofunctor $\ptdendo{T} = (T, \tau)$ on $\cat{E}$ whose unit is valued in $\class{M}$, write
  \[
    \begin{tikzcd}
      \Arr{\cat{E}} \ar[yshift=0.6em]{r}{\tau_*} \ar[phantom]{r}[font=\small]{\bot} & \Comma{T}{\cat{E}} \ar[yshift=-0.6em]{l}{\tau^!}
    \end{tikzcd}
  \]
  for the adjoint pair where $\tau^! \co T \downarrow \cat{E} \to \cat{E}^\to$ sends $(A, B, f \co TA \to B)$ to $f\tau_A \co A \to B$ and $\tau_*$ sends $f \co A \to B$ to the triple $(A, C, g \co TA \to C)$ defined by the pushout
  \begin{equation} \label{pointed-free-algebras:0.5}
    \begin{tikzcd}
      A \ar{d}[left]{f} \ar{r}{\tau_A} & TA \ar[dashed]{d}{g} \\
      B \ar[dashed]{r} & \pushout C \rlap{,}
    \end{tikzcd}
  \end{equation}
  which exists by closure of $\class{M}$ under cobase change in $\cat{E}$.
\end{notation}

\begin{notation}
  Given a category $\cat{E}$ and wide subcategory $\class{M} \hookrightarrow \cat{E}$, write $\ArrFull{\cat{E}}{\class{M}}$ for the full subcategory of $\Arr{\cat{E}}$ consisting of arrows in $\class{M}$.
\end{notation}

\begin{definition}
  \label{pointed-free-algebras:category}
  Let $\cat{E}$ be a category, $\class{M}$ be a wide subcategory of $\cat{E}$ closed under cobase change, and $\ptdendo{T}$ be a pointed endofunctor on $\cat{E}$.
  Define $\cat{E}^\qoppa$ via the following pullback, which is also a weak 2-pullback since $\class{M}$ is replete \cite[Theorem 1]{joyal-street:93}:
  \begin{equation} \label{pointed-free-algebras:0}
    \begin{tikzcd}
      \cat{E}^\qoppa \pullback \ar[hook,dashed]{d} \ar[dashed]{r} & \ArrFull{\cat{E}}{\class{M}} \ar[hook]{d} \\
      \Comma{T}{\cat{E}} \ar{r}[below]{\tau^!} & \Arr{\cat{E}} \rlap{.}
    \end{tikzcd}
  \end{equation}
  Explicitly, $\cat{E}^\qoppa$ is the full subcategory of $T \downarrow \cat{E}$ of objects $(A, B, f)$ for which $f\tau_A \co A \to B$ is in $\class{M}$.
  The left adjoint $\tau_*$ of $\tau_!$ restricts to a functor $\ArrFull{\cat{E}}{\class{M}} \to \cat{E}^\qoppa$, so we have a restricted adjoint pair
  \[
    \begin{tikzcd}
      \ArrFull{\cat{E}}{\class{M}} \ar[yshift=0.6em]{r}{\tau_*} \ar[phantom]{r}[font=\small]{\bot} & \cat{E}^\qoppa \ar[yshift=-0.6em]{l}{\tau^!}
    \end{tikzcd}
  \]
  that we abusively write with the same notation.
\end{definition}

\begin{notation}
  In the setting of \cref{pointed-free-algebras:category}, write $\class{M}^\qoppa$ for wide subcategory of $\cat{E}^\qoppa$ consisting of morphisms $(u \co A \to A', v \co B \to B') \co (A,B,f) \to (A',B',f')$ with $u, v \in \class{M}$.
\end{notation}

\begin{definition}
  \label{codomain-endofunctor}
  Given a category $\cat{E}$, write $\CodPtdEndo_{\cat{E}} = (\CodEndo_{\cat{E}},\codpt)$ for the well-pointed endofunctor on $\Arr{\cat{E}}$ defined as follows:
  \begin{enumerate}[label=(\roman*)]
  \item $\CodEndo_{\cat{E}}$ sends $f \co A \to B$ to the identity $\id_B \co B \to B$, with the evident functorial action.
  \item $\codpt \co \Id_{\cat{E}} \to \CodEndo_{\cat{E}}$ is given at $f \co A \to B$ by the square
  \[
    \begin{tikzcd}
      A \ar{d}[left]{f} \ar{r}{f} & B \ar[equals]{d} \\
      B \ar[equals]{r} & B \rlap{.}
    \end{tikzcd}
  \]
  \end{enumerate}
\end{definition}

\begin{proposition}
  \label{codomain-endofunctor-algebras}
  The category $\Alg{\CodPtdEndo_{\cat{E}}}$ is isomorphic over $\Arr{\cat{E}}$ to the full subcategory $\ArrFull{\cat{E}}{\cong} \hookrightarrow \Arr{\cat{E}}$ of isomorphisms in $\cat{E}$.
  \qed
\end{proposition}

\begin{lemma}
  \label{pointed-free-algebras:transfer}
  Let $(\cat{E},\class{M},\ptdendo{T}) \in \ConfP{\kappa}$.
  The pointed endofunctor $\CodPtdEndo_{\cat{E}}$ on $\Arr{\cat{E}}$ restricts to $\ArrFull{\cat{E}}{\class{M}}$, and the restricted pointed endofunctor transfers along $\adjunction{\ArrFull{\cat{E}}{\class{M}}}{\tau_*}{\tau^!}{\cat{E}^\qoppa}$ to define a well-pointed endofunctor $\ptdendo{T}^\qoppa$ on $\cat{E}^\qoppa$ whose unit is valued in $\class{M}^\qoppa$.
\end{lemma}
\begin{proof}
  That $\CodPtdEndo_{\cat{E}}$ restricts to $\ArrFull{\cat{E}}{\class{M}}$ is evident.
  Per \cref{transfer,transfer-well-pointed}, $\CodPtdEndo_{\cat{E}}$ transfers to define a well-pointed endofunctor $\ptdendo{T}^\qoppa$ on $\cat{E}^\qoppa$ provided that the following pushout in $\cat{E}^\qoppa$ exists for all $(A, B, f) \in \cat{E}^\qoppa$:
\begin{equation} \label{pointed-free-algebras:2}
  \begin{tikzcd}[column sep=7em]
    (A, B \sqcup_A TA, \inr) \ar{d}[left]{\epsilon_{(A,B,f)}} \ar{r}{\tau_* \codpt_{\tau^!(A,B,f)}} & (B,TB,\id_{TB}) \ar[dashed]{d} \\
    (A, B, f) \ar[dashed]{r} & \pushout (X,Y,k) \rlap{.}
  \end{tikzcd}
\end{equation}
The pushout $X$ of the domain components is trivially $B$; the pushout square is absolute and thus preserved by $T$.
Hence $Y$ is simply computed as a pushout of the codomain components:
\begin{equation} \label{pointed-free-algebras:3}
  \begin{tikzcd}[column sep=huge]
    B \sqcup_A TA \ar{d}[left]{\copair{\id}{f}} \ar{r}{\leibpushapp{\tau}(f\tau_A)} & TB \ar[dashed]{d} \\
    B \ar[dashed]{r}[below]{h} & \pushout Y \rlap{.}
  \end{tikzcd}
\end{equation}
The top map is the pushout application of $\tau$ to $f \tau_A$ and thus in $\class{M}$ by~\ref{pointed-configurations:leibniz-application}, since $f \tau_A$ is in $\class{M}$.
Hence the pushout \eqref{pointed-free-algebras:3} exists and the bottom map $h$ is in $\class{M}$ by~\ref{pointed-configurations:backdrop}.
The map $k \co TY \to X$ is given by functoriality of pushouts in the cube
\begin{equation} \label{pointed-free-algebras:3.5}
\begin{tikzcd}[column sep=3em]
  TA
  \ar[rr, "T(f\tau_A)"]
  \ar[dd, "\inr"]
  \ar[dr, equals]
  \ar[drrr, phantom, "\ulcorner"{pos=0.90,xslant=-1.8,xscale=3}]
&&
  TB
  \ar[dd, equals]
  \ar[dr, equals]
\\&
  TA
  \ar[rr, crossing over, "T(f\tau_A)", near start]
&&
  TB
  \ar[dd, dashed, "k"]
\\
  B \sqcup_A TA
  \ar[rr, "\leibpushapp{\tau}(f\tau_A)", near start]
  \ar[dr, "\copair{\id}{f}"']
  \ar[drrr, phantom, "\ulcorner"{pos=0.90,xslant=-1.8,xscale=3}]
&&
  TB
  \ar[dr, "k"]
\\&
  B
  \ar[from=uu, crossing over, "f", near end]
  \ar[rr, "h"']
&&
  Y \rlap{,}
\end{tikzcd}
\end{equation}
and is thus identified as indicated above with a leg of the bottom pushout square.
From this lower square, we see that $k \tau_B = k \circ \leibpushapp{\tau}(f\tau_A) \circ \inl = h \circ \copair{\id}{f} \circ \inl = h$.
Since we have established that $h$ is in $\class{M}$, this shows that $(X, Y, k) \in \cat{E}^\qoppa$.

Finally, we check that the unit of $\ptdendo{T}^\qoppa$ is valued in $\class{M}^\qoppa$.
At $(A, B, f) \in \cat{E}^\qoppa$, its value is the bottom horizontal map of~\eqref{pointed-free-algebras:2}.
The domain component is simply $f\tau_A \co A \to B$, which is in $\class{M}$ by definition of $\cat{E}^\qoppa$.
The codomain component is the bottom row of~\eqref{pointed-free-algebras:3}, which we have established is in $\class{M}$.
\end{proof}

\begin{lemma}
  \label{pointed-free-algebras:compare-algebras}
  Let $(\cat{E},\class{M},\ptdendo{T}) \in \ConfP{\kappa}$.
  Observe that the forgetful functor $U_{\ptdendo{T}} \co \Alg{\ptdendo{T}} \to \cat{E}$ factors as a composite
  \begin{equation} \label{pointed-free-algebras:0.75}
    \begin{tikzcd}
      (A,f) \ar[mapsto]{r} & (A,A,f) \\[-1.5em]
      \Alg{\ptdendo{T}} \ar{r} & \cat{E}^\qoppa \ar{r} & \cat{E} \\[-1.5em]
      & (A,B,f) \ar[mapsto]{r} & A \rlap{.}
    \end{tikzcd}
  \end{equation}
  Over $\cat{E}^\qoppa$, the category $\Alg{\ptdendo{T}} \to \cat{E}^\qoppa$ is equivalent to $U_{\cat{}\ptdendo{T}} \co \Alg{\cat{}\ptdendo{T}^\qoppa} \to \cat{E}^\qoppa$.
\end{lemma}
\begin{proof}
  Consider the composite square
  \[
    \begin{tikzcd}
      \Alg{\ptdendo{T}} \ar{d} \ar{r} & \ArrFull{\cat{E}}{\cong} \ar[hook]{d} \\
      \cat{E}^\qoppa \ar{d} \ar{r}[below]{\tau^!} & \ArrFull{\cat{E}}{\class{M}} \ar[hook]{d} \\
      \Comma{T}{\cat{E}} \ar{r}[below]{\tau^!} & \Arr{\cat{E}} \rlap{.}
    \end{tikzcd}
  \]
  The outer square is a weak 2-pullback and the  lower square is the weak 2-pullback~\eqref{pointed-free-algebras:0}, so the upper square is a weak 2-pullback by pullback pasting.
  By \cref{transfer-algebras,codomain-endofunctor-algebras}, the square
  \[
    \begin{tikzcd}
      \Alg{\ptdendo{T}^\qoppa} \ar{d} \ar{r} & \ArrFull{\cat{E}}{\cong} \ar[hook]{d} \\
      \cat{E}^\qoppa \ar{r}[below]{\tau^!} & \ArrFull{\cat{E}}{\class{M}}
    \end{tikzcd}
  \]
  is also a weak 2-pullback, so the result follows by uniqueness of weak 2-pullbacks up to equivalence.
\end{proof}

\begin{lemma}
  \label{pointed-free-algebras:composition}
  For any $(\cat{E},\class{M},\ptdendo{T}) \in \ConfP{\kappa}$, $\class{M}^\qoppa$ has colimits of $(1 + \alpha)$-chains in $\cat{E}^\qoppa$ for all $\alpha < \kappa$.
\end{lemma}
\begin{proof}
  Consider a diagram $(A, B, f) \co (1+\alpha,\preceq) \to \cat{E}^\qoppa$ given by functors $A, B \co (1+\alpha,\preceq) \to \cat{E}$ factoring through $\class{M}$ and a natural transformation $f \co TA \to B$ such that $f \circ \tau A$ has components in $\class{M}$.
  We will show that the colimit $(X, Y, k) \defeq \colim_{1+\alpha} (A, B,f)$ exists in $T \downarrow \cat{E}$ and lifts to $\cat{E}^\qoppa$.
  By~\cref{pointed-configurations:functor} and~\ref{pointed-configurations:backdrop}, the colimits of $A$, $B$, $T A$ exist and have coprojections in $\class{M}$.
  We take $X = \colim_{1+\alpha} A$ and compute $Y$ and $k$ via the following pushout:
  \begin{equation} \label{pointed-free-algebras:4}
    \begin{tikzcd}[column sep=huge]
      \colim_{1+\alpha} A \ar{dr}[below left]{\tau \colim_{1+\alpha} A} \ar{r}{\colim_{1+\alpha} \tau A} &[1em] \colim_{1+\alpha} TA \ar{d} \ar{r}{\colim_{1+\alpha} f} & \colim_{1+\alpha} B \ar[dashed]{d} \\
      & T \colim_{1+\alpha} A \ar[dashed]{r}{k} & \pushout Y \rlap{.}
    \end{tikzcd}
  \end{equation}
  Again by~\cref{pointed-configurations:functor} and~\ref{pointed-configurations:backdrop}, the middle vertical map is in $\class{M}$.
  By~\ref{pointed-configurations:backdrop}, the pushout exists and the rightmost vertical map is in $\class{M}$.
  Thus the colimit $(X,Y,k)$ exists in $T \downarrow \cat{E}$.
  By~\ref{pointed-configurations:backdrop}, the composite top row is in $\class{M}$.
  The composite $X \overset{\tau_X}{\longrightarrow} TX \overset{k}{\longrightarrow} Y$ thus factors as a sequence $\colim_{1+\alpha} A \to \colim_{1+\alpha} B \to Y$ of maps in $\class{M}$, so is itself in $\class{M}$.
  Thus the colimit lifts to $\cat{E}^\qoppa$.
\end{proof}

\begin{lemma}
  \label{pointed-free-algebras:end}
  For any $(\cat{E},\class{M},\ptdendo{T}) \in \ConfP{\kappa}$ and $\ptdendo{T}^\qoppa$-algebraized $\kappa$-chain $(X,x)$ such that $X$ and $x$ factor through $\class{M}^\qoppa$, $X$ admits a colimit in $\cat{E}^\qoppa$.
\end{lemma}
\begin{proof}
  Write $X = (A, B, f) \co (\kappa,\preceq) \to \cat{E}^\qoppa$.
  Our colimit is $(X,Y,k)$ where $X = \colim_\kappa A$, $Y = \colim_\kappa B$, and $k$ is the composite
  \[
    \begin{tikzcd}
      T \colim_\kappa A \ar{r}{\cong} & \colim_\kappa TA \ar{r}{\colim_\kappa f} &[2em] \colim_\kappa B \rlap{.}
    \end{tikzcd}
  \]
  To check that $k\tau_X$ is in $\class{M}$, observe that \cref{transfer-chain} gives us a $\CodPtdEndo_{\cat{E}}$-algebraized $\kappa$-chain $\tau^!(X,x)$ in $\ArrFull{\cat{E}}{\class{M}}$, which we can also see as a $\CodPtdEndo_{\cat{E}}$-algebraized $\kappa$-chain in $\Arr{\cat{E}}$.
  This chain has a colimit in $\Arr{\cat{E}}$, namely $\colim_{\alpha < \kappa} f_\alpha\tau_{A_\alpha}$, which is isomorphic to $k\tau_X$.
  Since $\CodPtdEndo_{\cat{E}}$ preserves this colimit, it follows from \cref{preserved-colimit-of-algebraized-chain-is-algebra} that the colimit is a $\CodPtdEndo_{\cat{E}}$-algebra, \ie, an isomorphism.
  Hence $k\tau_X$ is an isomorphism and in particular belongs to $\class{M}$.
\end{proof}

\begin{corollary}
  \label{pointed-to-wellpointed-configuration}
  For any $(\cat{E},\class{M},\ptdendo{T}) \in \ConfP{\kappa}$, we have $(\cat{E}^\qoppa,\class{M}^\qoppa,\ptdendo{T}^\qoppa) \in \ConfWP{\kappa}$.
\end{corollary}
\begin{proof}
  Condition \ref{wellpointed-configurations:unit} is part of \cref{pointed-free-algebras:transfer} and condition \ref{wellpointed-configurations:composition} is \cref{pointed-free-algebras:composition}.
  Condition \ref{wellpointed-configurations:convergence} follows from \cref{pointed-free-algebras:end}, \ref{pointed-configurations:convergence}, and commutativity of colimits.
\end{proof}

\begin{lemma}
  \label{pointed-to-wellpointed-configuration-functor}
  The assignment of \cref{pointed-to-wellpointed-configuration} extends to a functor $\ConfP{\kappa} \to \ConfWP{\kappa}$.
\end{lemma}
\begin{proof}
  Let $(F,\gamma) \co (\cat{E}_1, \class{M}_1, \ptdendo{T_1}) \to (\cat{E}_2, \class{M}_2, \ptdendo{T_2})$ be a morphism in $\ConfP{\kappa}$.
  To see that $F$ lifts to a functor $F^\qoppa \co \cat{E}_1^\qoppa \to \cat{E}_2^\qoppa$, we examine the pullback square~\eqref{pointed-free-algebras:0} defining $\cat{E}_1^\qoppa$ and $\cat{E}_2^\qoppa$.
  Since $(F,\gamma)$ is a map of pointed endofunctors, we have an isomorphism
  \[
    \begin{tikzcd}[sep=large]
      \Comma{T_1}{\cat{E}_1} \ar{d} \ar{r}{\tau_1^!} & \Arr{\cat{E}_1} \ar{d}{\Arr{F}} \\
      \Comma{T_2}{\cat{E}_2} \ar[phantom]{ur}[sloped]{\cong}[above left]{\gamma^!} \ar{r}[below]{\tau_2^!} & \Arr{\cat{E}_2} \rlap{.}
    \end{tikzcd}
  \]
  where the left vertical functor sends $(A,B,f \co T_1A \to B)$ to $(FA,FB,Ff \circ \gamma^{-1}_A \co T_2FA \to FB)$.
  By assumption, $\Arr{F} \co \Arr{\cat{E}_1} \to \Arr{\cat{E}_2}$ restricts to a functor $\ArrFull{(\cat{E}_1)}{\class{M}_1} \to \ArrFull{(\cat{E}_2)}{\class{M}_2}$.
  Together, these induce the lift $F^\qoppa \co \cat{E}_1^\qoppa \to \cat{E}_2^\qoppa$.
  Note that $F^\qoppa$ maps $\class{M}_1^\qoppa$ to $\class{M}_2^\qoppa$, since $F$ maps $\class{M}_1$ to $\class{M}_2$.

  To extend $F^\qoppa$ to a morphism $(\cat{E}_1,\ptdendo{T_1^\qoppa}) \to (\cat{E}_2,\ptdendo{T_2^\qoppa})$ in $\PtdEndo_s$, we use the functoriality of transfer (\cref{transfer-functoriality}).
  Clearly $\Arr{F} \co \ArrFull{(\cat{E}_1)}{\class{M}_1} \to \ArrFull{(\cat{E}_2)}{\class{M}_2}$ defines a strong morphism of pointed endofunctors from $\CodPtdEndo_{\cat{E}_1}$ to $\CodPtdEndo_{\cat{E}_2}$.
  Using that $F$ preserves the pushout~\eqref{pointed-free-algebras:0.5} defining $(\tau_1)_*$, we see $F^\qoppa (\tau_1)_* \cong (\tau_2)_*\Arr{F}$ and indeed that $\Arr{F}$ pairs with $F^\qoppa$ to define a strong morphism of adjunctions from $\adjunction{\ArrFull{(\cat{E}_1)}{\class{M}_1}}{(\tau_1)_*}{(\tau_1)^!}{\cat{E}_1^\qoppa}$ to $\adjunction{\ArrFull{(\cat{E}_2)}{\class{M}_2}}{(\tau_2)_*}{(\tau_2)^!}{\cat{E}_2^\qoppa}$.
  Finally, $F^\qoppa$ preserves the pushouts in the definition of the transfer $\ptdendo{T_1^\qoppa}$ because these are levelwise cobase changes of maps in $\class{M}_1$, see \eqref{pointed-free-algebras:3.5}.
  Thus we have a transferred morphism $(\cat{E}_1,\ptdendo{T_1^\qoppa}) \to (\cat{E}_2,\ptdendo{T_2^\qoppa})$ extending $F^\qoppa$ by \cref{transfer-functoriality}.

  To see that $F^\qoppa$ preserves colimits of $(1+\alpha)$-chains in $\class{M}_1$ for any $\alpha < \kappa$ and of $\ptdendo{T_1^\qoppa}$-algebraized $\kappa$-chains, recall their construction in \cref{pointed-free-algebras:composition,pointed-free-algebras:end} and note that $F$ preserves all of the involved colimits, namely cobase changes of maps in $\class{M}_1$ and colimits of $(1+\alpha)$- and $\kappa$-chains in $\class{M}_1$.
\end{proof}

\begin{theorem} \label{pointed-free-algebras}
The functor $\ConfP{\kappa} \to \Fun_s$ sending $(\cat{E}, \class{M}, \ptdendo{T})$ to the forgetful functor $U_{\ptdendo{T}} \co \Alg{\ptdendo{T}} \to \cat{E}$ lifts through the projection $\Adju_s \to \Fun_s$ of the right adjoint.
\end{theorem}
\begin{proof}
  Combining \cref{pointed-to-wellpointed-configuration-functor} and \cref{wellpointed-free-algebras}, we have a composite functor
  \[
    \begin{tikzcd}
      \ConfP{\kappa} \ar{r} & \ConfWP{\kappa} \ar{r} & \Adju_s
    \end{tikzcd}
  \]
  sending a configuration $(\cat{E},\class{M},\ptdendo{T})$ to the free algebra adjunction for $\ptdendo{T}^\qoppa$.
  Writing $\dom \co \cat{E}^\qoppa \to \cat{E}$ for the functor sending $(A,B,f)$ to $A$, we recall from \cref{pointed-free-algebras:compare-algebras} that we have an equivalence $\Alg{\ptdendo{T}} \simeq \Alg{\ptdendo{T}^\qoppa}$ over $\cat{E}^\qoppa$, providing a factorization
  \[
    \begin{tikzcd}[column sep=large]
      \Alg{\ptdendo{T}} \ar[bend left=20]{rrr}{U_{\ptdendo{T}}} \ar{r}[below]{\simeq} & \Alg{\ptdendo{T}^\qoppa} \ar{r}[below]{U_{\ptdendo{T}^\qoppa}} & \cat{E}^\qoppa \ar{r}[below]{\dom} & \cat{E} \rlap{.}
    \end{tikzcd}
  \]
  The projection $\dom \co \cat{E}^\qoppa \to \cat{E}$ has a left adjoint sending $A$ to $(A,TA, \id_{TA} \co TA \to TA)$, so a left adjoint to $U_{\ptdendo{T}^\qoppa}$ induces a left adjoint to $U_{\ptdendo{T}}$, the object part of our desired functor $\ConfP{\kappa} \to \Adju_s$.

  For the functorial action, given $(F,\gamma) \co (\cat{E}_1,\class{M}_1,\ptdendo{T}_1) \to (\cat{E}_2,\class{M}_2,\ptdendo{T}_2)$, \cref{wellpointed-free-algebras} gives us a strong morphism of adjunctions from the free algebra adjunction for $\ptdendo{T}_1^\qoppa$ to that for $\ptdendo{T}_2^\qoppa$.
  It is straightforward to check that $F^\qoppa$ and $F$ pair to define a strong morphism of adjunctions from the pair with right adjoint $\dom \co \cat{E}_1^\qoppa \to \cat{E}_1$ to the pair with right adjoint $\dom \co \cat{E}_2^\qoppa \to \cat{E}_2$.
  Composing these yields a strong morphism of adjunctions from the free algebra adjunction for $\ptdendo{T}_1$ to that for $\ptdendo{T}_2$.
\end{proof}

\begin{theorem}
  \label{pointed-free-monad}
  \input{pointed-free-monad}
\end{theorem}
\begin{proof}
  As in \cref{wellpointed-free-monad}, this is an immediate consequence of \cref{pointed-free-algebras}, using that the free and algebraically free monad on a pointed endofunctor is given by the monad of the free algebra adjunction \cite[Proposition 22.2 and Theorem 22.3]{kelly:80}.
\end{proof}

%% file: soa.tex
\section{The algebraic small object argument}
\label{sec:soa}

We now use \cref{sec:free-monad-sequence} as a tool to analyze the algebraic small object argument and prove our saturation theorems.

We recall the basic theory of \textsc{awfs}'s in \cref{sec:soa:prelims}.
In \cref{sec:soa:config}, we instantiate \citeauthor{bourke-garner:16}'s results with the free monad construction from \cref{sec:free-monad-sequence} to obtain a variation of algebraic small object argument for (possibly non-cocomplete) categories equipped with a backdrop (\cref{soa-backdrop}).
In \cref{sec:soa:notions} we introduce \emph{notions of composable structure}.
A notion of composable structure on a category is a notion of structured morphism such that identities have canonical structure and structured morphisms are closed under composition.
We define what it means for a notion of composable structure to be \emph{left-connected} (following Bourke and Garner) and \emph{cellular} and observe that the double categories $\DCoalg{\comonad{L}}$ and $\DCoalg{\UnderPtd{\comonad{L}}}$ associated to an \textsc{awfs} $(\comonad{L},\monad{R})$ have these properties.
Left-connected and cellular notions of composable structure are the targets for our main theorems.

In \cref{sec:soa:points}, we return to the abstract setting of \cref{sec:free-monad-sequence} and give conditions under which, for a given notion of composable structure, structure on the unit of a pointed endofunctor yields structure on the unit of its free monad (\cref{kelly-pointed-left-category}).
Finally, we prove our main results (\cref{soa-unit-vertical-point,soa-copointed-coalgebras-functor,soa-copointed-coalgebras-double-functor}) in \cref{sec:soa:algebraic-saturation} by applying the results of \cref{sec:soa:points} to the specific case of the free monad construction that produces a cofibrantly generated \textsc{awfs}.

\subsection{Preliminaries}
\label{sec:soa:prelims}

\subsubsection{Algebraic weak factorization systems}

The original definition of \textsc{awfs} is due to \textcite{grandis-tholen:06} under the name \emph{natural weak factorization system}.
The definition we use here includes an additional distributivity condition which was introduced by Garner \cite{garner:09}.
Recall from \cref{sec:intro:results} that an \textsc{awfs} $(\comonad{L},\monad{R})$ on a category $\cat{E}$ consists of a pair  of a comonad $\comonad{L} = (L,\Phi,\Sigma)$ and monad $\monad{R} = (R,\Lambda,\Pi)$ such that $L,R$ define a functorial factorization, $\Phi \co L \to \Id$ and $\Lambda \co \Id \to R$ are of the form
\[
  \begin{tikzcd}
    X \ar[phantom]{dr}{\Phi_f} \ar{d}[left]{Lf} \ar[equals]{r} & X \ar{d}{f} \\
    Ef \ar{r}[below]{Rf} & Y
  \end{tikzcd}
  \qquad
  \begin{tikzcd}
    X \ar[phantom]{dr}{\Lambda_f} \ar{d}[left]{f} \ar{r}{Lf} & Ef \ar{d}{Rf} \\
    Y \ar[equals]{r} & Y
  \end{tikzcd}
\]
respectively, and we have a distributive law \cite{beck:69} of $\comonad{L}$ over $\monad{R}$.

\begin{remark}
  The (co)unit laws imply that the comultiplication $\Sigma$ and multiplication $\Pi$ are of the forms
  \[
    \begin{tikzcd}
      X \ar[phantomcenter]{dr}{\Sigma_f} \ar{d}[left]{Lf} \ar[equals]{r} & X \ar{d}{LLf} \\
      Ef \ar{r}[below]{\delta_f} & ELf
    \end{tikzcd}
    \qquad
    \text{and}
    \qquad
    \begin{tikzcd}
      ERf \ar[phantomcenter]{dr}{\Pi_f} \ar{d}[left]{RRf} \ar{r}{\mu_f} & Ef \ar{d}{Rf} \\
      Y \ar[equals]{r} & Y
    \end{tikzcd}
  \]
  for some $\delta \co E \to EL$ and $\mu \co ER \to E$.
\end{remark}

\begin{proposition}[{\cite[Remark 2.17]{garner:09}~\cite[Remark 2.11]{riehl:11}}]
  An \textsc{awfs} $(\comonad{L},\monad{R})$ has an underlying \textsc{wfs} whose left and right classes may be described either as
  \begin{enumerate}[label=(\roman*)]
  \item the (co)algebras for the (co)pointed endofunctors $\UnderPtd{\comonad{L}}$ and $\UnderPtd{\monad{R}}$ respectively; or
  \item the retracts in $\Arr{\cat{E}}$ of (co)algebras for the (co)monads $\comonad{L}$ and $\monad{R}$ respectively. \qed
  \end{enumerate}
\end{proposition}

\begin{notation}
  \label{awfs-forgetful-double-functors}
  Given an \textsc{awfs} $(\comonad{L},\monad{R})$ on $\cat{E}$, we write $\bm{L} \co \Arr{\cat{E}} \to \Coalg{\comonad{L}}$ and $\bm{R} \co \Arr{\cat{E}} \to \Alg{\monad{R}}$ for the functors sending a map to its cofree $\comonad{L}$-coalgebra and free $\monad{R}$-algebra respectively.
\end{notation}

\begin{notation}
  Fix an \textsc{awfs} $(\comonad{L},\monad{R})$.
  An object of $\Coalg{\UnderPtd{\comonad{L}}}$ is, by definition, a morphism $f$ in $\cat{E}$ equipped with a section of $\Phi_f \co Lf \to f$.
  Such a section is determined by a map $s$ in $\cat{E}$ fitting in the diagram
  \[
    \begin{tikzcd}
      A \ar{d}[left]{f} \ar[equals]{r} & A \ar{d}{Lf} \ar[equals]{r} &[1em] A \ar{d}{f} \\
      B \ar[bend right=25,equals]{rr} \ar[dashed]{r}{s} & Ef \ar{r}{Rf} & B \rlap{.}
    \end{tikzcd}
  \]
  As such, we speak of objects of $\Coalg{\UnderPtd{\comonad{L}}}$ as pairs $(f,s)$.
  An object of $\Coalg{\comonad{L}}$ is such a pair with the additional property that $E(\id_{A},s)s = \delta_fs \co f \to ELf$.
\end{notation}

\subsubsection{Double categories}

Bourke and Garner \cite{bourke-garner:16,bourke-garner:16-ii,bourke:23} observed that the category of (co)monad (co)algebras associated to an \textsc{awfs} can be seen as the category of vertical morphisms of a double category: identities have canonical (co)algebra structures, and (co)algebra structures on a pair of composable maps can themselves be composed.
We recall first some double-categorical terminology, then the double categories of copointed endofunctor and comonad coalgebras induced by an \textsc{awfs}.

\begin{definition}[{{\textcite[\S7.1]{grandis-pare:99}}}]
  A \emph{pseudo double category} $\dcat{A}$ is a ``pseudo category object in the 2-category of categories''.
  By this, one means that $\dcat{A}$ consists of
  \begin{enumerate}
  \item categories $\dcat{A}_0$ and $\dcat{A}_1$;
  \item functors
    \[
      \begin{tikzcd}
        \dcat{A}_0 \ar{r}[description]{\idv} &[1em] \dcat{A}_1 \ar[bend right]{l}[above]{\vdom} \ar[bend left]{l}{\vcod} & \ar{l}[above]{\vcirc} \dcat{A}_1 \times_{\dcat{A}_0} \dcat{A}_1 \rlap{,}
      \end{tikzcd}
    \]
    making the diagrams
    \[
      \begin{tikzcd}[column sep=large]
        & \dcat{A}_0 \\
        \dcat{A}_0 \ar[equals]{ur} \ar[equals]{dr} \ar{r}{\idv} & \dcat{A}_1 \ar{u}[right]{\vdom} \ar{d}{\vcod} \\
        & \dcat{A}_0
      \end{tikzcd}
      \qquad
      \begin{tikzcd}
        \dcat{A}_1 \ar{r}{\vdom} & \dcat{A}_0 \\
        \dcat{A}_1 \times_{\dcat{A}_0} \dcat{A}_1 \ar{u}{\projr} \ar{d}[left]{\projl} \ar{r}{\vcirc} & \dcat{A}_1 \ar{u}[right]{\vdom} \ar{d}{\vcod} \\
        \dcat{A}_1 \ar{r}[below]{\vcod} & \dcat{A}_0
      \end{tikzcd}
    \]
    commute strictly;
  \item natural isomorphisms $\alpha,\lambda,\rho$ witnessing the associativity and unitality of the composition $\vcirc$ and identity $\idv$ and satisfying various coherence laws.
  \end{enumerate}
  We refer to \textcite[\S7.1]{grandis-pare:99} for a complete definition.
  A \emph{double category} is a pseudo double category in which the associator and unitors for $\vcirc$ and $\idv$ are strict equalities, in which case the coherence laws are automatically satisfied.
\end{definition}

We call objects of $\dcat{A}_0$ \emph{objects of $\dcat{A}$}, morphisms of $\dcat{A}_0$ \emph{horizontal morphisms}, objects of $\dcat{A}_1$ \emph{vertical morphisms}, and morphisms in $\dcat{A}_1$ \emph{squares}.
We denote horizontal morphisms of a double category by ordinary arrows, as in $f \co A \to B$.
We write vertical morphisms of a double category in boldface, \eg, $\bm{f}$, and indicate that $\vdom \bm{f} = A$ and $\vcod \bm{f} = B$ by writing a dotted arrow, as in $\bm{f} \co A \verto B$.
We draw squares $\beta \co \bm{f} \to \bm{g}$ as squares, as in
\[
  \begin{tikzcd}
    A \ar[phantom]{dr}{\beta} \ar[verto]{d}[left]{\bm{f}} \ar{r}{h} & B \ar[verto]{d}{\bm{g}} \\
    C \ar{r}[below]{k} & D
  \end{tikzcd}
\]
where $h = \vdom \beta$ and $k = \vcod \beta$, and draw vertical and horizontal composition of squares as vertical and horizontal juxtaposition respectively.\footnote{We never draw any diagram with more than two vertical ``layers'', so there is no chance of misinterpreting vertical juxtaposition as strictly associative.}

\begin{notation}
  Given a pseudo double category $\dcat{A} = (\dcat{A}_0,\dcat{A}_1,\ldots)$, write $\VertArr{\dcat{A}}$ for its category of vertical morphisms $\dcat{A}_1$.
\end{notation}

\begin{definition}
  A double category $\dcat{A}$ is \emph{thin} when every square of $\dcat{A}$ is uniquely determined by its boundary.
\end{definition}

\begin{example}
  For any category $\cat{E}$, there is a thin \emph{double category of squares} $\Sq{\cat{E}}$ whose objects are objects of $\cat{E}$, whose horizontal and vertical morphisms are both the morphisms of $\cat{E}$, and whose squares are commutative squares in $\cat{E}$.
\end{example}

\begin{definition}
  A \emph{pseudo double functor} $F \co \dcat{A} \to \dcat{B}$ between pseudo double categories consists of functors $F_0 \co \dcat{A}_0 \to \dcat{B}_0$ and $F_1 \co \dcat{A}_1 \to \dcat{B}_1$ such that $\vdom F_1 = F_0 \vdom$ and $\vcod F_1 = F_0 \vcod$, together with horizontally invertible comparison squares
  \[
    \begin{tikzcd}[row sep=large]
      FA \ar[phantom]{dr}{\phi^{\idv}_A} \ar[verto]{d}[left]{\idv} \ar[equals]{r} & FA \ar[verto]{d}{F(\idv)} \\
      FA \ar[equals]{r} & FA
    \end{tikzcd}
    \qquad
    \begin{tikzcd}[column sep=large]
      FA \ar[phantom]{ddr}{\phi^{\vcirc}_{\bm{g},\bm{f}}} \ar[verto]{d}[left]{F\bm{f}} \ar[equals]{r} & FA \ar[verto]{dd}{F(\bm{g} \vcirc \bm{f})} \\
      FB \ar[verto]{d}[left]{F\bm{g}} \\
      FC \ar[equals]{r} & FC
    \end{tikzcd}
  \]
  satisfying naturality and coherence conditions; again, we refer to \textcite[\S7.2]{grandis-pare:99} for a complete definition.
  A \emph{double functor} is a pseudo double functor for which the comparison isomorphisms above are identities, in which case the coherence conditions are automatically satisfied.
\end{definition}

\textcite{bourke-garner:16} define a double category for (co)monad (co)algebras associated to an \textsc{awfs}.
The same definitions also yield a double category for (co)pointed endofunctor coalgebras:

\begin{definition}[{{\cite[\S2]{bourke-garner:16}}}]
  For any \textsc{awfs} $(\comonad{L},\monad{R})$ on a category $\cat{E}$, we have a double category $\DCoalg{\UnderPtd{\comonad{L}}}$ in which
  \begin{enumerate}[label=(\roman*)]
  \item objects and horizontal morphisms are objects and morphisms of $\cat{E}$ respectively;
  \item vertical morphisms $(f,s) \co A \verto B$ are morphisms $f \co A \to B$ of $\cat{E}$ equipped with an $\UnderPtd{\comonad{L}}$-coalgebra structure $s \co B \to Ef$;
  \item squares
    \[
      \begin{tikzcd}
        A \ar[verto]{d}[left]{\bm{f}} \ar{r}{h} & C \ar[verto]{d}{\bm{g}} \\
        B \ar{r}[below]{k} & D
      \end{tikzcd}
    \]
    are morphisms $(h,k) \co \bm{f} \to \bm{g}$ in $\Coalg{\UnderPtd{\comonad{L}}}$.
  \end{enumerate}
  In other words, we set $\VertArr{\DCoalg{\UnderPtd{\comonad{L}}}} \defeq \Coalg{\UnderPtd{\comonad{L}}}$.
  
  The vertical identities are given by $\idv_A \defeq (\id_A,L(\id_A)) \co A \verto A$.
  The vertical composition of morphisms $(g,t) \co B \verto C$ and $(f,s) \co A \verto B$ is given by the pair $(gf,u)$ where $u$ is the bottom horizontal composite of the rectangle
  \[
    \begin{tikzcd}
      A \ar{d}[left]{f} \ar[equals]{r} & A \ar{d}{f} \ar[equals]{r} &[2em] A \ar{d}[left]{Lf} \ar[equals]{r} &[5em] A \ar{d}[left]{L(gf)}\ar[equals]{r} & A \ar{d}{LL(gf)} \ar[equals]{r} & A \ar{dd}{L(gf)} \\
      B \ar{d}[left]{g} \ar[equals]{r} & B \ar{d}{Lg} \ar{r}{s} & Ef \ar{d}[left]{L(g \circ Rf)} \ar{r}{E(\id_A,g)} & E(gf) \ar[equals]{dr} \ar{d}[left]{LR(gf)} \ar{r}{\delta_{gf}} & EL(gf) \ar{d}{RL(gf)} \\
      C \ar{r}[below]{t} & Eg \ar{r}[below]{E(s,\id_C)} & E(g \circ Rf) \ar{r}[below]{E(E(\id_A,g),\id_C)} & ER(gf) \ar{r}[below]{\mu_{gf}} & E(gf) \ar[equals]{r} & E(gf)
    \end{tikzcd}
  \]
  which defines a section of $\Phi_{gf}$.
  We leave the unit and associativity laws for vertical composition as a tedious exercise for the reader.
  The forgetful functor from $\Coalg{\UnderPtd{\comonad{L}}}$ extends to a double functor $U_{\UnderPtd{\comonad{L}}} \co \DCoalg{\UnderPtd{\comonad{L}}} \to \Sq{\cat{E}}$.
\end{definition}

Restricting to those vertical morphisms which are coalgebras for the comonad $\comonad{L}$, we likewise have the double category $\DCoalg{\comonad{L}}$ and forgetful double functor $U_{\comonad{L}} \co \DCoalg{\comonad{L}} \to \Sq{\cat{E}}$ as defined by Bourke and Garner.
By the self-duality of \textsc{awfs}'s, we also have double categories $U_{\monad{R}} \co \DAlg{\monad{R}} \to \Sq{\cat{E}}$ and $U_{\UnderPtd{\monad{R}}} \co \DAlg{\UnderPtd{\monad{R}}} \to \Sq{\cat{E}}$ whose vertical morphisms are algebras for the monad of the \textsc{awfs} and its underlying pointed endofunctor.

\subsubsection{Cofibrant generation}

Now we can recall what it means for an \textsc{awfs} $(\comonad{L},\monad{R})$ to be cofibrantly generated by a diagram $u \co \cat{J} \to \Arr{\cat{E}}$, namely that an $\monad{R}$-algebra on a map consists of a coherent choice of lifts against $u$.

\begin{definition}[{{\cite[\S5.2]{bourke-garner:16}}}]
  Given a diagram $u \co \cat{J} \to \Arr{\cat{E}}$, write $\drlp{u} \co \drlp{\cat{J}} \to \Sq{\cat{E}}$ for the double category in which
  \begin{enumerate}[label=(\roman*)]
  \item objects and horizontal morphisms are objects and morphisms of $\cat{E}$ respectively;
  \item vertical morphisms $(f,\psi) \co X \verto Y$ are morphisms in $\cat{E}$ equipped with a choice of filler $\psi(i,\alpha)$ for each lifting problem $\alpha \co u_i \to f$ such that for every $t \co j \to i$ in $\cat{J}$ and $\alpha \co u_i \to f$, the triangle
    \[
      \begin{tikzcd}[row sep=large]
        A_j \ar[phantomcenter]{dr}{u_t} \ar{d} \ar{r} & A_i \ar[phantomcenter]{dr}{\alpha} \ar{d} \ar{r} & X \ar{d}{f} \\
        B_j \ar[dashed,bend right=10]{urr} \ar{r} & B_i \ar[dashed,bend right=15]{ur} \ar{r} & Y
      \end{tikzcd}
    \]
    formed by $\psi(i,\alpha)$ and $\psi(j,\alpha u_t)$ commutes;
  \item squares
    \[
      \begin{tikzcd}
        X \ar[phantomcenter]{dr}{\beta} \ar[verto]{d}[left]{(f,\psi)} \ar{r}{x} & X' \ar[verto]{d}{(f',\psi')} \\
        Y \ar{r}[below]{y} & Y'
      \end{tikzcd}
    \]
    are commutative squares $\beta = (x,y) \co f \to f'$ in $\cat{E}$ such that for each lifting problem $\alpha \co u_i \to f$, we have $\psi'(i,\beta\alpha) = x\psi(i,\alpha)$.
  \end{enumerate}
  For identity and composition of vertical morphisms in $\drlp{\cat{J}}$, we use the standard arguments that a class of maps defined by a lifting property contains all identities and is closed under composition.
\end{definition}

\begin{notation}[{{\emph{cf.}~\cite[Proposition 3.8]{garner:09}}}]
  Given $u \co \cat{J} \to \Arr{\cat{E}}$, we write $\rlp{\cat{J}}$ for the category of vertical morphisms of $\drlp{\cat{J}}$ and $\rlp{u} \co \rlp{\cat{J}} \to \Arr{\cat{E}}$ for the underlying map functor.
\end{notation}

\begin{definition}
  \label{cofibrantly-generated}
  Let $\cat{E}$ be a category and $u \co \cat{J} \to \Arr{\cat{E}}$ be a diagram of arrows.
  An \textsc{awfs} $(\comonad{L},\monad{R})$ on $\cat{E}$ is \emph{cofibrantly generated by $u$} when we have an isomorphism
  \[
    \begin{tikzcd}[column sep=tiny]
      \DAlg{\monad{R}} \ar{dr}[below left,pos=0.4]{U_{\monad{R}}} \ar[dashed]{rr}{\cong} & & \drlp{\cat{J}} \ar{dl}[pos=0.4]{\drlp{u}} \\
      & \Sq{\cat{E}}
    \end{tikzcd}
  \]
  of double categories.
\end{definition}

The projection sending an \textsc{awfs} to its double category of $\monad{R}$-algebras is 2-fully faithful \cite[Proposition 2]{bourke-garner:16}, so the \textsc{awfs} cofibrantly generated by a category is unique up to isomorphism if it exists.
Bourke and Garner also define what it means for an \textsc{awfs} to be generated by a \emph{double} functor $\dcat{J} \to \Sq{\cat{E}}$ \cite[\S6]{bourke-garner:16}, but we will not consider this more general case.

\subsection{Configurations for the algebraic small object argument}
\label{sec:soa:config}

We now revisit Bourke and Garner's presentation of the algebraic small object argument, which simplifies Garner's original presentation.
The construction rests on the following observation:

\begin{proposition}[{{Corollary of \cite[Theorem 6]{bourke-garner:16}}}]
  \label{bourke-garner-recognition}
  Let $u \co \cat{J} \to \Arr{\cat{E}}$ be a diagram.
  If the $\rlp{u} \co \rlp{\cat{J}} \to \Arr{\cat{E}}$ is a monadic functor, then its associated monad is, up to isomorphism, the monad of an \textsc{awfs} cofibrantly generated by $u$.
\qed
\end{proposition}

That is, we can construct an \textsc{awfs} $u \co \cat{J} \to \Arr{\cat{E}}$ by exhibiting $\rlp{u} \co \rlp{\cat{J}} \to \Arr{\cat{E}}$ as the forgetful functor for a category of algebras for a monad.
As a corollary, if $\rlp{u} \co \rlp{\cat{J}} \to \Arr{\cat{E}}$ is the forgetful functor for a category of algebras for a \emph{pointed endofunctor}, then we have an \textsc{awfs} cofibrantly generated by $u$ as soon as we can construct the algebraically free monad on that pointed endofunctor.

Bourke and Garner show that when $\cat{E}$ is locally presentable, we can always exhibit such a pointed endofunctor and construct its algebraically free monad by appeal to Kelly \cite[Proposition 16]{bourke-garner:16}.
We instead apply our results from \cref{sec:free-monad-sequence} assuming a $\kappa$-backdrop in $\cat{E}$ compatible with the generating category and a $\kappa$-smallness condition on the domains of the generators.

\begin{definition}
  When $\cat{E}$ is a category and $\class{M},\class{N}$ are wide subcategories of $\cat{E}$, write $\Class{\cat{E}}{\class{M}}{\class{N}}$ for the wide subcategory of $\Arr{\cat{E}}$ consisting of squares $f \to g$ of the form
  \[
    \begin{tikzcd}[column sep=large]
      A \ar{d}[left]{f} \ar{r}{\in \class{M}} & C \ar{d}{g} \\
      B \ar{r}[below]{\in \class{N}} & D \rlap{.}
    \end{tikzcd}
  \]
\end{definition}

\begin{definition}[{{\textcite{kock:66}}}]
  \label{density-comonad}
  Given a functor $v \co \cat{C} \to \cat{D}$, its \emph{density comonad} $\Den{v} \co \cat{D} \to \cat{D}$ is (when it exists) the pointwise left Kan extension $\Lan_v v \co \cat{D} \to \cat{D}$ of $v$ along itself:
  \[
    \Den{v}(d) = \colim \left(
    \begin{tikzcd}[cramped,ampersand replacement=\&,sep=small]
    \Comma{v}{d} \ar{r}{\proj{}} \& \cat{C} \ar{r}{v} \&[1em] \cat{D}
    \end{tikzcd}\right) \rlap{.}
  \]
\end{definition}

\begin{definition}
  \label{compatible}
  Let $\cat{E}$ be a category, $\kappa > 0$ be a limit ordinal, and $\class{M}$ be a $\kappa$-backdrop in $\cat{E}$.
  Given a diagram $u \co \cat{J} \to \Arr{\cat{E}}$, we say \emph{$u$ is compatible with $\class{M}$} if
  \begin{enumerate}[label=(\alph*),ref=\thedefinition(\alph*)]
  \item \label{backdrop:density} $u$ admits a density comonad $\Den{u} \co \Arr{\cat{E}} \to \Arr{\cat{E}}$ valued in $\class{M}$;
  \item \label{backdrop:density-leibniz} the pushout application $\leibpushapp{\Den{u}} \co \Arr{(\Arr{\cat{E}})} \to \Arr{\cat{E}}$ sends squares in $\LevelClass{\cat{E}}{\class{M}}$ to morphisms in $\class{M}$.
  \end{enumerate}
\end{definition}

Condition \ref{backdrop:density} implies that the pushout application $\leibpushapp{\Den{u}}$ is defined, by assumption that $\class{M}$ is a backdrop.

\begin{example}
  For any cocomplete category $\cat{E}$, diagram $u \co \cat{J} \to \Arr{\cat{E}}$, and limit ordinal $\kappa > 0$, $u$ is compatible with the $\kappa$-backdrop consisting of all arrows in $\cat{E}$.
\end{example}

We use the following proposition to construct uniform fibrations \textsc{awfs}'s in \cref{sec:uniform-fibrations}.

\begin{proposition}
  \label{union-backdrop-compatible}
  Let $\class{M}$ be a $\kappa$-backdrop in a category $\cat{E}$ and suppose $\class{M}$ is closed under binary unions.
  For $u \co \cat{J} \to \Arr{\cat{E}}$ to be compatible with $\class{M}$, it suffices (but is not necessary) that $\Den{u}$ exists, is valued in $\class{M}$, and sends squares in $\LevelClass{\cat{E}}{\class{M}}$ to cartesian squares in $\LevelClass{\cat{E}}{\class{M}}$.
\end{proposition}
\begin{proof}
  Under these hypotheses, given a square $(h,k) \co f \to g$ with $h,k$ in $\class{M}$, $\leibpushapp{\Den{u}}(h,k)$ computes a union of maps in $\class{M}$ and is thus in $\class{M}$ itself.
\end{proof}

\begin{example}
  The $\kappa$-backdrop of monomorphisms (\cref{monomorphisms-backdrop}) in an adhesive and $\kappa$-exhaustive category is closed under binary unions \cite[Theorem 5.1]{lack-sobocinski:05}, so \cref{union-backdrop-compatible} can be used in this case.
  The same goes for the $\omega$-backdrop of complemented monomorphisms (\cref{complemented-backdrop}) in a category with coproducts \cite[Lemma 2.1.7(i)]{gambino-sattler-szumilo:22}, and likewise for the $\omega$-backdrop of levelwise complemented monomorphisms in a functor category $\Func{\cat{C}}{\cat{E}}$ where $\cat{E}$ has coproducts.
\end{example}

\begin{definition}
  \label{small-object}
  Let $\cat{E}$ be a category with a $\kappa$-backdrop $\class{M}$.
  An object $A \in \cat{E}$ is \emph{$(\kappa,\class{M})$-small} when $\Hom{\cat{E}}{A}{-}$ preserves colimits of $\kappa$-chains in $\class{M}$.
\end{definition}

\begin{definition}
  \label{one-step-endofunctor}
  Let $\cat{E}$ be a category, $\kappa > 0$ be a limit ordinal, $\class{M}$ be a $\kappa$-backdrop in $\cat{E}$, and $u \co \cat{J} \to \Arr{\cat{E}}$ be compatible with $\class{M}$.
  We define a pointed endofunctor $\SplPtdEndo[u]_{\cat{E}} = (\SplEndo[u]_{\cat{E}},\splpt[u])$ on $\Arr{\cat{E}}$ by the pushout diagram
  \[
    \begin{tikzcd}[column sep=large]
      \Den{u} \ar{d}[left]{\epsilon} \ar{r}{\codpt\,\Den{u}} & \CodEndo_{\cat{E}} \Den{u} \ar[dashed]{d} \\
       \Id_{\Arr{\cat{E}}} \ar[dashed]{r}[below]{\splpt[u]} & \pushout \SplEndo[u]_{\cat{E}} \rlap{,}
    \end{tikzcd}
  \]
  where $\CodPtdEndo_{\cat{E}}$ is the pointed endofunctor of \cref{codomain-endofunctor} and $\epsilon$ is the counit of the density comonad.
  Note that this pushout exists by assumption that $u$ is compatible with $\class{M}$.
\end{definition}

Unpacked, $\SplEndo[u]_{\cat{E}}$ sends a morphism $f \co A \to B$ to the pushout gap map
\[
  \begin{tikzcd}[sep=large]
    \dom \Den{u}f \ar{d}[left]{\Den{u}f} \ar{r}{\dom \epsilon_f} & A \ar[dashed]{d} \ar[bend left=30]{ddr}{f} \\
    \cod \Den{u}f \ar[bend right=20]{drr}[below]{\cod \epsilon_f} \ar[dashed]{r} & \pushout \object \ar[dashed]{dr}[pos=0.4,description]{\SplEndo[u]_{\cat{E}}f} \\
    & & B \rlap{.}
  \end{tikzcd}
\]
of the counit of the density comonad, while the unit map $\splpt[u] \colon f \to \SplEndo[u]_{\cat{E}}f$ is the right triangle in the diagram above.

\begin{proposition}
  Let $\cat{E}$ be a category, $\kappa > 0$ be a limit ordinal, $\class{M}$ be a $\kappa$-backdrop in $\cat{E}$, and $u \co \cat{J} \to \Arr{\cat{E}}$ be compatible with $\class{M}$.
  Then we have an isomorphism $\rlp{\mathsf{}\cat{J}} \cong \Alg{\SplPtdEndo[u]_{\cat{E}}}$ over $\Arr{\cat{E}}$.
\end{proposition}
\begin{proof}
  A $\SplPtdEndo[u]_{\cat{E}}$-algebra structure $\SplPtdEndo[u]_{\cat{E}}f \to f$ on $f \colon A \to B$ is determined by a map $h \colon \cod \Den{u} f \to A$ fitting into the diagram
  \begin{equation}
    \label{split-is-lifting-diagram}
    \begin{tikzcd}[column sep=large]
      A \ar{d}[left]{f} \ar{r}{\inl} & A \sqcup_{\dom \Den{u} f} \cod \Den{u} f \ar{d}{\SplPtdEndo[u]_{\cat{E}}{f}} \ar[dashed]{r}{\copair{\id_A}{h}} & A \ar{d}{f} \\
      B \ar[equal]{r} & B \ar[equal]{r} & B \rlap{.}
    \end{tikzcd}
  \end{equation}
  Such a map consists, by universal property of the colimit defining $\cod \Den{u}$, of an assignment sending each $i \in \cat{J}$ and $\alpha \colon u_i \to f$ to a map $h\inj{\alpha} \co \cod u_i \to A$, coherently with respect to the morphisms of $\cat{J}$, and the diagram \eqref{split-is-lifting-diagram} requires precisely that the diagrams
  \[
    \begin{tikzcd}[column sep=large]
      \dom u_i \ar{d}[left]{u_i} \ar{r}{\dom \alpha} & A \ar{d}{f} \\
      \cod u_i \ar[dashed]{ur}[description]{h\inj{\alpha}} \ar{r}[below]{\cod \alpha} & B
    \end{tikzcd}
  \]
  commute, \ie, that each output is a solution to the input lifting problem.
  A morphism of $\SplPtdEndo[u]_{\cat{E}}$-algebras is similarly seen to correspond to a morphism of $\rlp{\cat{J}}$.
\end{proof}

\begin{lemma}
  \label{small-to-splendo}
  Let $\cat{E}$ be a cocomplete category with a $\kappa$-backdrop $\class{M}$ and let $u \co \cat{J} \to \Arr{\cat{E}}$ be a diagram compatible with $\class{M}$.
  If $\dom u \co \cat{J} \to \cat{E}$ is levelwise $(\kappa,\class{M})$-small for some $\kappa > 0$, then $\SplEndo[u]_{\cat{E}}$ preserves colimits of $\kappa$-chains in $\DomClass{\cat{E}}{\class{M}}$.
\end{lemma}
\begin{proof}
  We have $\Hom{\Arr{\cat{E}}}{u_i}{-} \cong \Hom{\cat{E}}{\dom u_i}{\dom (-)} \times_{\Hom{\cat{E}}{\dom u_i}{\cod (-)}} \Hom{\cat{E}}{\cod u_i}{\cod (-)}$.
  From the assumption that $\Hom{\cat{E}}{\dom u_i}{-}$ preserves colimits of $\kappa$-chains in $\class{M}$, it thus follows that $\Hom{\Arr{\cat{E}}}{u_i}{-}$ preserves colimits of $\kappa$-chains in $\DomClass{\cat{E}}{\class{M}}$.
  For any $\kappa$-chain $f \co (\kappa,\preceq) \to \DomClass{\cat{E}}{\class{M}}$, we then have $\colim_{\alpha < \kappa}(\Comma{u}{f_\alpha}) \cong \Comma{u}{(\colim_{\alpha < \kappa}f_\alpha)}$.
  The result then follows from the colimit formula for $\Den{u}$ (\cref{density-comonad}).
\end{proof}

\begin{theorem}[Algebraic small object argument with a backdrop]
  \label{soa-backdrop}
  \input{soa-backdrop}
\end{theorem}
\begin{proof}
  By Bourke and Garner's result (\cref{bourke-garner-recognition}), if the algebraically free monad on $\SplPtdEndo[u]_{\cat{E}}$ exists, then it is the monad of an $\textsc{awfs}$ cofibrantly generated by $u$.
  To construct the free monad on $\SplPtdEndo[u]_{\cat{E}}$, it suffices to check that $(\Arr{\mathsf{}\cat{E}},\DomClass{\mathsf{}\cat{E}}{\class{M}},\SplPtdEndo[u]_{\cat{E}}) \in \ConfP{\kappa}$:
  \begin{description}
  \item{\ref{pointed-configurations:backdrop}}
    $\DomClass{\cat{E}}{\class{M}}$ is a $\kappa$-backdrop in $\Arr{\cat{E}}$ because $\class{M}$ is a $\kappa$-backdrop in $\cat{E}$.
  \item{\ref{pointed-configurations:unit}}
    $\splpt[u]$ is valued in $\DomClass{\cat{E}}{\class{M}}$. This follows from \ref{backdrop:density} by the closure of $\DomClass{\cat{E}}{\class{M}}$ under cobase change in $\Arr{\cat{E}}$.
  \item{\ref{pointed-configurations:leibniz-application}}
    $\DomClass{\cat{E}}{\class{M}}$ is closed under $\leibpushapp{\splpt[u]}$.
    By cocontinuity of Leibniz pushout application in the natural transformation argument \cite[Lemma 4.8]{riehl-verity:14} and the definition of $\SplPtdEndo[u]_{\cat{E}}$, pushout application of $\splpt[u]$ is a cobase change of pushout application of $\codpt\,\Den{u}$.
  By calculation, pushout application of $\codpt\,\Den{u}$ is isomorphic to $\codpt\,\leibpushapp{\Den{u}}$, which preserves $\DomClass{\cat{E}}{\class{M}}$ by \ref{backdrop:density}, so the property follows from the closure of $\DomClass{\cat{E}}{\class{M}}$ under cobase change in $\Arr{\cat{E}}$.
  \item{\ref{pointed-configurations:convergence}}
    $\SplEndo[u]_{\cat{E}}$ preserves colimits of $\kappa$-chains in $\DomClass{\cat{E}}{\class{M}}$ by \cref{small-to-splendo}.
  \end{description}

  \Cref{pointed-free-monad} thus provides an algebraically free monad $\monad{R} = (R,\Lambda,\Pi)$ on $\SplPtdEndo[u]_{\cat{E}}$ whose unit is valued in $\DomClass{\cat{E}}{\class{M}}$, and by \cref{bourke-garner-recognition} we have an \textsc{awfs} $(\comonad{L},\monad{R})$.
\end{proof}

\begin{notation}
  \label{density-to-comonad-transformation}
  In the situation of \cref{soa-backdrop}, we have a transformation $\codpt\, \Den{u} \to \Lambda$ given by the composite
  \[
    \begin{tikzcd}
      \Den{u} \ar{d}[left]{\codpt\,\Den{u}} \ar{r}{\epsilon} & \Id_{\Arr{\cat{E}}} \ar{d}{\splpt[u]} \ar[equals]{r} & \Id_{\Arr{\cat{E}}} \ar{d}{\Lambda}  \\
      \CodEndo_{\cat{E}} \Den{u} \ar{r} & \pushout \SplEndo[u]_{\cat{E}} \ar{r} & R \rlap{.}
    \end{tikzcd}
  \]
  We write $\oneStep{u} \co \Den{u} \to L$ for its whiskering with the domain projection.
\end{notation}

\subsection{Notions of composable structure}
\label{sec:soa:notions}

At this point we begin working in earnest towards a saturation principle for cofibrantly generated \textsc{awfs}'s.
First, we abstract from the situation of $U_{\UnderPtd{\comonad{L}}} \co \DCoalg{\UnderPtd{\comonad{L}}} \to \Sq{\cat{E}}$ to define a class of \emph{composable notions of structure} on maps in $\cat{E}$.
We then define what it means for such a notion of structure to be cellular.

The double categories $\DCoalg{\UnderPtd{\comonad{L}}}$ and $\DCoalg{\comonad{L}}$ (and their duals $\DAlg{\UnderPtd{\monad{R}}}$ and $\DAlg{\monad{R}}$) are what Bourke and Garner call \emph{concrete double categories}, in which the vertical morphisms are horizontal morphisms equipped with some structure.
To suit our desired generality, we work with a slightly different abstraction which we call a \emph{notion of composable structure}, imposing some additional sanity conditions but allowing for non-thin and pseudo double categories.\footnote{%
  The name mimics Shulman's \emph{notions of fibred structure} \cite[\S3]{shulman:19-all}.
  Both axiomatize notions of structure on the morphisms of a category, but \emph{fibred} structures are closed under base change while \emph{composable} structures can be composed.
}
Permitting non-thin double categories is necessary for our application to constructing extension operations (\cref{sec:applications:extension}).

\begin{definition}
  \label{notion-of-composable-structure}
  A \emph{notion of composable structure} on a category $\cat{E}$ is a double functor $U \co \dcat{A} \to \Sq{\cat{E}}$ from a pseudo double category $\dcat{A}$ such that $U_0 \co \dcat{A}_0 \to \cat{E}$ is an identity (in particular $\dcat{A}_0 = \cat{E}$) and $\VertArr{U} \co \VertArr{\dcat{A}} \to \Arr{\cat{E}}$ is a conservative isofibration.
\end{definition}

\begin{notation}
  Following \citeauthor{bourke-garner:16}, when we write a vertical morphism in a notion of composable structure $U \co \dcat{A} \to \Sq{\cat{E}}$ in boldface, \eg $\bm{f} \co A \verto B$, we implicitly bind the same letter without boldface to its underlying morphism, \eg using $f \co A \to B$ for $U\bm{f}$.
\end{notation}

\begin{proposition}
  For any \textsc{awfs} $(\comonad{L},\monad{R})$ on a category $\cat{E}$, the projections $\DCoalg{\comonad{L}} \to \Sq{\cat{E}}$ and $\DCoalg{\UnderPtd{\comonad{L}}} \to \Sq{\cat{E}}$ are notions of composable structure.
\qed
\end{proposition}

\begin{definition}[{\cite[\S3.5]{bourke-garner:16} or \cite[\S2.5]{bourke-garner:16-ii}}]
  A pseudo double category $\dcat{A}$ is \emph{left-connected} when $\idv \co \dcat{A}_0 \to \VertArr{\dcat{A}}$ is left adjoint to $\vdom \co \VertArr{\dcat{A}} \to \dcat{A}_0$.
  In this case, the counit $\leftcnx \co \idv \circ \vdom \to \Id$ at $\bm{f} \co A \verto B$ defines a square
  \begin{equation}
    \label{left-connection}
    \begin{tikzcd}
      A \ar[phantomcenter]{dr}{\leftcnx} \ar[verto]{d}[left]{\idv} \ar[equals]{r} & A \ar[verto]{d}{\bm{f}} \\
      A \ar[dashed]{r} & B
    \end{tikzcd}
  \end{equation}
  which we call a \emph{left connection}.

  Any left-connected pseudo double category $\dcat{A}$ induces a double functor $\dcat{A} \to \Sq{\dcat{A}_0}$ which is the identity on horizontal morphisms and sends $\bm{f} \co A \verto B$ to the dashed morphism of \eqref{left-connection} \cite[\S3.5]{bourke-garner:16}.
  We say that a notion of composable structure $U \co \dcat{A} \to \Sq{\cat{E}}$ is \emph{left-connected} when $\dcat{A}$ is left-connected and $U$ is the double functor induced by the left connection.
\end{definition}

\begin{proposition}[{{see \cite[Theorem 6(ii)]{bourke-garner:16}}}]
  For any \textsc{awfs} $(\comonad{L},\monad{R})$ on $\cat{E}$, the notions of composable structure $U_{\comonad{L}} \co \DCoalg{\comonad{L}} \to \Sq{\cat{E}}$ and $U_{\UnderPtd{\comonad{L}}} \co \DCoalg{\UnderPtd{\comonad{L}}} \to \Sq{\cat{E}}$ are left-connected.
\qed
\end{proposition}

\begin{notation}
  When $U \co \dcat{A} \to \Sq{\cat{E}}$ is a notion of composable structure and $\class{M},\class{N}$ are wide subcategories of $\cat{E}$, write $\DClass{\dcat{A}}{\class{M}}{\class{N}}$ for the wide subcategory of $\VertArr{\dcat{A}}$ given by the pullback
  \[
    \begin{tikzcd}[column sep=large]
      \DClass{\dcat{A}}{\class{M}}{\class{N}} \pullback \ar{d} \ar[hook]{r} & \VertArr{\dcat{A}} \ar{d}{\VertArr{U}} \\
      \Class{\cat{E}}{\class{M}}{\class{N}} \ar[hook]{r} & \Arr{\cat{E}} \rlap{.}
    \end{tikzcd}
  \]
\end{notation}

\begin{definition}
  \label{cellular}
  Let $\class{M}$ be a $\kappa$-backdrop in a category $\cat{E}$.
  A notion of composable structure $U \co \dcat{A} \to \Sq{\cat{E}}$ is \emph{$(\kappa,\class{M})$-cellular} if
  \begin{enumerate}
  \item $(1 + \alpha)$-chains in $\DCodClass{\dcat{A}}{\class{M}}$ have colimits in $\VertArr{\dcat{A}}$ for $\alpha \preceq \kappa$ and $\VertArr{U}$ preserves these;
  \item $\DCodClass{\dcat{A}}{\class{M}}$ is closed under cobase change in $\VertArr{\dcat{A}}$ and $\VertArr{U}$ preserves these pushouts.
  \end{enumerate}
  Equivalently, $U$ is $(\kappa,\class{M})$-cellular when $\DCodClass{\dcat{A}}{\class{M}}$ is a $\kappa$-backdrop in $\VertArr{\dcat{A}}$ and $\VertArr{U}$ is a $\kappa$-backdrop-preserving functor $(\VertArr{\dcat{A}},\DCodClass{\dcat{A}}{\class{M}}) \longrightarrow (\Arr{\cat{E}},\CodClass{\cat{E}}{\class{M}})$.
\end{definition}

\begin{proposition}
  \label{coalgebra-double-category-cellular}
  For any \textsc{awfs} $(\comonad{L},\monad{R})$ on a category $\cat{E}$, the notions of composable structure $U_{\comonad{L}} \co \DCoalg{\comonad{L}} \to \Sq{\cat{E}}$ and $U_{\UnderPtd{\comonad{L}}} \co \DCoalg{\UnderPtd{\comonad{L}}} \to \Sq{\cat{E}}$ are $(\kappa,\class{M})$-cellular for every limit ordinal $\kappa > 0$ and $\kappa$-backdrop $\class{M}$.
\end{proposition}
\begin{proof}
  The forgetful functors associated to a comonad or copointed endofunctor create \emph{any} colimits existing in the base category.
  See, \eg, \textcite[Proposition 4.3.1]{borceux:94-2} for the (co)monad case; the (co)pointed endofunctor case goes by the same argument.
\end{proof}

\subsection{Structured points on free monads}
\label{sec:soa:points}

We want to extend an assignment of $\dcat{A}$-structures from the generators of a cofibrantly generated \textsc{awfs} $(\comonad{L},\monad{R})$ to the left factor $Lf$ of every map $f$.
In the algebraic small object argument, $L$ arises indirectly as the domain component of the unit $\Gh$ of the monad $\monad{R}$.
We abstract away the details of this case and consider the problem of extending an assignment of $\dcat{A}$-structures from the unit of a pointed endofunctor to the unit of its free monad.

We define ``$\dcat{A}$-structures on a unit'' in a somewhat indirect way that will be easiest to manipulate in our proofs; we unpack the intuition in \cref{vertical-point-unpack} below.

\begin{notation}
  Let $\cat{A} \overset{F}{\to} \cat{C} \overset{G}{\leftarrow} \cat{B}$ be a cospan of categories.
  Let $\ptdendo{T} = (T,\tau)$ be a pointed endofunctor on $\cat{A}$ such that $F$ defines a strict morphism of pointed endofunctors $(\cat{A},\ptdendo{T}) \to (\cat{C},\IdPtdEndo{\cat{C}})$.
  We write $\ptdendo{T} \times_{\cat{C}} \cat{B}$ for the pointed endofunctor on $\cat{A} \times_{\cat{C}} \cat{B}$ sending a pair $(A,B) \in \cat{A} \times_{\cat{C}} \cat{B}$ to $(TA,B)$, with evident unit.
  When $\monad{M}$ is a monad and $F$ defines a strict morphism of monads $(\cat{A},\monad{M}) \to (\cat{C},\IdMonad{\cat{C}})$, we write $\monad{M} \times_{\cat{C}} \cat{B}$ for the analogous monad.
\end{notation}

\begin{definition}
  Let $\ptdendo{T}$ be a pointed endofunctor on $\cat{E}$ and $U \co \dcat{A} \to \Sq{\cat{E}}$ be a notion of composable structure.
  An \emph{$\dcat{A}$-point $(\overline{\ptdendo{T}},\theta)$ over $\ptdendo{T}$} consists of a pointed endofunctor $\overline{\ptdendo{T}}$ on $\VertArr{\dcat{A}}$ for which $\vdom \co (\VertArr{\dcat{A}},\overline{\ptdendo{T}}) \to (\cat{E},\Id_{\cat{E}})$ and $\vcod \co (\VertArr{\dcat{A}},\overline{\ptdendo{T}}) \to (\cat{E},\ptdendo{T})$ are strict morphisms of pointed endofunctors, together with a natural isomorphism $\Gth$ making $({\vcirc},\theta) \co (\VertArr{\dcat{A}} \times_{\mathsf{}\cat{E}} \VertArr{\dcat{A}},\overline{\ptdendo{T}} \times_{\cat{E}} \VertArr{\dcat{A}}) \to (\VertArr{\dcat{A}},\overline{\ptdendo{T}})$ a strong morphism of pointed endofunctors.

  If $\Gg \co \ptdendo{T} \to \ptdendo{T'}$ is a morphism of pointed endofunctors on $\cat{E}$ and $(\overline{\ptdendo{T}},\theta)$ and $(\overline{\ptdendo{T}'},\theta')$ are $\dcat{A}$-points over $\ptdendo{T}$ and $\ptdendo{T'}$ respectively, then a \emph{morphism of $\dcat{A}$-points} $(\overline{\ptdendo{T}},\theta) \to (\overline{\ptdendo{T}'},\theta')$ over $\Gg$ is a morphism of pointed endofunctors $\overline{\Gg} \co \overline{\ptdendo{T}} \to \overline{\ptdendo{T'}}$ such that $\vdom \overline{\Gg} = \id$, $\vcod \overline{\Gg} = \Gg$, and
  \[
    \begin{tikzcd}[column sep=7em]
      \overline{T}(\projl{-}) \vcirc (\projr{-}) \ar{d}[left]{\theta} \ar{r}{\overline{\gamma}(\projl{-}) \vcirc (\projr{-})} & \overline{T'}(\projl{-}) \vcirc (\projr{-}) \ar{d}{\theta'} \\
      \overline{T}(\projl{-} \vcirc \projr{-}) \ar{r}[below]{\overline{\gamma}(\projl{-} \vcirc \projr{-})} & \overline{T'}(\projl{-} \vcirc \projr{-})
    \end{tikzcd}
  \]
  commutes.

  Mutatis mutandis, we have definitions of \emph{$\dcat{A}$-point $(\overline{\monad{M}},\theta)$ over a monad $\monad{M}$} and \emph{morphism of $\dcat{A}$-points $\overline{\Gg} \co (\overline{\monad{M}},\theta) \to (\overline{\monad{M'}},\theta')$ over a morphism of monads $\Gg \co \monad{M} \to \monad{M'}$}.
\end{definition}

\begin{remark}
  \label{vertical-point-unpack}
  Let $(T,\tau)$ be a pointed endofunctor and $((\overline{T},\overline{\tau}),\theta)$ be an $\dcat{A}$-point over $(T,\tau)$.
  Since $\theta$ tells us that $\overline{T}(\bm{f}) \cong \overline{T}(\idv_B \vcirc \bm{f}) \cong \overline{T}(\idv_B) \vcirc \bm{f}$ for $\bm{f} \co A \verto B$, we see that $\overline{T}$ is defined up to isomorphism by its restriction to vertical identities.
  For a given $A \in \cat{E}$, $\overline{\tau}_{\idv_A}$ is a square of the form
  \[
    \begin{tikzcd}
      A \ar{d}[left]{\idv_{A}} \ar[equals]{r} & A \ar{d}{\overline{T}(\idv_A)} \\
      A \ar{r}[below]{\tau_A} & TA \rlap{.}
    \end{tikzcd}
  \]
  When $U$ is left-connected, there is only one such square of this form, namely the left connection for $\overline{T}(\idv_A)$.
  Conversely, again when $U$ is left-connected, any functor $\bm{\tau} \co \cat{E} \to \VertArr{\dcat{A}}$ such that $\VertArr{U}\bm{\tau} = \tau$ determines an $\dcat{A}$-point $((\bm{\Gt}_!,\leftcnx_!),\alpha)$ over $\ptdendo{T}$ where $\bm{\Gt}_!(\bm{f}) = \bm{\Gt}_B \vcirc \bm{f}$, $\leftcnx_!$ sends $\bm{f}$ to the square
    \[
      \begin{tikzcd}
        A \ar[phantomcenter]{ddr}{\rho} \ar[verto]{dd}[left]{\bm{f}} \ar[equals]{r} & A \ar[phantomcenter]{dr}{\id_{\bm{f}}} \ar[verto]{d}[left]{\bm{f}} \ar[equals]{r} & A \ar[verto]{d}{\bm{f}} \\
        & B \ar[phantomcenter]{dr}{\leftcnx} \ar[verto]{d}[left]{\idv} \ar[equals]{r} & B \ar[verto]{d}{\bm{\Gt}_B} \\
        B \ar[equals]{r} & B \ar{r}[below]{\Gt_B} & TB
      \end{tikzcd}
    \]
    where $\rho$ is the vertical unitor for $\dcat{A}$, and $\alpha$ is the associator.
    Thus, we can think of an $\dcat{A}$-point over $(T,\tau)$ as a lift of $\tau$ through the vertical arrow category.
\end{remark}

We use \cref{pointed-free-monad}---and in particular fine-grained functoriality---to show that for \emph{cellular} notions of composable structure $U \co \dcat{A} \to \Sq{\cat{E}}$, an $\dcat{A}$-point over a pointed endofunctor induces a canonical $\dcat{A}$-point over its free monad.

\begin{theorem}
  \label{kelly-pointed-left-category}
  Let $(\cat{E},\class{M},\ptdendo{T}) \in \ConfP{\kappa}$ be a configuration for the free monad sequence on a pointed endofunctor.
  Let $U \co \dcat{A} \to \Sq{\cat{E}}$ be a $(\kappa,\class{M})$-cellular notion of composable structure.
  For any $\dcat{A}$-point $(\overline{\ptdendo{T}},\theta)$ over $\ptdendo{T}$, we have the following:
  \begin{enumerate}[label=(\roman*)]
  \item the free monad $\monad{M}$ on $\ptdendo{T}$ admits an $\dcat{A}$-point $(\overline{\monad{M}},\psi)$;
  \item the canonical morphism $\Gg \co \ptdendo{T} \to \UnderPtd{\monad{M}}$ lifts to a morphism of $\dcat{A}$-points $\overline{\Gg} \co (\ptdendo{T},\theta) \to (\UnderPtd{\overline{\monad{M}}},\psi)$;
  \item for any monad morphism $\Ga \co \monad{M} \to \monad{M'}$, $\dcat{A}$-point $(\overline{\monad{M'}},\psi')$ over $\monad{M'}$, and $\overline{\Gb} \co (\overline{\ptdendo{T}},\theta) \to (\UnderPtd{\overline{\monad{M'}}},\psi')$ over $\Ga\Gg \co \ptdendo{T} \to \UnderPtd{\monad{M}}$, there is a unique $\overline{\Ga} \co (\overline{\monad{M}},\psi) \to (\overline{\monad{M'}},\psi')$ over $\Ga$ such that $\overline{\Ga}\,\overline{\Gg} = \overline{\Gb}$.
  \end{enumerate}
\end{theorem}
\begin{proof}
  We check that $\overline{\ptdendo{T}} = (\overline{T},\overline{\tau})$ defines a configuration $(\VertArr{\dcat{A}},\DCodClass{\dcat{A}}{\class{M}},\overline{\ptdendo{T}}) \in \ConfP{\kappa}$.
  \begin{description}
  \item{\ref{pointed-configurations:backdrop}} $\DCodClass{\dcat{A}}{\class{M}}$ is a $\kappa$-backdrop in $\VertArr{\dcat{A}}$ by the assumption that $U$ is $(\kappa,\class{M})$-cellular.
  \item{\ref{pointed-configurations:unit}}
    $\overline{\tau}$ is valued in $\DCodClass{\dcat{A}}{\class{M}}$ because $\Gt$ is valued in $\class{M}$, by definition of $\ConfP{\kappa}$.
  \item{\ref{pointed-configurations:leibniz-application}}
    Given a square $(h,k) \co \bm{f} \to \bm{g}$ in $\DCodClass{\dcat{A}}{\class{M}}$, the pushout application $\leibpush{\overline{\Gt}}(h,k)$ is a square of the form
    \[
      \begin{tikzcd}[column sep=huge, row sep=large]
        \bullet \ar[phantomcenter]{dr}{\leibpush{\overline{\Gt}}(h,k)} \ar[verto]{d}[left]{\bm{g} \sqcup_{\bm{f}} \overline{T}{\bm{f}}} \ar[equals]{r} & \bullet \ar[verto]{d}{\overline{T}{\bm{g}}} \\
        \bullet \ar{r}[below]{\leibpushapp{\Gt}(h,k)} & \bullet \rlap{,}
      \end{tikzcd}
    \]
    so belongs to $\DCodClass{\dcat{A}}{\class{M}}$  by definition of $\ConfP{\kappa}$.
  \item{\ref{pointed-configurations:convergence}}
    The functor $\overline{T}$ preserves colimits of $\kappa$-chains in $\DCodClass{\dcat{A}}{\class{M}}$ because $T$ preserves colimits of $\kappa$-chains in $\class{M}$ (by definition of $\ConfP{\kappa}$) and $\VertArr{U}$ is conservative.
  \end{description}
  By definition, the projections $\vdom,\vcod \co \VertArr{\dcat{A}} \to \dcat{A}$ form a span
  \[
    \begin{tikzcd}[column sep=large]
      (\textsf{}\cat{E},\cong,\IdPtdEndo{\cat{E}}) & \ar{l}[above]{\vdom} (\VertArr{\dcat{A}},\DCodClass{\dcat{A}}{\class{M}},\ptdendo{\overline{T}}) \ar{r}{\vcod} & (\cat{E},\class{M},\ptdendo{T})
    \end{tikzcd}
  \]
  in $\ConfP{\kappa}$.
  By \cref{pointed-free-monad}, applying the free monad construction yields a span
  \begin{equation}
    \label{double-category-monad-span}
    \begin{tikzcd}[column sep=large]
      \IdMonad{\cat{E}} & \ar{l}[above]{\vdom} \overline{\monad{M}} \ar{r}{\vcod} & \monad{M}
    \end{tikzcd}
  \end{equation}
  in $\Mnd_s$ where $\monad{M}$ is the free monad on $\ptdendo{T}$ and $\overline{M}$ is a monad on $\VertArr{\dcat{A}}$.
  Because $\VertArr{U} \co \VertArr{\dcat{A}} \to \Arr{\cat{E}}$ is an isofibration, we can assume without loss of generality that the associated isomorphisms $\vdom \circ \overline{M} \cong \Id \circ \vdom$ and $\vcod \circ \overline{M} \cong M \circ \vcod$ are equalities.

  The argument above also shows that $(\VertArr{\dcat{A}} \times_{\mathsf{}\cat{E}} \VertArr{\dcat{A}}, \DCodClass{\dcat{A}}{\class{M}} \times_{\mathsf{}\cat{E}} \DCodClass{\dcat{A}}{\cong}, \overline{\ptdendo{T}} \times_{\cat{E}} \VertArr{\dcat{A}}) \in \ConfP{\kappa}$.
  We have morphisms of configurations
  \[
    \begin{tikzcd}[row sep=large]
      & (\VertArr{\dcat{A}} \times_{\mathsf{}\cat{E}} \VertArr{\dcat{A}}, \DCodClass{\dcat{A}}{\class{M}} \times_{\mathsf{}\cat{E}} \DCodClass{\dcat{A}}{\cong}, \overline{\ptdendo{T}} \times_{\cat{E}} \VertArr{\dcat{A}}) \ar{dl}{\projl} \ar{d}[left]{(\vcirc,\theta)} \ar{dr}[below left]{\projr} \\
      (\VertArr{\dcat{A}},\DCodClass{\dcat{A}}{\class{M}},\overline{\ptdendo{T}}) & (\VertArr{\dcat{A}},\DCodClass{\dcat{A}}{\class{M}},\overline{\ptdendo{T}}) & (\VertArr{\dcat{A}},\DCodClass{\dcat{A}}{\cong},\Id)
    \end{tikzcd}
  \]
  where the two projections are strict morphisms.
  Note that $\vcirc$ must preserve certain colimits to be a morphism of configurations; this follows from conservativity of $\VertArr{U}$.
  By \cref{pointed-free-monad}, the morphisms of configurations induce corresponding strong morphisms of free monads.
  The strong morphisms associated to the projections force the free monad on $\overline{\ptdendo{T}} \times_{\cat{E}} \VertArr{\dcat{A}}$ to be $\overline{\monad{M}} \times_{\cat{E}} \VertArr{\dcat{A}}$, which makes the strong morphism associated to $(\Gth,\vcirc)$ the desired $({\vcirc},\psi) \co (\VertArr{\dcat{A}} \times_{\mathsf{}\cat{E}} \VertArr{\dcat{A}},\overline{\monad{M}} \times_{\cat{E}} \VertArr{\dcat{A}}) \to (\VertArr{\dcat{A}},\overline{\monad{M}})$.
  The universal property of $(\overline{\monad{M}},\psi)$ described in the theorem statement follows straightforwardly from the universal property of $\overline{\monad{M}}$ as the free monad on $\overline{\ptdendo{T}}$.
\end{proof}

Although our target application of \cref{kelly-pointed-left-category} is to the algebraic small object argument, it also can be used to derive properties of free monads more generally.

\begin{example}
  \label{subobject-vertically-pointed-monad-example}
  Let $\cat{E}$ be an elementary topos and let $\ptdendo{T} = (T,\Gt)$ be a pointed endofunctor on $\cat{E}$ that preserves colimits of $\kappa$-chains for some limit ordinal $\kappa > 0$.
  If $\Gt$ is a cartesian natural transformation valued in monomorphisms, then the unit of the free monad on $\ptdendo{T}$ (which exists by \cref{pointed-free-monad}) is also cartesian and valued in monomorphisms.
\end{example}
\begin{proof}
  By \cref{kelly-pointed-left-category} with $\class{M} \defeq \cat{E}$ and the notion of composable structure $U \co \dcat{A} \to \Sq{\cat{E}}$ whose vertical morphisms are monomorphisms in $\cat{E}$ and whose squares are pullback squares in $\cat{E}$.
  This notion of composable structure is left-connected, so per \cref{vertical-point-unpack} an $\dcat{A}$-point over $\ptdendo{T}$ consists of a lift of $\tau$ through $\VertArr{U}$; the assumptions on $\tau$ say exactly that such a lift exists.
  The existence of a subobject classifier and pullback-stable colimits implies that $\DCodClass{\dcat{A}}{\cat{E}}$ is closed under all colimits.
  Thus the free monad on $\ptdendo{T}$ also admits an $\dcat{A}$-point, which implies in particular that its unit also lifts through $\VertArr{U}$.
\end{proof}

\subsection{The algebraic small object argument as a saturation}
\label{sec:soa:algebraic-saturation}

For our saturation principles, we start with a category $\cat{E}$ with an \textsc{awfs} cofibrantly generated by some $u \co \cat{J} \to \Arr{\cat{E}}$ and are given a notion of composable structure $U \co \dcat{A} \to \Sq{\cat{E}}$.
In order to apply \cref{kelly-pointed-left-category} with the pointed endofunctor $\SplPtdEndo[u]_{\cat{E}}$ on $\Arr{\cat{E}}$, we want to derive a notion of composable structure on the arrow category $\Arr{\cat{E}}$.
For this, we use a \emph{horizontal comma pseudo double category} construction \cite[\S2.5]{grandis-pare:04}.
Here we only need a comma of the form $\DComma{F}{\dcat{B}}$ where $F \co \dcat{A} \to \dcat{B}$.

\begin{definition}
  \label{gluing-double-category}
  Given a pseudo double functor $F \co \dcat{A} \to \dcat{B}$ between pseudo double categories, the \emph{horizontal comma pseudo double category} $\DComma{F}{\dcat{B}}$ is a pseudo double category where
  \begin{enumerate}
  \item an object is a triple $(A,B,f)$ consisting of  where $A \in \dcat{A}$, $B \in \dcat{B}$, and $f \co FA \to B$;
  \item a horizontal morphism $(a,b) \co (A,B,f) \to (A',B',f')$ is a pair of a horizontal morphism $a \co A \to A'$ and a horizontal morphism in $b \co B \to B'$ fitting in the diagram
    \[
      \begin{tikzcd}[row sep=tiny, column sep=small]
        A \ar{rrr}{a} & & &[-1em] A' \\
        & FA \ar{dr}[below left]{f} \ar{rrr}{Fa} & & & FA' \ar{dr}{f'} \\
        & & B \ar{rrr}[below]{b} & & & B' \rlap{,}
      \end{tikzcd}
    \]
    \ie, a morphism of the comma category $\Comma{F_0}{\dcat{B}_0}$;
  \item a vertical morphism $(\bm{a},\bm{b},\gamma) \co (A,B,f) \verto (A',B',f')$ is a triple of a vertical morphism $\bm{a} \co A \verto A'$, vertical morphism $\bm{b} \co B \verto B'$, and square $\gamma$ of $\dcat{B}$ fitting in the diagram
    \[
      \begin{tikzcd}[row sep=tiny, column sep=small]
        A \ar[verto]{ddd}[left]{\bm{a}} \\
        & FA \ar[phantom]{ddddr}{\gamma} \ar[verto]{ddd}[left]{F\bm{a}} \ar{dr}{f} \\
        & & B \ar[verto]{ddd}{\bm{b}}\\
        A' \\
        & FA' \ar{dr}[below left]{f'} \\
        & & B' \rlap{,}
      \end{tikzcd}
    \]
    \ie, an object of the comma category $\Comma{F_1}{\dcat{B}_1}$;
  \item a square is a morphism of the comma category $\Comma{F_1}{\dcat{B}_1}$, \ie, a pair of a square $\alpha$ in $\dcat{A}$ and square $\beta$ in $\dcat{B}$ fitting into the evident commutative diagram in $\dcat{B}_1$.
  \end{enumerate}
  We refer to \textcite[\S2.5]{grandis-pare:04} for the remaining details.
\end{definition}

\begin{definition}
  Let $U \co \dcat{A} \to \Sq{\cat{E}}$ be a notion of composable structure on $\cat{E}$.
  We write $\glueconc{U} \co \GlueDbl{U} \to \Sq{\Arr{\cat{E}}}$ for the \emph{glued} notion of composable structure with $\GlueDbl{U} \defeq \DComma{U}{\Sq{\cat{E}}}$ and $\glueconc{U}$ defined on vertical morphisms by $\glueconc{U}(\bm{a},b,\gamma) \defeq (a,b)$.
  We write $\gluedom \co \GlueDbl{U} \to \dcat{A}$ for the domain projection.
\end{definition}

\begin{proposition}
  \label{gluing-left-connected}
  Let $U \co \dcat{A} \to \Sq{\cat{E}}$ be a notion of composable structure. If $U$ is left-connected, then $\glueconc{U} \co \GlueDbl{U} \to \Sq{\Arr{\cat{E}}}$ is also left-connected.
\qed
\end{proposition}

\begin{lemma}
  \label{gap-cobase-to-opfibration}
  Let $U \co \dcat{A} \to \Sq{\cat{E}}$ be a left-connected notion of composable structure.
  Let $\class{M}$ be a wide subcategory of $\cat{E}$ such that $\DCodClass{\dcat{A}}{\class{M}}$ is closed under cobase change in $\VertArr{\dcat{A}}$ and these pushouts are preserved by $\VertArr{U}$.
  Then given $\bm{f} \co A \verto B$ in $\class{M}$ and $g \co A \to C$, we have a square
  \[
    \begin{tikzcd}
      A \ar[phantomcenter]{dr}{\beta} \ar[verto]{d}[left]{\bm{f}} \ar{r}{g} & C \ar[dashed,verto]{d} \\
      B \ar[dashed]{r} & D
    \end{tikzcd}
  \]
  which is an opcartesian lift of $\bm{f}$ along $g$ in $\vdom \co \VertArr{\dcat{A}} \to \cat{E}$ and is sent to a pushout square by $U$.
\end{lemma}
\begin{proof}
  Given $\bm{f} \co A \verto B$ in $\class{M}$ and $g \co A \to C$, the left connection $\idv_A \to \bm{f}$ belongs to $\DCodClass{\dcat{A}}{\class{M}}$, so we can form its pushout along $\idv_g \co \idv_A \to \idv_C$, yielding a cube of the form
  \[
    \begin{tikzcd}[column sep={4.5em,between origins},row sep=2.4em]
      A \ar[drrr, phantom, "\ulcorner"{pos=0.85,xslant=-1.8,xscale=2.5}] \ar[equals]{dr} \ar[verto]{dd}[left]{\idv_A} \ar{rr}{g} && C \ar[verto={pos=0.7}]{dd}[pos=0.7]{\idv_C} \ar[dashed,equals]{dr} & \\[-1.3em]
      & A \ar[crossing over,dashed]{rr} && C \ar[verto,dashed]{dd}{\bm{h}} \\
      A \ar[drrr, phantom, "\ulcorner"{pos=0.85,xslant=-1.8,xscale=2.5}] \ar{dr}[below left]{f} \ar{rr}[near end]{g} && C \ar[dashed]{dr} \\[-1.3em]
      & B \ar[from=uu,crossing over,verto={pos=0.4},"\bm{f}"{pos=0.4,left}] \ar[dashed]{rr} && D \rlap{.}
    \end{tikzcd}
  \]
  The front face is our opcartesian lift $\beta$.
  Its universal property follows straightforwardly from the universal property of the pushout: for any $t \co C \to C'$ and square $\gamma \co \bm{f} \to \bm{h'}$ under $tg$, we get a cocone
  \[
    \begin{tikzcd}[column sep=large]
      A \ar[verto]{d}[left]{\bm{f}} \ar[phantomcenter]{dr}{\gamma} \ar{r}[above]{tg} & |[yshift=-0.8em]| C'\ar[verto]{d}{\bm{h'}} & |[yshift=-0.4em]| \ar[phantomcenter]{dl}{\leftcnx*} \ar[equals]{l} C' \ar[verto]{d} & C \ar[phantomcenter]{dl}{\idv_t} \ar{l}[above]{t} \ar{d}{\idv_C} \\
      B \ar{r} & |[yshift=-0.8em]| D' & |[yshift=-0.4em]| \ar{l}[below]{h'} C' & \ar{l}[below]{t} C
    \end{tikzcd}
  \]
  under $\bm{f} \leftarrow \idv_A \to \idv_C$ which then induces a factorization of $\gamma$ through $\alpha$.
  That the square is a pushout in $\cat{E}$ follows from pushout pasting, as the back face of the cube is also a pushout in $\cat{E}$.
\end{proof}

We now prove our first main theorem, which exhibits the left-coalgebras of a cofibrantly generated \textsc{awfs} as the least cellular notion of composable structure containing the generators.
Although the statement of the following theorem makes clear that its output is unique with some properties, the useful data of that output is buried in the notion of $\GlueDbl{U}$-point.
We extract this data in \cref{soa-unit-vertical-point}.

\begin{theorem}
  \label{soa-unit-vertical-point-abstract}
  Fix a category $\cat{E}$ equipped with a $\kappa$-backdrop $\class{M}$, a diagram $u \co \cat{J} \to \Arr{\cat{E}}$ compatible with $\class{M}$ such that $\dom u$ is levelwise $(\kappa,\class{M})$-small, and a left-connected, $(\kappa,\class{M})$-cellular notion of composable structure $U \co \dcat{A} \to \Sq{\cat{E}}$.
  If $(\comonad{L},\monad{R})$ is the \textsc{awfs} cofibrantly generated by $u$, then any lift $\bm{D}_{\dcat{A}} \co \Arr{\cat{E}} \to \VertArr{\dcat{A}}$ of $\Den{u}$ through $\VertArr{U}$ induces an essentially unique $\GlueDbl{U}$-point $(\overline{\monad{R}},\theta)$ over $\monad{R}$ with a natural transformation $\Gg \co \bm{D}_{\dcat{A}} \to \gluedom \overline{R}\idv$ such that $\VertArr{U}\Gg = \oneStep{u}$.
\end{theorem}
\begin{proof}
  The \textsc{awfs} cofibrantly generated by $u$ exists by \cref{soa-backdrop}.
  Recall from the proof of \cref{soa-backdrop} that $\monad{R}$ is the free monad on $\SplPtdEndo[u]_{\cat{E}}$.
  We aim to apply \cref{kelly-pointed-left-category} with the configuration $(\Arr{\cat{E}},\DomClass{\cat{E}}{\class{M}},\SplPtdEndo[u]_{\cat{E}}) \in \ConfP{\kappa}$ established in the proof of \cref{soa-backdrop} to obtain a $\GlueDbl{U}$-point over $\monad{R}$.
  Observe that $\glueconc{U} \co \GlueDbl{U} \to \Sq{\Arr{\cat{E}}}$ is $(\kappa,\DomClass{\cat{E}}{\class{M}})$-cellular, as required, by the assumption that $U$ is $(\kappa,\class{M})$-cellular.

  It remains to construct a $\GlueDbl{U}$-point over $\SplPtdEndo[u]_{\cat{E}}$.
  Recall that $\splpt[u]$ is the cobase change of $\codpt\,\Den{u} \co \Den{u} \to \CodEndo_{\cat{E}} \Den{u}$ along the counit $\Ge \co \Den{u} \to \Id_{\Arr{\cat{E}}}$.
  By definition of $\GlueDbl{U}$, $\bm{D}_{\dcat{A}}$ induces a lift of $\codpt\,\Den{u}$ through $\glueconc{U}$:
  \[
    \begin{tikzcd}
      \Arr{\cat{E}} \ar[bend right]{ddr}[below left]{\bm{D}_{\dcat{A}}} \ar[bend left=25]{drr}[pos=0.7]{\codpt\,\Den{u}} \ar[dashed]{dr}[pos=.7]{\bm{D}_{\gluelabel}} \\
      &[-2em] \VertArr{\GlueDbl{U}} \pullback \ar{d} \ar{r} & \Arr{(\Arr{\cat{E}})} \ar{d}{\dom} \\
      & \VertArr{\dcat{A}} \ar{r}[below]{\VertArr{U}} & \Arr{\cat{E}} \rlap{.}
    \end{tikzcd}
  \]
  By \cref{gap-cobase-to-opfibration}, we have opcartesian lifts
  \begin{equation}
    \label{soa-unit-vertical-point:opcartesian-lift}
    \begin{tikzcd}
      \Den{u} \ar[phantomcenter]{dr}{\beta} \ar[verto]{d}[left]{\bm{D}_{\gluelabel}} \ar{r}{\Ge} & \Id_{\Arr{\cat{E}}} \ar[dashed,verto]{d}{\bm{\tau}} \\
      \CodEndo_{\cat{E}} \Den{u} \ar[dashed]{r} & \SplEndo[u]_{\cat{E}} \rlap{,}
    \end{tikzcd}
  \end{equation}
  living over pushout squares in $\Arr{\cat{E}}$, which define a lift $\bm{\tau} \co \Arr{\cat{E}} \to \VertArr{\GlueDbl{U}}$ of $\splpt[u]$ through $\VertArr{{\glueconc{U}}}$.
  Since $\GlueDbl{U}$ is left-connected, such a lift determines a $\GlueDbl{U}$-point $((\bm{\tau}_!,\leftcnx_!),\alpha)$ over $\SplPtdEndo[u]_{\cat{E}}$ as described in \cref{vertical-point-unpack}.

  \Cref{kelly-pointed-left-category} now gives a $\GlueDbl{U}$-point $(\overline{\monad{R}},\theta)$ over $\monad{R}$ and morphism $\overline{\gamma} \co ((\bm{\tau}_!,\leftcnx_!),\alpha) \to (\overline{\monad{R}},\theta)$ over the canonical map $\SplPtdEndo[u]_{\cat{E}} \to \UnderPtd{\monad{R}}$.
  We obtain $\gamma_f \co \bm{D}_{\dcat{A}}(f) \to \gluedom \overline{R}(\idv_f)$ for $f \co A \to B$ as the composite
  \[
    \begin{tikzcd}
      \bm{D}_{\dcat{A}}(f) = \gluedom \bm{D}_{\gluelabel}(f) \ar{r}{\gluedom \beta_{f}} &[2em] \gluedom \bm{\tau}_f \ar{r}{\cong} &\gluedom (\bm{\tau}_f \vcirc \idv_f) \ar{r}{\gluedom \overline{\Gg}_{\idv_f}} &[3em] \gluedom \overline{R}(\idv_{f})\rlap{.}
    \end{tikzcd}
  \]

  To see essential uniqueness, suppose we have a $\GlueDbl{U}$-point $(\overline{\monad{R}}',\theta')$ over $\monad{R}$ and a natural transformation $\Gg' \co \bm{D}_{\dcat{A}} \to \gluedom \overline{R}'\idv$ such that $\VertArr{U}\Gg' = \oneStep{u}$.
  The latter extends to some $\bm{D}_{\gluelabel} \to \overline{R}'\idv$ which by the opcartesian property of $\beta$ factors through a transformation $\bm{\tau} \to \overline{R}'\idv$.
  The latter transformation extends to a morphism of $\GlueDbl{U}$-points $((\bm{\tau}_!,\leftcnx_!),\alpha) \to (\overline{\monad{R}}',\theta')$ over $\SplPtdEndo[u]_{\cat{E}} \to \UnderPtd{\monad{R}}$.
  Finally, the universal property of $(\overline{\monad{R}},\theta)$ implies that there is a morphism $(\overline{\monad{R}},\theta)\to (\overline{\monad{R}}',\theta')$ over the identity on $\monad{R}$.
  Because $\VertArr{U}$ is conversative, it must be an isomorphism.
\end{proof}

As a corollary:

\begin{theorem}
  \label{soa-unit-vertical-point}
  \input{soa-unit-vertical-point}
\end{theorem}
\begin{proof}
  By \cref{soa-unit-vertical-point-abstract}, we have a $\GlueDbl{U}$-point $(\overline{\monad{R}},\theta)$ over $\monad{R}$ and some $\Gg \co \bm{D}_{\dcat{A}} \to \gluedom \overline{R}\idv$ such that $\VertArr{U}\Gg = \oneStep{u}$.
  Recall that $L \co \Arr{\cat{E}} \to \Arr{\cat{E}}$ is the domain component of the unit of $\monad{R}$; we may thus define $\bm{L}_{\dcat{A}} \defeq \gluedom \overline{R}\idv$, in which case $\gamma$ is the desired transformation $\bm{D}_{\dcat{A}} \to \bm{L}_{\dcat{A}}$.
\end{proof}

\begin{proposition}
  \label{soa-unit-vertical-point-multiplication}
  In the situation of \cref{soa-unit-vertical-point}, we have some $\overline{\Gm} \co \bm{L}_{\dcat{A}}R \vcirc \bm{L}_{\dcat{A}} \to \bm{L}_{\dcat{A}}$ lifting $(\id,\Gm) \co LR \vcirc L \to L$.
\end{proposition}
\begin{proof}
  Again using the output of \cref{soa-unit-vertical-point-abstract}, we extract $\overline{\Gm}$ from the multiplication for $\overline{\monad{R}}$ like so:
  \[
    \begin{tikzcd}[column sep={2em}]
     \bm{L}_{\dcat{A}}Rf \vcirc \bm{L}_{\dcat{A}}f = \gluedom (\overline{R}(\idv_{Rf}) \vcirc \overline{R}(\idv_f)) \ar{r}[below]{\cong}{\theta} & \gluedom \overline{R}(\overline{R}(\idv_f)) \ar{r} & \gluedom \overline{R}(\idv_f) = \bm{L}_{\dcat{A}}f \rlap{.}
    \end{tikzcd} \qedhere
  \]
\end{proof}

In the remainder of this section, we consider the problem of extending structure from the \emph{free} left maps to \emph{all} left maps, which is to say all $\UnderPtd{\comonad{L}}$-coalgebras.
The traditional story is that any left map of a \textsc{wfs} is a (codomain) retract of the left factor of its factorization, so any cellular class of maps closed under retracts which contains the generators of a cofibrantly generated \textsc{wfs} will contain every left map.
The same logic applies here: if we are in the situation of \cref{soa-unit-vertical-point} and know that every retract of an $\dcat{A}$-structured map admits an $\dcat{A}$-structure, then we can assign an $\dcat{A}$-structure to every $\UnderPtd{\comonad{L}}$-coalgebra.
In the algebraic situation, however, it is natural to investigate when we can get a functorial or pseudo double functorial assignment of $\dcat{A}$-structures to $\UnderPtd{\comonad{L}}$-coalgebras.

\begin{definition}
  \label{retract-category}
  The \emph{walking retract} is the category generated by the graph
  \[
    \begin{tikzcd}
      \mathsf{0} \ar{r}{\mathsf{s}} & \mathsf{1} \ar{r}{\mathsf{r}} & \mathsf{0}
    \end{tikzcd}
  \]
  and equation $\mathsf{r}\mathsf{s} = \id$.
  Given a category $\cat{E}$, its \emph{category of retracts} $\Retract(\cat{E})$ is the category of functors from the walking retract to $\cat{E}$.
  Write $\projret,\projobj \co \Retract(\cat{E}) \to \cat{E}$ for evaluation at $\mathsf{0}$ and $\mathsf{1}$ respectively and $\Delta \co \cat{E} \to \Retract(\cat{E})$ for the functor sending $A \in \cat{E}$ to the identity retract $A \overset{=}\to A \overset{=}\to A$.
\end{definition}

\begin{definition}
  \label{codomain-retract}
  Given a category $\cat{E}$, its \emph{category of codomain retracts} $\CodRet(\cat{E})$ is the pullback
  \[
    \begin{tikzcd}
      \CodRet(\cat{E}) \pullback \ar[dashed]{d} \ar[dashed]{r} & \Retract(\Arr{\cat{E}}) \ar{d}{\Retract(\dom)} \\
      \cat{E} \ar{r}[below]{\Delta} & \Retract(\cat{E}) \rlap{,}
    \end{tikzcd}
  \]
  \ie, the subcategory of $\Retract(\Arr{\cat{E}})$ consisting of retract diagrams $g \overset{\Ga}{\to} f \overset{\Gb}{\to} g$ in $\Arr{\cat{E}}$ such that $\dom \Ga = \vdom \Gb = \id$.
\end{definition}

\begin{definition}
  \label{lifts-codomain-retracts}
  For a notion of composable structure $U \co \dcat{A} \to \Sq{\cat{E}}$, a \emph{codomain retract lifting operator} is a functor $\VertArr{\dcat{A}} \times_{\Arr{\cat{E}}} \CodRet(\cat{E}) \to \VertArr{\dcat{A}}$, where the domain is the pullback of $\VertArr{U} \co \VertArr{\dcat{A}} \to \Arr{\cat{E}}$ and $\projobj \co \CodRet(\cat{E}) \to \Arr{\cat{E}}$, fitting into the diagram
  \[
    \begin{tikzcd}
      \VertArr{\dcat{A}} \times_{\Arr{\cat{E}}} \CodRet(\cat{E}) \ar{d}[left]{\projr} \ar[dashed]{r} & \VertArr{\dcat{A}} \ar{d}{U} \\
      \CodRet(\cat{E}) \ar{r}[below]{\projret} & \Arr{\cat{E}} \rlap{.}
    \end{tikzcd}
  \]
\end{definition}

\begin{example}
  \label{coalgebra-retract-lifting}
  Given an \textsc{awfs} $(\comonad{L},\monad{R})$ on a category $\cat{E}$, $U_{\UnderPtd{\comonad{L}}} \co \Coalg{\UnderPtd{\comonad{L}}} \to \Sq{\cat{E}}$ admits a codomain retract lifting operator.
\end{example}
\begin{proof}
  Recall that a vertical morphism over $f \in \Arr{\cat{E}}$ in $\Coalg{\UnderPtd{\comonad{L}}}$ is a section $\sigma \co f \to Lf$ of $\Phi_f$.
  Given a retract $g \overset{\alpha}\to f \overset{\beta}\to g$ in $\Arr{\cat{E}}$, it follows from naturality of $\Phi$ that $L\beta \circ \sigma \circ \alpha \co g \to Lg$ is a section of $\Phi_g$, and this assignment is evidently functorial.
\end{proof}

\begin{remark}
  Note that a codomain retract lifting operator only lifts the retract \emph{object}, not entire retract diagrams.
  In \cref{coalgebra-retract-lifting}, for instance, we cannot expect that $(\alpha,\beta)$ exhibits $L\beta \circ \sigma \circ \alpha \co g \to Lg$ as a retract of $\sigma \co f \to Lf$ in $\Coalg{\UnderPtd{\comonad{L}}}$.
\end{remark}

\begin{theorem}
  \label{soa-copointed-coalgebras-functor}
  \input{soa-copointed-coalgebras-functor}
\end{theorem}
\begin{proof}
  Under the assumptions above, we can apply \cref{soa-unit-vertical-point} to obtain lifts $\bm{L}_{\dcat{A}} \co \Arr{\cat{E}} \to \VertArr{\dcat{A}}$ of $L \co \Arr{\cat{E}} \to \Arr{\cat{E}}$ through $\VertArr{U}$ such that $\oneStep{u} \co \Den{u} \to L$ lifts to a transformation $\bm{D}_{\dcat{A}} \to \bm{L}_{\dcat{A}}$.
  For any object $\bm{f} = (f,s) \in \Coalg{\UnderPtd{\comonad{L}}}$, we have by definition codomain retract diagram
  \begin{equation}
    \label{canonical-coalgebra-retract}
    \begin{tikzcd}[sep=large]
      f \ar{r}{(\id,s)} & Lf \ar{r}{\Phi_f} & f
    \end{tikzcd}
  \end{equation}
  By assumption, we can lift $\bm{L}_{\dcat{A}}f \in \dcat{A}$ along this retract to obtain $j\bm{f} \in \VertArr{\dcat{A}}$  with $\VertArr{U}(j\bm{f}) = f$.
  Similarly, any morphism $(h,k) \co \bm{f} \to \bm{g}$ in $\Coalg{\UnderPtd{\comonad{L}}}$ induces by definition a morphism of retract diagrams as above and a morphism $\bm{L}_{\dcat{A}}(h,k) \co \bm{L}_{\dcat{A}}f \to \bm{L}_{\dcat{A}}g$, whence a choice of lift $j(h,k) \co j\bm{f} \to j\bm{g}$ with $U(j(h,k)) = (h,k)$.
\end{proof}

To get a (pseudo) \emph{double} functor from $\DCoalg{\UnderPtd{\comonad{L}}}$, we require an additional compositionality condition on the retract lifting operator.
For the sake of simplicity, we now only consider the case where $\dcat{A}$ is thin.

\begin{definition}
  \label{compositional-retract-lifting}
  A codomain retract lifting operator on a thin notion of composable structure $U \co \dcat{A} \to \Sq{\cat{E}}$ is \emph{compositional} if for all diagrams
  \begin{equation}
    \label{compositional-retract-lifting-input}
    \begin{tikzcd}
      & A \ar{dl}[above left]{f'} \ar[verto]{d}{\bm{f}} \ar{dr}{f'} \\
      B' \ar[bend right=20,equals]{rr} \ar{r}{s_B} & B \ar{r}{r_B} & B'
    \end{tikzcd}
    \qquad
    \begin{tikzcd}
      B' \ar[verto]{d}{\bm{g'}} \ar{r}{s_B} & B \ar[verto]{d}{\bm{g}} \ar{r}{r_B} & B' \ar[verto]{d}[left]{\bm{g'}} \\
      C' \ar{r}{u} & C \ar{r}{v} & C'
    \end{tikzcd}
    \qquad
    \begin{tikzcd}
      & B' \ar{dl}[above left]{g''} \ar[verto]{d}{\bm{g'}} \ar{dr}{g''} \\
      C'' \ar[bend right=20,equals]{rr} \ar{r}{s_C} & C' \ar{r}{r_C} & C''
    \end{tikzcd}
  \end{equation}
  such that $r_Cvu = r_C$, the chosen lift of $\bm{g} \vcirc \bm{f}$ along the composite codomain retract $g''f' \to gf \to g''f'$ is isomorphic over $g''f'$ to the vertical composite of the chosen lift of $\bm{g'}$ along $g'' \to g' \to g''$ with the chosen lift of $\bm{f}$ along $f' \to f \to f'$.
\end{definition}

\begin{example}
  \label{coalgebra-retract-lifting-compositional}
  Given an \textsc{awfs} $(\comonad{L},\monad{R})$ on a category $\cat{E}$, the codomain retract lifting operator on $U_{\UnderPtd{\comonad{L}}} \co \Coalg{\UnderPtd{\comonad{L}}} \to \Sq{\cat{E}}$ described in \cref{coalgebra-retract-lifting} is compositional.
\end{example}
\begin{proof}
  Suppose we have diagrams as in \eqref{compositional-retract-lifting-input}. Write $\bm{f} = (f,s) \co A \verto B$,  $\bm{g} = (g,t) \co B \verto C$, and $\bm{g'} = (g',t') \co B' \verto C'$ for the coalgebras.
  Because $(s_B,u) \co \bm{g'} \to \bm{g}$ is a morphism of coalgebras, we have in particular $t \circ u = E(s_B,u) \circ t'$.

  Recall that the coalgebra structure on $(g,t) \vcirc (f,s)$ is given by the composite
  \begin{equation}
    \label{compositional-retract-composite}
    \begin{tikzcd}
      C \ar{r}[below]{t} & Eg \ar{r}[below]{E(s,\id_C)} &[2em] E(g \circ Rf) \ar{r}[below]{E(E(\id_A,g),\id_C)} &[5em] ER(gf) \ar{r}[below]{\mu_{gf}} & E(gf) \rlap{.}
    \end{tikzcd}
  \end{equation}
  Lifting \eqref{compositional-retract-composite} along the retract $g''f' \to gf \to g''f'$ consists in pre-composing with $s_Cu \co C'' \to C$ and post-composing with $E(\id_A,r_Cv) \co E(gf) \to E(g''f')$.
  Writing $s' = E(\id_A,r_B) \circ s \circ s_B$ for the lifted coalgebra structure on $f'$, and $t'' = E(\id_B,r_C) \circ t' \circ s_C$ for the lifted structure on $g''$, we see from the diagram
  \[
    \begin{tikzcd}[row sep=large]
      C'' \ar[bend right=20]{dddr}[below left]{t''} \ar{r}{s_C} & C' \ar{d}{t'} \ar{r}{u} &[2em] C \ar{d}{t} \\
      & Eg' \ar{dd}{E(\id_{B'},r_C)} \ar{r}{E(s_B,u)} & Eg \ar{d}{E(s,\id_C)} \\
      & & E(g \circ Rf) \ar{d}{E(E(\id_A,r_B),r_Cv)} \ar{r}{E(E(\id_A,g),\id_C)} &[6em] ER(gf) \ar{r}{\mu_{gf}} \ar{d}{ER(\id_A,r_Cv)} & E(gf) \ar{d}{E(\id_A,r_Cv)} \\
      & Eg'' \ar{r}[below]{E(s',\id_C)} & E(g'' \circ Rf') \ar{r}[below]{E(E(\id_A,g''),\id_{C'})} & ER(g''f') \ar{r}[below]{\mu_{g''f'}} & E(g''f')
    \end{tikzcd}
  \]
  that the necessary coherence holds.
\end{proof}

\begin{theorem}
  \label{soa-copointed-coalgebras-double-functor}
  Fix a category $\cat{E}$ equipped with a $\kappa$-backdrop $\class{M}$, a diagram $u \co \cat{J} \to \Arr{\cat{E}}$ compatible with $\class{M}$ such that $\dom u$ is levelwise $(\kappa,\class{M})$-small, and a thin, left-connected, $(\kappa,\class{M})$-cellular notion of composable structure $U \co \dcat{A} \to \Sq{\cat{E}}$ with a compositional codomain retract lifting operator.
  Write $(\comonad{L},\monad{R})$ for the \textsc{awfs} cofibrantly generated by $u$, which exists by \cref{soa-backdrop}.
  If $(\comonad{L},\monad{R})$ is the \textsc{awfs} cofibrantly generated by $u$, then any lift $\bm{D}_{\dcat{A}} \co \Arr{\cat{E}} \to \VertArr{\dcat{A}}$ of $\Den{u} \co \Arr{\cat{E}} \to \Arr{\cat{E}}$ through $\VertArr{U}$ induces a pseudo double functor $j \co \DCoalg{\UnderPtd{\comonad{L}}} \to \dcat{A}$ fitting in the diagram
  \[
    \begin{tikzcd}[sep=small]
      \DCoalg{\UnderPtd{\comonad{L}}} \ar[dashed]{rr}{j} \ar{dr}[below left]{U_{\UnderPtd{\comonad{L}}}} && \dcat{A} \ar{dl}{U} \rlap{.} \\
       & \Sq{\cat{E}}
    \end{tikzcd}
  \]
\end{theorem}
\begin{proof}
  Under the assumptions above, we can apply \cref{soa-unit-vertical-point} to obtain lifts $\bm{L}_{\dcat{A}} \co \Arr{\cat{E}} \to \VertArr{\dcat{A}}$ of $L \co \Arr{\cat{E}} \to \Arr{\cat{E}}$ through $\VertArr{U}$.
  Following the proof of \cref{soa-copointed-coalgebras-functor}, we can define $\VertArr{j} \co \Coalg{\UnderPtd{\comonad{L}}} \to \VertArr{\dcat{A}}$ such that $\VertArr{U}\VertArr{j} = \VertArr{U_{\UnderPtd{\comonad{L}}}}$.
  We show that $\VertArr{j}$ extends to a pseudo double functor $j \co \DCoalg{\UnderPtd{\comonad{L}}} \to \dcat{A}$.

  Observe first that $j$ preserves identities up to $\leftcnx_{j(\idv_A)} \co \idv_A \to j(\idv_A)$, which is an isomorphism by conservativity of $\VertArr{U}$.
  It remains to check that $j$ preserves vertical composition up to isomorphism.
  Suppose we have vertically composable coalgebras $\bm{f} = (f,s) \co A \verto B$ and $\bm{g} = (g,t) \co B \verto C$.
  We have a diagram
  \begin{equation}
    \label{soa-double-functor-first-pyramid}
    \begin{tikzcd}[column sep=huge, row sep=large]
      &
      & A \ar{dl}[above left]{f} \ar{d}[description]{Lf} \ar{dr}{f} \\
      & \ar{dl}[above left]{g} B \ar{d}[description]{Lg} \ar{r}[description]{s}
      & Ef \ar{d}[description]{L(g \circ Rf)} \ar{r}[description]{Rf}
      & B \ar{d}[description]{Lg} \ar{dr}{g} \\
      C \ar{r}[below]{t}
      & Eg \ar{r}[below]{E(s,\id_C)}
      & E(g \circ Rf) \ar{r}[below]{E(Rf,\id_C)}
      & Eg \ar{r}[below]{Rg}
      & C
    \end{tikzcd}
  \end{equation}
  where the upper vertical map lifts to $\bm{L}_{\dcat{A}}f \in \VertArr{\dcat{A}}$ and the inner rectangle lifts to a retract diagram
  \[
    \begin{tikzcd}[cramped,column sep=5em]
      \bm{L}_{\dcat{A}}g \ar{r}{\bm{L}_{\dcat{A}}(s,\id_C)} & \bm{L}_{\dcat{A}}(g \circ Rf) \ar{r}{\bm{L}_{\dcat{A}}(Rf,\id_C)} & \bm{L}_{\dcat{A}}g
    \end{tikzcd}
  \]
  in $\VertArr{\dcat{A}}$.
  It follows by compositionality of the codomain retract lifting operator that $j\bm{g} \vcirc j\bm{f}$ is isomorphic to the lift of $\bm{L}_{\dcat{A}}(g \circ Rf) \vcirc \bm{L}_{\dcat{A}}f$ along the codomain retract $gf \to L(g\circ Rf) \circ Lf \to gf$ displayed above.
  Now, recalling the transformation $\overline{\Gm}$ provided by \cref{soa-unit-vertical-point-multiplication}, we consider the composite $\bm{L}_{\dcat{A}}(g \circ Rf) \vcirc \bm{L}_{\dcat{A}}f \to \bm{L}_{\dcat{A}}(gf) \to \bm{L}_{\dcat{A}}(g \circ Rf) \vcirc \bm{L}_{\dcat{A}}f$ given by the diagram
  \[
    \begin{tikzcd}[sep=huge]
      A \ar[phantomcenter]{dr}{\bm{L}_{\dcat{A}}(\id_A,g)} \ar[verto]{d}[left]{\bm{L}_{\dcat{A}}f} \ar[equals]{r}
      &[1.5em] A \ar[phantomcenter]{ddr}{\overline{\Gm}_{gf}} \ar[verto]{d}{\bm{L}_{\dcat{A}}(gf)} \ar[equals]{r}
      &[-2.7em] A \ar[phantomcenter]{ddr}{\lambda} \ar[verto]{dd}[description,pos=.35]{\bm{L}_{\dcat{A}}(gf)} \ar[equals]{r}
      &[-3em] A \ar[phantomcenter]{dr}{\leftcnx} \ar[verto]{d}[left]{\idv} \ar[equals]{r} & A \ar[verto]{d}{\bm{L}_{\dcat{A}}f} \\
      Ef \ar[phantomcenter]{dr}[xshift=0.2em]{\bm{L}_{\dcat{A}}(E(\id_A,g),\id_C)} \ar[verto]{d}[left]{\bm{L}_{\dcat{A}}(g \circ Rf)} \ar{r}[description]{E(\id_A,g)}
      & E(gf) \ar[verto]{d}{\bm{L}_{\dcat{A}}R(gf)}
      & 
      & A \ar[verto]{d}[left]{\bm{L}_{\dcat{A}}(gf)} \ar[phantomcenter]{dr}{\bm{L}_{\dcat{A}}(Lf,\id_C)} \ar{r}[description]{Lf} & Ef \ar[verto]{d}{\bm{L}_{\dcat{A}}(g \circ Rf)} \\
      E(g \circ Rf) \ar{r}[below]{E(E(\id_A,g),\id_C)}
      & ER(gf) \ar{r}[below]{\Gm_{gf}}
      & E(gf) \ar[equals]{r}
      & E(gf) \ar{r}[below]{E(Lf,\id_C)} & E(g \circ Rf) \rlap{.}
    \end{tikzcd}
  \]
  Although this is not itself a retract diagram, post-composing the bottom horizontal row with $Rg \circ E(Rf,\id_C) \co E(g \circ Rf) \to C$ yields
  \begin{align*}
    Rg \circ E(f,\id_C) \circ \mu_{gf} \circ E(E(\id_A,g),\id_C)
    &= R(gf) \circ \mu_{gf} \circ E(E(\id_A,g),\id_C) \\
    &= RR(gf) \circ E(E(\id_A,g),\id_C) \\
    &= R(g \circ Rf) \\
    &= Rg \circ E(Rf,\id_C).
  \end{align*}
  Thus a second application of compositionality (with a trivial top retract) implies that the lift of $\bm{L}_{\dcat{A}}(g \circ Rf) \vcirc \bm{L}_{\dcat{A}}f$ along \eqref{soa-double-functor-first-pyramid} is isomorphic to the lift of $\bm{L}_{\dcat{A}}(gf)$ along the composite retract $gf \to L(g\circ Rf) \circ Lf \to L(gf) \to L(g\circ Rf) \circ Lf \to gf$, which is by definition $j(\bm{g} \vcirc \bm{f})$.
\end{proof}

%% file: applications.tex
\section{Applications}
\label{sec:applications}

\subsection{Uniform box-filling fibrations}
\label{sec:uniform-fibrations}

To apply the theorems of \cref{sec:soa:algebraic-saturation} effectively, we must be able to compute the density comonads of generating diagrams.
We illustrate how this proceeds using the example of \textsc{awfs}'s for \emph{uniform fibrations}, which appear in semantics of homotopy type theory \cite{gambino-sattler:17}.
For simplicity we consider only ``biased'' uniform fibrations, in \citeauthor{awodey:23}'s terminology \cite[\S3]{awodey:23}, and not unbiased or equivariant variations which are better-behaved in some settings \cite{cavallo-mortberg-swan:20,awodey:23,accrs:24}.

\begin{definition}
  \label{uniform-fibration-configuration}
  A \emph{uniform fibration configuration} $(I,t)$ in a presheaf category $\cat{E} = \PSh{\cat{C}}$ consists of
  \begin{enumerate}
  \item a \emph{cofibration classifier}: a monomorphism $t \co \CofT \mono \Cof$;
  \item a \emph{functorial cylinder}: a functor $I \co \cat{C} \to \cat{C}$ with monomorphism-valued natural transformations $\delta^0,\delta^1 \colon \Id_{\cat{C}} \to I$ and a Yoneda extension $\lan{I} \co \cat{E} \to \cat{E}$ such that
    \begin{enumerate}[label=(\roman*)]
    \item $\lan{I}$ preserves pullbacks;
    \item the induced natural transformations $\delta^0,\delta^1 \co \Id_{\cat{E}} \to \lan{I}$ are cartesian.
    \end{enumerate}
  \end{enumerate}
  We write $\subst{I}$ for the right adjoint to $\lan{I}$ and ${\partial_i} \co \subst{I} \to \Id_{\cat{E}}$ for the conjugate of $\delta^i \co \Id_{\cat{E}} \to \lan{I}$.
\end{definition}

\begin{definition}
  \label{decode-characteristic}
  Given $f \colon Y \to X$ in $\cat{E} = \PSh{\cat{C}}$, write $\phi^f \co \Slice{\cat{E}}{X} \to \Arr{\cat{E}}$ for the functor that sends an object $a \co B \to X$ to the base change
  \[
    \begin{tikzcd}
      A \pullback \ar[dashed,mono]{d}[left]{\phi^f(a)} \ar[dashed]{r} & Y \ar{d}{f} \\
      B \ar[below]{r}[below]{a} & X \rlap{.}
    \end{tikzcd}
  \]
  Write $u^f \co \catEl_{\cat{C}}{X} \to \Arr{\cat{E}}$ for the composite
  \begin{tikzcd}[cramped]
    \catEl_{\cat{C}} X \ar{r}{\yo}  & \Slice{\cat{E}}{X} \ar{r}{\phi^f} & \Arr{\cat{E}}
  \end{tikzcd}.
\end{definition}

\begin{definition}
  \label{uniform-fibration-generators}
  Fix $\cat{E} = \PSh{\cat{C}}$ with a uniform fibration configuration $(t,I)$ as in \cref{uniform-fibration-configuration}.
  Write $\boxDiagram{t}{I} \defeq \copair{\leibpushapp{{\delta^0}}u^t}{\leibpushapp{{\delta^1}}u^t}$ for the diagram
  \[
    \begin{tikzcd}[row sep=large]
      \textstyle\catEl_{\cat{C}} \Cof \ar{dr}[below left]{\leibpushapp{{\delta^0}}u^t} \ar{r}{\inl} & \textstyle(\catEl_{\cat{C}} \Cof) \sqcup (\catEl_{\cat{C}} \Cof) \ar[dashed]{d}[pos=0.4]{\boxDiagram{t}{I}} & \textstyle\catEl_{\cat{C}} \Cof \ar{l}[above]{\inr} \ar{dl}{\leibpushapp{{\delta^1}}u^t} \rlap{.} \\
      & \Arr{\cat{E}}
    \end{tikzcd}
  \]
  The \emph{uniform fibration \textsc{awfs}} on $(t,I)$, when it exists, is the \textsc{awfs} cofibrantly generated by $\boxDiagram{t}{I}$.
  We call the left maps of the \textsc{awfs} \emph{trivial cofibrations}.
  We write $\TCof{t}{I}$ and $\DTCof{t}{I}$ for the category and double category of copointed endofunctor coalgebras respectively.
\end{definition}

For a monomorphism $m \colon A \mono B$, we think of $\leibpushapp{{\delta^i}}(m)$ as an ``generalized open box'': its codomain is the $B$-shaped cylinder $I_!B$, and its domain, the pushout $B \sqcup_{A} I_!A$, comprises one end of the cylinder (determined by $i \in \braces{0,1}$) together with the $A$-shaped sub-part of its interior and other end.
Such open boxes have long been used for example in simplicial \cite[\S IV.2]{gabriel-zisman:67} and cubical \cite[\S8.4]{cisinski:06} homotopy theory.

Classically, the uniform fibration \textsc{awfs} exists for any choice of configuration, as presheaf categories are locally presentable.
Constructively, even without quotients in $\Set$, we can obtain it via \cref{soa-backdrop} under additional finiteness and decidability assumptions on $(t,I)$.

\begin{definition}
  A map $f \co Y \to X$ in $\cat{E} = \PSh{\cat{C}}$ is \emph{locally $(\kappa,\class{M})$-small} if for every $a \in \catEl_{\cat{C}} X$ the domain of $u^f(a)$ is $(\kappa,\class{M})$-small.
  We say $f$ is \emph{locally finite} if for every $a \in \catEl_{\cat{C}} X$, the domain of $u^f(a)$ is a finite colimit of representables.
\end{definition}

A locally finite map is locally $(\kappa,\class{M})$-small for all $\kappa$ and $\class{M}$, since representables are $(\kappa,\class{M})$-small and the $(\kappa,\class{M})$-small objects are closed under finite colimits.

\begin{definition}
  \label{finitary-uniform-fibration-configuration}
  A uniform fibration configuration $(t,I)$ is \emph{finitary} when
  \begin{enumerate}
  \item $t \co \CofT \mono \Cof$ is levelwise complemented and locally finite.
  \item $I_!$ preserves levelwise complemented morphisms and $\delta^0, \delta^1$ are valued in levelwise complemented monomorphisms.
  \end{enumerate}
\end{definition}

\begin{example}
  \label{cubical-sets-example}
  Consider presheaves on the category $\cat{C} = \FinSet_+$ of inhabited finite sets, which are called \emph{symmetric simplicial sets} (see, \eg, \textcite{grandis:01}).
  The category $\cat{C}$ is the idempotent completion of the full subcategory $\square_{\mathrm{B}} \hookrightarrow \cat{C}$ of objects of the form $\braces{0,1}^n$ for $n \ge 0$ \cite[Proposition 8.11]{campion:23}, which is known as the \emph{Boolean cube category} \cite{buchholtz-morehouse:17}; thus $\PSh{\cat{C}} \simeq \PSh{\square_{\mathrm{B}}}$ can also be called the category of \emph{Boolean cubical sets}.

  The functor $(-) \times \braces{0,1} \colon \cat{C} \to \cat{C}$ defines a functorial cylinder whose left Kan extension $(-) \times \ival \colon \PSh{\cat{C}} \to \PSh{\cat{C}}$ is product with the representable $\ival = \yo{\braces{0,1}} \in \PSh{\cat{C}}$.
  For a cofibration classifier, we can take the \emph{levelwise complemented subobject classifier} $\top_{\mathrm{lc}} \co 1 \mono \Omega_{\mathrm{lc}}$, with $(\Omega_{\mathrm{lc}})_I$ defined to be the set of sieves on $I \in \cat{C}$ with decidable membership and $\top_{\mathrm{lc}}$ to select the total sieve.
  In a classical metatheory, $\Omega$ is simply the subobject classifier.
  Every morphism in $\cat{C}$ factors as a split epimorphism followed by a monomorphism, so every decidable sieve on $I \in \cat{C}$ is generated by a (necessarily) finite collection of monomorphisms, meaning that every levelwise complemented subobject is a finite colimit of representables.\footnote{This is related to the fact that $\cat{C}$ is an Eilenberg--Zilber category \cite[Theorem 8.12]{campion:23}.}
  Thus $1 \mono \Omega_{\mathrm{lc}}$ is locally finite, and we have a finitary uniform fibration configuration $(\top_{\mathrm{lc}},(-) \times \braces{0,1})$.
  
  The uniform fibration \textsc{wfs} on this configuration is the (trivial cofibration, fibration) \textsc{wfs} of a Quillen model structure whose cofibrations are those maps classified by $1 \mono \Omega_{\mathrm{lc}}$ \cite{gambino-sattler:17,sattler:17}, and $\PSh{\cat{C}}$ admits a constructive interpretation of homotopy type theory in which types are the uniform fibrations \cite{cohen-coquand-huber-mortberg:15,orton-pitts:18,licata-orton-pitts-spitters:18}.
\end{example}

This example is one of many where $\PSh{\cat{C}}$ is a category of \emph{cubical sets}, that is, of presheaves over a monoidal category whose objects are powers of an interval object (see, \eg, \textcite{grandis-mauri:03} or \textcite{buchholtz-morehouse:17}).
  
\begin{example}
  \label{de-morgan-example}
  \Textcite{cohen-coquand-huber-mortberg:15} give a model of homotopy type theory in \emph{De Morgan cubical sets}, presheaves over the Lawvere theory $\square_{\mathrm{DM}}$ of De Morgan algebras.
  Product with the generating object of the Lawvere theory defines a functorial cylinder; as in boolean cubical sets, its left Kan extension is product $(-) \times \ival$ with an interval object $\ival \in \PSh{\square_{\mathrm{DM}}}$.
  \citeauthor{cohen-coquand-huber-mortberg:15}~interpret types as the uniform fibrations for a cofibration classifier they call the \emph{face lattice} \cite[\S4.1]{cohen-coquand-huber-mortberg:15}, which is the minimal subclassifier $\Cof \subseteq \Omega_{\mathrm{lc}}$ classifying the face inclusions $\ival^k \times \delta^i \times \ival^{n-k} \co \ival^n \mono \ival^{n+1}$ and closed under finite intersection and union.
  The face lattice excludes for example the diagonal subobject $\Delta_{\ival} \co \ival \mono \ival \times \ival$.
  It is locally finite essentially by construction.

  This example is notable in our context because it is not generated by any set, even in a classical metatheory.
  Suppose for sake of contradiction that we had a generating set $S \subseteq \Arr{\PSh{\square_{\mathrm{DM}}}}$ of trivial cofibrations, so that every trivial cofibration is a cell complex of maps in $S$.
  For any cardinal $\lambda$, we can attach $\lambda$ copies of the 2-cube along their diagonals to obtain an object
  \[
    \begin{tikzcd}
      \coprod_\lambda \ival \ar[mono]{d}[left]{\coprod_\lambda \Delta_{\ival}} \ar{r}{\nabla} & \ival \ar[mono,dashed]{d}{m} \\
      \coprod_\lambda (\ival \times \ival) \ar[dashed]{r}[below]{q} & \pushout W_\lambda \rlap{.}
    \end{tikzcd}
  \]
  The composite
  \begin{tikzcd}[cramped]
    1 \ar[mono]{r}{\delta_0} & \ival \ar[mono]{r}{m} & W_\lambda
  \end{tikzcd}
  is classified by the face lattice and indeed a trivial cofibration, although $m$ is not.
  By our working assumption, $m\delta_0$ can be written as a cell complex of maps in $S$, in particular as a transfinite composite of maps $i_\alpha \co A_\alpha \mono A_{\alpha+1}$.
  Any stage of the transfinite composite that contains the interval $m \co \ival \mono W_\lambda$ must contain \emph{all} of $W_\lambda$, since we cannot glue on a square along its diagonal.
  Since $\ival$ is tiny (that is, $\Hom{\PSh{\square_{\mathrm{DM}}}}{\ival}{-}$ preserves all colimits), the least $\alpha$ such that $A_\alpha$ contains $\ival$ must be a successor ordinal $\alpha = \beta + 1$, and so we have a pushout square
  \[
    \begin{tikzcd}[column sep=large]
      \coprod_{i \in I} C_i \ar[mono]{d}[left]{\coprod_{i \in I} f_i} \ar{r}{[c_i]_{i \in I}} & A_\beta \ar[mono]{d} \\
      \coprod_{i \in I} D_i \ar[below]{r}[below]{[d_i]_{i \in I}} & W_\lambda
    \end{tikzcd}
  \]
  where the maps $f_i$ belong to $S$.
  There must be exactly one $i \in I$ such that the image of $d_i$ contains the image of $q \co \coprod_\lambda (\ival \times \ival) \to W_\lambda$, in which case we have a lower bound $\card{\Hom{\PSh{\square_{\mathrm{DM}}}}{\ival^2}{D_i}} \ge \lambda$ on the number of 2-cubes in $D_i$.
  But while $\lambda$ was arbitrary, the set $\set{\card{\Hom{\PSh{\square_{\mathrm{DM}}}}{\ival^2}{D}}}{(f \co C \to D) \in S}$ is bounded by some fixed cardinal; thus we have a contradiction.
\end{example}

In \cref{uniform-fibrations} below we shall construct the uniform fibration \textsc{awfs} on any finitary configuration.
We first need to establish the existence of the density comonad for $\boxDiagram{t}{I}$.

\begin{notation}
  \label{ut-density-adjoints}
  The functor $\phi^t$ of \cref{decode-characteristic} decomposes as a sequence of left adjoints:
  \[
    \begin{tikzcd}
      \Slice{\cat{E}}{\Cof} \ar{r}{\Slice{\id_{(-)}}{\Cof}} &[2em] \Slice{\Arr{\cat{E}}}{\id_\Cof} \ar{r}{\subst{(t,\id_\Cof)}} &[2em] \Slice{\Arr{\cat{E}}}{t} \ar{r}{\sum_t} & \Arr{\cat{E}} \rlap{.}
    \end{tikzcd}
  \]
  We write $\nu^t \co \Arr{\cat{E}} \to \Slice{\cat{E}}{\Cof}$ for its right adjoint, which is thus given by the composite
  \[
    \begin{tikzcd}
      \Arr{\cat{E}} \ar{r}{(-) \times t} &[2em] \Slice{\Arr{\cat{E}}}{t} \ar{r}{\ran{(t,\id_\Cof)}} &[2em] \Slice{\Arr{\cat{E}}}{\id_\Cof} \ar{r}{\Slice{\dom}{\id_\Cof}} &[2em] \Slice{\cat{E}}{\Cof} \rlap{.}
    \end{tikzcd}
  \]
\end{notation}

\begin{lemma}
  \label{lan-adjoint}
  Let $\adjunction{F}{\cat{C}}{\cat{D}}{G}$ and let $u \co \cat{J} \to \cat{C}$ be a diagram whose density comonad $\Den{u}$ is defined.
  Then $F\Den{u}G$ is the density comonad of $Fu$.
\end{lemma}
\begin{proof}
  We have $\Lan_{Fu} (Fu) \cong F (\Lan_{Fu} u)$ because left adjoints preserve left Kan extensions \cite[Theorem X.5.1]{mac-lane:98}, and $\Lan_{Fu} u \cong (\Lan_u u)G$ by the isomorphism of comma categories $\Comma{Fu}{D} \cong \Comma{u}{GD}$ natural in $D \in \cat{D}$.
\end{proof}

\begin{corollary}
  \label{cofibration-density-calculation}
  For any $t \co \CofT \mono \Cof$ in $\cat{E} = \PSh{\cat{C}}$, the density comonad of $u^t$ is $\phi^t\nu^t \co \Arr{\cat{E}} \to \Arr{\cat{E}}$.
\end{corollary}
\begin{proof}
  The Yoneda embedding $\yo \co \catEl_{\cat{C}} \Cof \to \Slice{\cat{E}} \Cof$ is a dense functor, so its density comonad is the identity \cite[Proposition 1 and Corollary 3]{mac-lane:98}.
  Since $u^t = \phi^t\yo$ by definition, it follows from \cref{lan-adjoint} that $\phi^t\Den{\yo}\nu^t$ is the density comonad of $u^t$.
\end{proof}

\begin{lemma}
  \label{lan-coproduct}
  Let $S$ be a set and $u_i \co \cat{J}_i \to \cat{C}$ for $i \in S$ be a family of diagrams in a cocomplete category $\cat{E}$.
  If $u \co \coprod_{i \in S} \cat{J} \to \cat{C}$ is the induced map from the coproduct, then the density comonoad of $u$ is the functor $C \mapsto \coprod_{i \in S} \Den{u_i}C$.
\end{lemma}
\begin{proof}
  The density comonad $\Den{u}$ is computed at $C \in \cat{C}$ by the colimit
  \begin{align*}
    \Den{u}(C)
    &\cong \colim \left(
      \begin{tikzcd}[cramped,ampersand replacement=\&,sep=small]
      \Comma{u}{C} \ar{r}{\proj{}} \& \coprod_{i \in S} \cat{J}_i \ar{r}{u} \&[1em] \cat{E}
      \end{tikzcd}\right) \\
    &\cong \textstyle\coprod_{i \in S} \colim \left(
      \begin{tikzcd}[cramped,ampersand replacement=\&,sep=small]
        \Comma{u_i}{C} \ar{r}{\proj{}} \& \cat{J_i} \ar{r}{u_i} \&[1em] \cat{E}
      \end{tikzcd}\right) \\
    &\cong \textstyle\coprod_{i \in S} \Den{u_i}(C) \rlap{.} \qedhere
  \end{align*}
\end{proof}

\begin{corollary}
  \label{trivial-cofibration-density-calculation}
  For a uniform fibration configuration $(t,I)$ on a presheaf category $\cat{E}$, the density comonad associated to $\boxDiagram{t}{I} \co \Arr{\cat{E}} \to \Arr{\cat{E}}$ is given by $\Den{\boxDiagram{t}{I}}f \cong \leibpushapp{{\delta^0}}\Den{u^t}(\leibpullapp{{\partial_0}}(f)) \sqcup \leibpushapp{{\delta^1}}\Den{u^t}(\leibpullapp{{\partial_1}}(f))$.
\end{corollary}
\begin{proof}
  By \cref{lan-adjoint,lan-coproduct,cofibration-density-calculation}.
\end{proof}

We also need to know that $\Den{\boxDiagram{t}{I}}$ interacts well with levelwise complemented monomorphisms.

\begin{notation}
  Given a presheaf category $\cat{E}$, write $\LCMono{\cat{E}} \hookrightarrow \cat{E}$ for its wide subcategory of levelwise complemented monomorphisms.
  Write $\CartArr{\cat{E}} \hookrightarrow \Arr{\cat{E}}$ for the wide subcategory of pullback squares, $\CartMono{\cat{E}} \hookrightarrow \CartArr{\cat{E}}$ for its full subcategory whose objects are the monomorphisms in $\cat{E}$, and $\CartLCMono{\cat{E}} \hookrightarrow \CartLCMono{\cat{E}}$ for the further full subcategory of maps in $\LCMono{\cat{E}}$.
\end{notation}

\begin{lemma}
  \label{pushforward-levelwise-complemented}
  If $f \colon Y \to X$ is a locally finite morphism in $\cat{E} = \PSh{\cat{C}}$, then the functorial action of the pushforward $\ran{f} \colon \Slice{\cat{E}}{Y} \to \Slice{\cat{E}}{X}$ (that is, the right adjoint to pullback) preserves $\LCMono{\cat{E}} \hookrightarrow \cat{E}$.
\end{lemma}
\begin{proof}
  A monomorphism $m \colon A \mono B$ is in $\LCMono{\cat{E}}$ if and only if for every $d \co \yo c \to B$, we can decide whether $d$ lifts through $m$.
  Let $m \colon (A,a) \mono (B,b)$ be a morphism in $\Slice{\cat{E}}{Y}$ whose underlying morphism is in $\LCMono{\cat{E}}$ and write $\ran{f}m \colon \ranob{f}A \mono \ranob{f}B$ for the morphism in $\cat{E}$ underlying its pushforward along $f$.
  For $d \co \yo c \to \ranob{f}B$, we see by transposition that $d$ lifts through $\ran{f}m$ if and only if we have a lift as indicated in the diagram
  \[
    \begin{tikzcd}[column sep=large]
      \yo c \times_{X} Y \ar{dr}[below left,pos=.4]{\projr} \ar[dashed]{r} \ar[bend left]{rr}{d^\dagger} & A \ar{d}[left,pos=.4]{a} \ar[mono]{r}[pos=.4]{m} & B \ar{dl}[pos=.3]{b} \rlap{,} \\
        & Y
    \end{tikzcd}
  \]
  \ie, if and only if $d^\dagger \colon \yo \times_X Y \to B$ lifts through $m$.
  The object $\yo c \times_X Y$ is the domain of $u^f(d)$, thus a finite colimit of representables by assumption.
  If $(\inj{i} \colon \yo c_i \to \yo c \times_X Y)_{i \in I}$ is the colimit cocone, then $d^\dagger$ lifts through $m$ if and only if $d^\dagger\inj{i} \colon \yo c_i \to Y$ lifts through $m$ for all $i \in I$, and this we can decide by assumption.
\end{proof}

\begin{lemma}
  \label{cofibrations-mono}
  Let $(t,I)$ be a uniform fibration configuration in a presheaf category $\cat{E}$.
  The density comonad $\Den{u^t} \co \Arr{\cat{E}} \to \Arr{\cat{E}}$ lifts through $\CartMono{\cat{E}} \hookrightarrow \Arr{\cat{E}}$ and preserves monomorphisms.
  If $(t,I)$ is finitary, then the density comonad furthermore lifts through $\CartLCMono{\cat{E}} \hookrightarrow \Arr{\cat{E}}$ and preserves $\LevelClass{\cat{E}}{\LCMono{\cat{E}}}$.
\end{lemma}
\begin{proof}
  By \cref{cofibration-density-calculation} we have $\Den{u^t} \cong \phi^t\nu^t$.
  The functor $\phi^t$ sends all morphisms in $\Slice{\cat{E}}{\Cof}$ to cartesian squares by pullback pasting, so $\phi^t$ and thus $\Den{u^t}$ lifts through $\CartArr{\cat{E}} \hookrightarrow \Arr{\cat{E}}$.
  Likewise, $\phi^t$ is valued in monomorphisms and in $\CartLCMono{\cat{E}}$ if $(t,I)$ is finitary, so $\Den{u^t}$ lifts through $\CartMono{\cat{E}} \hookrightarrow \Arr{\cat{E}}$ and through $\CartLCMono{\cat{E}} \hookrightarrow \Arr{\cat{E}}$ if $(t,I)$ is finitary.

  We know $\nu^t$ preserves monomorphisms as a right adjoint, and $\phi^t$ preserves monomorphisms by cancellation, so $\Den{u^t}$ preserves monomorphisms.
  Suppose $(t,I)$ is finitary.
  Inspecting the decompositions of $\phi^t$ and $\nu^t$ in \cref{ut-density-adjoints}, we can see that $\Den{u^t}$ will preserve $\LevelClass{\cat{E}}{\LCMono{\cat{E}}}$ as long as the action of $\ran{(t,\id_{\Cof})} \colon \Slice{\Arr{\cat{E}}}{t} \to \Slice{\Arr{\cat{E}}}{\id_{\Cof}}$ preserves $\LevelClass{\cat{E}}{\LCMono{\cat{E}}}$.
  This follows from \cref{pushforward-levelwise-complemented} applied with the presheaf category $\Arr{\cat{E}} \simeq \PSh{\Arr{\cat{C}}}$ and the morphism $(t,\id_{\Cof}) \co t \to \id_{\Cof}$, which is locally finite if $t$ is.
\end{proof}

\begin{corollary}
  \label{trivial-cofibrations-mono}
  For any uniform fibration configuration $(t,I)$ on a presheaf category $\cat{E}$, if the monomorphisms form a $\kappa$-backdrop (\cref{monomorphisms-backdrop}) then the diagram $\boxDiagram{t}{I}$ is compatible with this backdrop.
  If $(t,I)$ is finitary, then $\boxDiagram{t}{I}$ is compatible with the backdrop of levelwise complemented monomorphisms (\cref{complemented-backdrop}).
\end{corollary}
\begin{proof}
  For the first claim, it suffices by \cref{union-backdrop-compatible} to check that $\Den{\boxDiagram{t}{I}}$ lifts through $\CartMono{\cat{E}} \hookrightarrow \Arr{\cat{E}}$ and preserves levelwise monomorphisms.
  By \cref{cofibrations-mono,trivial-cofibration-density-calculation}, it suffices in turn to show that each $\leibpushapp{{\delta^i}} \co \Arr{\cat{E}} \to \Arr{\cat{E}}$ restricts to a functor $\CartMono{\cat{E}} \to \CartMono{\cat{E}}$ and that this restriction preserves monomorphisms; note that the functors $\leibpullapp{{\partial_i}}$ are right adjoints and thus automatically preserve monomorphisms.

  For $i \in \braces{0,1}$ and a monomorphism $m \co A \to B$ in $\cat{E}$, $\leibpushapp{{\delta^i}}(f)$ is defined as the pushout gap map
  \[
    \begin{tikzcd}[column sep=large]
      A \ar{d}[left]{m} \ar{r}{\delta^i_A} & \lan{I}A \ar{d} \ar[bend left=20]{ddr}{\lan{I}m} \\
      B \ar[bend right=20]{drr}[below]{\delta^i_B} \ar{r} & \pushout \object \ar{dr}[pos=0.2]{\leibpushapp{{\delta^i}}(m)} \\[-2em]
      & & \lan{I}B \rlap{.}
    \end{tikzcd}
  \]
  Combined with the assumptions that $\delta^i$ is cartesian and that $\lan{I}$ preserves pullbacks and thus monomorphisms, this makes $\leibpushapp{{\delta^i}}(m)$ the union of the monomorphisms $\delta^i_B$ and $\lan{I}m$ \cite[Theorem 5.1]{lack-sobocinski:05} and thus a monomorphism.

  For any $f \co B' \to B$, $\leibpushapp{{\delta^i}}$ applied to the pullback of $m$ along $f$ is the pullback of $\leibpushapp{\delta^i}(m)$ along $\lan{I}f$, by pullback stability of colimits in $\cat{E}$ together with the assumptions that $\lan{I}$ preserves pullbacks and $\delta_i$ is cartesian.
  Thus $\leibpushapp{\delta^i}$ lifts to a functor $\CartMono{\cat{E}} \to \CartMono{\cat{E}}$.

  Given a levelwise monomorphism $(k,\ell) \co m \to n$ in $\CartMono{\cat{E}}$, the codomain component of $\leibpushapp{\delta^i}(k,\ell)$ is a monomorphism because $\lan{I}$ preserves monomorphisms, whence the domain component is a monomorphism by cancellation.

  We now assume $(t,I)$ is finitary and check the second claim.
  It suffices by \cref{union-backdrop-compatible} to check that $\Den{\boxDiagram{t}{I}}$ lifts through $\CartLCMono{\cat{E}} \hookrightarrow \Arr{\cat{E}}$ and preserves $\LevelClass{\cat{E}}{\LCMono{\cat{E}}}$.
  The first property follows from the assumptions that $\delta^i$ is valued in and $I_!$ preserves $\LCMono{\cat{E}}$ and the closure of levelwise complemented monomorphisms under binary union \cite[Lemma 2.1.7(i)]{gambino-sattler-szumilo:22}.
  For the second, we first observe that $\leibpullapp{{\partial_i}}$ preserves $\LevelClass{\cat{E}}{\LCMono{\cat{E}}}$ because $\subst{I}$ preserves $\LCMono{\cat{E}}$ and complemented monomorphisms are closed under finite limits \cite[Lemma 2.1.7(ii)]{gambino-sattler-szumilo:22}.
  Thus it is enough, as in the monomorphism case, to check that the restriction of $\leibpushapp{{\delta^i}}$  to $\CartLCMono{\cat{E}}$ preserves $\LevelClass{\cat{E}}{\LCMono{\cat{E}}}$, and this follows from cancellation for levelwise complemented monomorphisms.
\end{proof}

\begin{theorem}
  \label{uniform-fibrations}
  \input{uniform-fibrations}
\end{theorem}
\begin{proof}
  By \cref{soa-backdrop} applied either with the $\kappa$-backdrop of monomorphisms or with the $\omega$-backdrop $\LCMono{\cat{E}}$ and with the diagram $\boxDiagram{t}{I}$, which is suitable by \cref{trivial-cofibrations-mono}.
  For $i \in \braces{0,1}$ and $a \co \yo c \to \Cof$, the domain of $\leibpushapp{{\delta^i}}u^t(a)$ is the pushout
  \[
    \begin{tikzcd}
      \dom u^t(a) \ar[mono]{d} \ar{r}{\delta^i_A} & \lan{I}(\dom u^t(a)) \ar[mono]{d} \\
      \yo c \ar{r} & \pushout \object \rlap{.}
    \end{tikzcd}
  \]
  Since $(\kappa,\class{M})$-small objects are closed under finite colimits, the smallness condition on $\boxDiagram{t}{I}$ is thus satisfied in each case.
\end{proof}

In the following sections, we use uniform fibration \textsc{awfs}'s to illustrate applications of the theorems of \cref{sec:soa:algebraic-saturation}.
To avoid becoming repetitive we focus on the case where $(t,I)$ is finitary.

\subsection{Action of a functor on left maps}
\label{sec:applications:action}

Suppose we have a functor $F \co \cat{E} \to \cat{F}$ and a \textsc{wfs} $(\class{L},\class{R})$ on $\cat{E}$ generated by a set $S$ using Quillen's small object argument, and we want to show that $F$ sends $\class{L}$ into some saturated class of maps $\class{A}$---perhaps the left class of another \textsc{wfs} on $\cat{F}$.
The saturation characterization of $\class{L}$ tells us that if $F$ preserves cell complex structure, then it is enough to check that $F$ sends $S$ into $\class{A}$.
For our first application, we prove an analogue of this result for \textsc{awfs}'s cofibrantly generated by a diagram.

\begin{definition}
  Let $F \co \cat{E} \to \cat{F}$ be a functor and $U \co \dcat{A} \to \Sq{\cat{F}}$ be a notion of composable structure.
  We define the notion of composable structure $\subst{F}U \co \subst{F}\dcat{A} \to \Sq{\cat{E}}$ by pulling back along $F$:
  \begin{equation}
    \label{preimage-vertical-structure-pullback}
    \begin{tikzcd}[column sep=large]
      \VertArr{(\subst{F}\dcat{A})} \ar[dashed]{d}[left]{\VertArr{(\subst{F}U)}} \pullback \ar[dashed]{r} & \VertArr{\dcat{A}} \ar{d}{\VertArr{U}} \\
      \Arr{\cat{E}} \ar{r}[below]{\Arr{F}} & \Arr{\cat{F}} \rlap{.}
    \end{tikzcd}
  \end{equation}
  That is, a vertical morphism $A \verto B$ in $\subst{F}\dcat{A}$ consists of a morphism $f \co A \to B$ in $\dcat{D}$ together with a vertical morphism $\bm{g} \co FA \verto FB$ in $\dcat{A}$ such that $U\bm{g} = Ff$.
  Likewise, a square in $\subst{F}\dcat{A}$ consists of a commutative square in $\cat{E}$ together with a square of $\dcat{A}$ over its image by $F$.
  Identities and composition are given in the evident way.
  The diagram \eqref{preimage-vertical-structure-pullback} extends to a double functor
  \[
    \begin{tikzcd}[column sep=large]
      \subst{F}\dcat{A} \ar[dashed]{d}[left]{\subst{F}U} \pullback \ar[dashed]{r} & \dcat{A} \ar{d}{U} \\
      \Sq{\cat{E}} \ar{r}[below]{\Sq{F}} & \Sq{\cat{F}} \rlap{.}
    \end{tikzcd}
  \]
\end{definition}

\begin{proposition}
  \label{preimage-left-connected}
  Let $U \co \dcat{A} \to \Sq{\cat{F}}$ be a notion of composable structure and $F \co \cat{E} \to \cat{F}$ be a functor.
  If $U$ is left-connected, then so is $\subst{F}U$.
\qed
\end{proposition}

\begin{proposition}
  \label{preimage-cellular}
  Let $U \co \dcat{A} \to \Sq{\cat{F}}$ be a notion of composable structure and $F \co (\cat{E},\class{M}) \to (\cat{F},\class{N})$ be a $\kappa$-backdrop-preserving functor.
  If $U$ is $(\kappa,\class{N})$-cellular, then $\subst{F}U \co \subst{F}\dcat{A} \to \Sq{\cat{E}}$ is $(\kappa,\class{M})$-cellular.
\end{proposition}
\begin{proof}
  The projections from $\VertArr{(\subst{F}\dcat{A})}$ to $\Arr{\cat{E}}$ and $\VertArr{\dcat{A}}$ send $\DCodClass{(\subst{F}\dcat{A})}{\class{M}}$ into $\CodClass{\cat{E}}{\class{M}}$ and $\DCodClass{\dcat{A}}{\class{N}}$ respectively, and the cospan of the defining pullback \eqref{preimage-vertical-structure-pullback} extends to a diagram of $\kappa$-backdrop-preserving functors
  \[
    \begin{tikzcd}[sep=large]
       & (\VertArr{\dcat{A}},\DCodClass{\dcat{A}}{\class{N}}) \ar{d}{\VertArr{U}} \\
      (\Arr{\cat{E}},\CodClass{\cat{E}}{\class{M}}) \ar{r}[below]{\Arr{F}} & (\Arr{\cat{F}},\CodClass{\cat{F}}{\class{N}}) \rlap{.}
    \end{tikzcd}
  \]
  It follows that that the relevant colimits in $\VertArr{(\subst{F}\dcat{A})}$ can be computed levelwise and that the pullback projection $\VertArr{(\subst{F}U)}$ preserves them, as required.
\end{proof}

\begin{proposition}
  \label{preimage-retracts}
  Let $F \co \cat{E} \to \cat{F}$ be a functor and $U \co \dcat{A} \to \Sq{\cat{F}}$ be a notion of composable structure.
  Any (compositional) codomain retract lifting operator on $U$ induces a (compositional) codomain retract lifting operator on $\subst{F}U$.
\qed
\end{proposition}

We now deduce the following generalization of \cref{soa-copointed-coalgebras-functor,soa-copointed-coalgebras-double-functor}.

\begin{theorem}
  \label{preimage-coalgebras-functor}
  \input{preimage-coalgebras-functor}
\end{theorem}
\begin{proof}
  We apply \cref{soa-copointed-coalgebras-functor} with the $\kappa$-backdrop $\class{M}$ and the notion of composable structure
  \[
    \subst{F}U \co \subst{F}\dcat{A} \to \Sq{\cat{E}} \rlap{,}
  \]
  which is $(\kappa,\class{M})$-cellular by \cref{coalgebra-double-category-cellular,preimage-cellular} and admits retract lifting by \cref{preimage-retracts}.
  We obtain a pseudo double functor
  \[
    \begin{tikzcd}[row sep=large]
      \Coalg{\UnderPtd{\comonad{L}}} \ar{dr}[below left]{U_{\UnderPtd{\comonad{L}}}} \ar[dashed]{rr} &[-1em]&[-1em] \VertArr{\subst{F}\dcat{A}} \ar{dl}{\VertArr{\subst{F}U}} \ar{r} & \VertArr{\dcat{A}} \ar{d}{\VertArr{U}} \\
      & \Arr{\cat{E}} \ar{rr}[below]{\Arr{F}} & & \Arr{\cat{F}} \rlap{.}
    \end{tikzcd}
  \]
  The top horizontal composite is the desired functor $j \co \Coalg{\UnderPtd{\comonad{L}}} \to \VertArr{\dcat{A}}$.
  When $\dcat{A}$ is thin and its codomain retract lifting operator is compositional, we can instead apply \cref{soa-copointed-coalgebras-double-functor} to get a pseudo double functor.
\end{proof}

\begin{remark}
  Note that the analogous theorem for Quillen's small object argument only requires $F$ to preserve cobase changes and transfinite composites of maps in the left class, whereas our theorem requires this for the larger class $\class{M}$.
  It is not clear to us whether the stronger result can be derived from ours when the generating diagram is discrete.
\end{remark}

A typical application is to show that that a backdrop-preserving functor sends the left maps of one \textsc{awfs} to left maps of another \textsc{awfs}.
In the specific case of a uniform fibration $\textsc{awfs}$, the density comonad condition reduces to a natural requirement on pushout applications of the endpoints $\delta_i$ to cofibrations:

\begin{example}
  \label{functor-preserves-left-maps-uniform}
  Let $(t,I)$ be a finitary uniform fibration configuration in a presheaf category $\cat{E} = \PSh{\cat{C}}$.
  Let $(\comonad{L},\comonad{R})$ be an \textsc{awfs} on a category $\cat{F}$ and let $F \co (\cat{E},\LCMono{\cat{E}}) \to (\cat{F},\class{N})$ be an $\omega$-backdrop-preserving functor.
  If $F\leibpushapp{\delta_i}\phi^t \co \Slice{\cat{E}}{\Cof} \to \Arr{\cat{F}}$ lifts through $\Coalg{\UnderPtd{\comonad{L}}}$ for $i \in \braces{0,1}$, then $\Sq{F} \co \Sq{\cat{E}} \to \Sq{\cat{F}}$ lifts to a double functor $\DTCof{t}{I} \to \DCoalg{\UnderPtd{\comonad{L}}}$.
\end{example}
\begin{proof}
  The notion of composable structure $\DCoalg{\UnderPtd{\comonad{L}}}$ with the backdrop $\class{N}$ satisfies the requirements of \cref{preimage-coalgebras-functor} by \cref{coalgebra-double-category-cellular,coalgebra-retract-lifting-compositional}.
  The result follows using \cref{trivial-cofibrations-mono,trivial-cofibration-density-calculation}.
  Note that the pseudo double functor we obtain is necessarily a double functor because $\VertArr{U_{\UnderPtd{\comonad{L'}}}}$ is amnestic, \ie, any isomorphism it sends to an identity is itself an identity.
\end{proof}

\begin{example}
  In work currently in progress, we use \cref{functor-preserves-left-maps-uniform} to analyze exponentiation by quotients of representables in cubical set categories.
  We briefly sketch the situation here, using the Boolean cubical sets from \cref{cubical-sets-example}.

  For this form of cubical set, we have an isomorphism $\sigma \co \ival \cong \ival$ that swaps $0$ and $1$.
  Exponentiation by the quotient $\ival/\sigma$ in $\PSh{\cat{C}}$ defines a functor $(-)^{\ival/\sigma} \co \PSh{\cat{C}} \to \PSh{\cat{C}}$ that does not preserve all colimits, but turns out to preserve colimits of $\omega$-chains and pushouts along monomorphisms.
  It moreover preserves $\ival$: the inclusion of constant maps $\ival \to \ival^{\ival/\sigma}$ is an isomorphism.
  It follows from \cref{functor-preserves-left-maps-uniform} that $(-)^{\ival/\sigma}$ preserves the trivial cofibrations of any finitary uniform fibration configuration for which $(-)^{\ival}$ (and thus $(-)^{\ival/\sigma}$) preserves cofibrations.
  Ken Brown's lemma \cite[Lemma 1.1.12]{hovey:99} then implies that $(-)^{\ival/\sigma}$ preserves weak equivalences between cofibrant objects.

  This implies that the model structure mentioned in \cref{cubical-sets-example} does not coincide, in classical set theory, with Cisinski's \emph{test model structure} on this presheaf category \cite{cisinski:06,buchholtz-morehouse:17}.
  Indeed, the object $\ival/\sigma$ is contractible in the test model structure but cannot be contractible in the uniform fibration model structure: the argument above would then imply that $(\ival/\sigma)^{\ival/\sigma}$ is also contractible, but $(\ival/\sigma)^{\ival/\sigma}$ is isomorphic to $1 \sqcup \ival/\sigma$ and thus patently non-contractible.

  In the case of \emph{cartesian cubical sets}, an unpublished note of \textcite{coquand:18:counterexample} describes an argument by the second author that the quotient of the 2-cube by the reflection along its diagonal is not contractible in a uniform fibration model structure, using an explicit inductive description of the (trivial cofibration, fibration) factorization for that model structure.
  Our saturation principle abstracts the inductive process involved in this argument.
\end{example}

\subsection{Extension operations}
\label{sec:applications:extension}

As a second application, we consider the construction of \emph{extension operations}.
The general pattern is that we have some notion of structure on objects of our category and want to show that any structure on the domain of a left map induces a compatible structure on its codomain.
The motivating example is the extension of \emph{fibrations along trivial cofibrations}, where the structure on an object is a fibration over it:

\begin{example}
  A Quillen premodel category \cite{barton:19} has the \emph{fibration extension property} when for any trivial cofibration $m \co A \tcof B$ and fibration $p \co X \fib A$ over its domain, there is a square
  \[
    \begin{tikzcd}
      X \pullback \ar[fib]{d}[left]{p} \ar[dashed]{r} & Y \ar[dashed,fib]{d} \\
      A \ar[tcof]{r}[below]{m} & B
    \end{tikzcd}
  \]
  exhibiting $p$ as a pullback of a fibration over $B$.
\end{example}

The property is used in homotopical semantics of type theory, for example by \textcite[Proposition 2.21]{cisinski:14} to build interpretations of type theory from Quillen model categories, and conversely by \textcite{sattler:17} and others \cite{awodey:23,cavallo-sattler:23,accrs:24} to build Quillen model categories from or in conjunction with interpretations of type theory.
It is, in particular, connected to fibrancy of universes classifying fibrations: under some conditions, a fibration $p \co \widetilde{U} \fib U$ will have a fibrant base object $U$ if and only if $U$-small fibrations extend along trivial cofibrations \cite[Remark 7.6]{sattler:17}~\cite[Proposition 2.2.4]{stenzel:19}.
It and related properties also appear in homotopy theory more broadly (\eg, \cite[Lemma 1.7.1]{joyal-tierney:99}~\cite[Corollary 5.2.11]{cisinski:19}).

In these situations, one often first shows explicitly that fibrations extend along \emph{generating} trivial cofibrations.
When the extension property can be reformulated as fibrancy of a hierarchy of universes, extension along generators immediately gives extension along all trivial cofibrations.
If one is primarily interested in building a model category, however, it would be conceptually cleaner not to depend on a stratification of the category by such a hierarchy.

Instead, one can use a decomposition of trivial cofibrations by some kind of small object argument, as is done in \textcite[Corollary 7.4]{sattler:17} and \textcite[Corollary 4.2.7]{gambino-sattler-szumilo:22}.
We show in this section how to carry out such an argument for a class of trivial cofibrations generated using the \emph{algebraic} small object argument, where previous work was limited to factorization systems generated by Quillen's argument.

\begin{remark}
  An alternative approach to the following result is described in unpublished work of the second author \cite{sattler-free-monad}, using the fine-grained functoriality of \cref{sec:free-monad-sequence} directly in place of this article's saturation machinery.
  That approach has the advantage of avoiding pushforward manipulations by working with individual extension problems in place of extension operations.
\end{remark}

We consider notions of structure on objects of a category given as a Grothendieck fibration over that category.
We first fix some notation.

\begin{notation}
  Let $P \co \cat{E} \to \cat{B}$ be a Grothendieck fibration.
  \begin{itemize}
  \item For $a \in \cat{B}$, we write $\cat{E}_a$ for the fiber of $P$ over $a$.
  \item Similarly, for $u \co a \to b$ in $\cat{B}$ we write $\Arr{\cat{E}}_u$ for the fiber of $\Arr{P} \co \Arr{\cat{E}} \to \Arr{\cat{B}}$ over $u$.
  \item We assume that our fibrations come with a choice of cartesian lifts (\ie, are \emph{cloven}) and write $\overline{u}e \co \subst{u}e \to e$ for the chosen cartesian lift of some $e \in \cat{E}_a$ along $u \co b \to a$.
  \item We write $\FibOb{\cat{B}}{P} \co \FibOb{\cat{B}}{\cat{E}} \to \cat{B}$ for the restriction of $P$ to the wide subcategory $\FibOb{\cat{B}}{\cat{E}}$ consisting of $P$-cartesian morphisms (sometimes called the \emph{fibration of objects}).
  \end{itemize}
\end{notation}

We will also need an analogue of the concept of \emph{Van Kampen colimit} \cite{lack-sobocinski:04,cockett-xuo:07,heindel-sobocinski:11} in the setting of a Grothendieck fibration.

\begin{definition}
  \label{van-kampen}
  Let $P \co \cat{E} \to \cat{B}$ be a Grothendieck fibration.
  We say that a colimit cocone $\beta \co b \to \Delta b_0$ under a diagram $b \co \cat{K} \to \cat{B}$ is \emph{Van Kampen for $P$} when
  \begin{enumerate}[label=(\alph*),ref=\thedefinition(\alph*)]
  \item \label{vk-existence} every $e \co \cat{K} \to \FibOb{\cat{B}}{\cat{E}}$ over $b$ admits a colimit cocone in $\FibOb{\cat{B}}{\cat{E}}$ over $\beta$;
  \item \label{vk-restrict}
    given $e \co \cat{K} \to \FibOb{\cat{B}}{\cat{E}}$, every cocone $\eta \co e \to \Delta e_0$ in $\FibOb{\cat{B}}{\cat{E}}$ over $\beta$ is a colimit cocone in $\cat{E}$.
  \end{enumerate}
\end{definition}

\begin{remark}
  The cocone $\beta$ is Van Kampen for $P$ exactly if the corresponding pseudofunctor $\cat{E}_{(-)} \co \op{\cat{B}} \to \Cat$ sends it to a bilimit of categories.
\end{remark}

A Van Kampen colimit in the usual sense in a category $\cat{E}$ with pullbacks is then a colimit which is Van Kampen for the codomain fibration $\cod \co \Arr{\cat{E}} \to \cat{E}$.

\subsubsection{Cartesian pushforwards}

We want to associate left maps with \emph{operations} that take a structure on the domain as input and output an extension to the codomain.
To define a category of such operations, we make use of pushforwards in $\Cat$.
Recall that a functor $P \co \cat{E} \to \cat{B}$ is called \emph{exponentiable} if the pullback functor $\subst{P} \co \Slice{\Cat}{\cat{B}} \to \Slice{\Cat}{\cat{E}}$ has a right adjoint $\ran{P}$, and that any Grothendieck fibration is exponentiable \cite[Lemme 4.3, Th\'eor\`eme 4.4]{giraud:64}.
Given $F \co \cat{F} \to \cat{E}$, we write $\ran{P}F \co \ranob{P}{F} \to \cat{B}$ for the application of the right adjoint and call this the \emph{pushforward of $F$ along $P$}.

In fact we want a variation on the pushforward with a stronger condition on morphisms.

\begin{definition}
  Let $P \co \cat{E} \to \cat{B}$ and $Q \co \cat{F} \to \cat{B}$ be Grothendieck fibrations and $V \co (\cat{F},Q) \to (\cat{E},P)$ be a fibered functor over $\cat{E}$.
  The \emph{cartesian pushforward} $\ranc{P}{V} \co \rancob{P}{V} \to \cat{B}$ is the pullback
  \[
    \begin{tikzcd}
      \rancob{P}{V} \pullback \ar[dashed,bend left=20]{rr}{\ranc{P}{V}} \ar[dashed]{d} \ar[dashed]{r} & \prod_P V \ar{d} \ar{r}[below]{\ran{P}V} & \cat{B} \rlap{,} \\
      \prod_{\FibOb{\cat{B}}{P}} \FibOb{\cat{B}}{V} \ar{r} & \prod_{\FibOb{\cat{B}}{P}} V
    \end{tikzcd}
  \]
  where $\FibOb{\cat{B}}{V} \co \FibOb{\cat{B}}{P} \to \FibOb{\cat{B}}{Q}$ is the restriction of $V$ to cartesian arrows over $\cat{B}$.
\end{definition}

The objects $\xi \in \ranob{P}{V}$ over $b \in \cat{B}$ correspond, by transposition, to sections $\xi^\dagger \co \cat{E}_b \to \cat{F}_b$ of the restriction $V_b \co \cat{F}_b \to \cat{E}_b$ of $V$ to the fibers over $b$.
The category $\rancob{P}{V}$ has the same objects: they are sections $\cat{E}_b \to \cat{F}_b$ which preserve cartesian morphisms over $\cat{B}$, but this is no requirement since any vertical cartesian morphism is an isomorphism.
Its morphisms are however more restrictive than those of $\ranob{P}{V}$, as we now unpack.

\begin{proposition}
  \label{pushforward-morphisms}
  Let $P \co \cat{E} \to \cat{B}$ and $V \co \cat{F} \to \cat{E}$ be a functor.
  Given a morphism $\alpha \co b \to b'$ in $\cat{B}$ and objects $\xi \in \left(\ranob{P}{V}\right)_b$ and $\xi' \in \left(\ranob{P}{V}\right)_{b'}$, the morphisms $\beta \co \xi \to \xi'$ in $\ranob{P}{V}$ over $\alpha$ correspond to natural transformations
  \[
    \begin{tikzcd}[row sep=tiny,column sep=large]
      \cat{E}_b \ar{dr}{\xi^\dagger} \\
      { } \ar[phantom,pos=.3]{r}{\Downarrow \beta^{\ddagger}} & \cat{F} \\
      \cat{E}_{b'} \ar{uu}{\subst{\alpha}} \ar{ur}[below right]{\xi'^\dagger}
    \end{tikzcd}
  \]
  such that $V \beta^{\ddagger}_e = \overline{\alpha} e \co \subst{\alpha}e \to e$ for $e \in \cat{E}_{b'}$.
\end{proposition}
\begin{proof}
  Use the universal property of $\ran{P}{V}$ to characterize the functors $\2 \to \ranob{P}{V}$.
\end{proof}

\begin{corollary}
  \label{cartesian-pushforward-morphisms}
  Let $P \co \cat{E} \to \cat{B}$ and $Q \co \cat{F} \to \cat{B}$ be Grothendieck fibrations and $V \co (\cat{F},Q) \to (\cat{E},P)$ be a fibered functor over $\cat{E}$.
  Given a morphism $\alpha \co b \to b'$ in $\cat{B}$ and objects $\xi \in \left(\rancob{P}{V}\right)_b$ and $\xi' \in \left(\rancob{P}{V}\right)_{b'}$, the morphisms $\beta \co \xi \to \xi'$ in $\rancob{P}{V}$ over $\alpha$ are those natural transformations as in \cref{pushforward-morphisms} which are valued in $Q$-cartesian morphisms. These correspond in particular to natural isomorphisms $\theta \co \xi^\dagger\subst{\alpha} \to \subst{\alpha}\xi'^\dagger$ such that
  \[
    \begin{tikzcd}
      \cat{E}_{b} \ar[phantom]{dr}{\theta} \ar{r}{\xi^\dagger} & \cat{F}_b \ar[phantomcenter]{dr}{\cong} \ar{r}{V_b} & \cat{E}_b \\
      \cat{E}_{b'} \ar{u}{\subst{\alpha}} \ar{r}[below]{\xi'^\dagger} & \cat{F}_{b'} \ar{u}{\subst{\alpha}} \ar{r}[below]{V_{b'}} & \cat{E}_{b'} \ar{u}[right]{\subst{\alpha}}
    \end{tikzcd}
    \quad
    =
    \quad
    \begin{tikzcd}
      \cat{E}_b \ar[equals]{r} & \cat{E}_b \\
      \cat{E}_{b'} \ar{u}{\subst{\alpha}} \ar[equals]{r} & \cat{E}_{b'} \ar{u}[right]{\subst{\alpha}} \rlap{.}
    \end{tikzcd}
  \]
\end{corollary}

Since the goal is to build a cellular notion of composable structure for extension operations, we are interested in colimits in pushforward categories.
Here we need Van Kampen colimits in the base.

\begin{notation}
  Fix a Grothendieck fibration $P \co \cat{E} \to \cat{B}$, a diagram $b \co \cat{K} \to \cat{B}$, and a cocone $\beta \co b \to \Delta b_0$ in $\cat{B}$.
  Then we have a functor and natural transformation
  \[
    \begin{tikzcd}[column sep=small,row sep=large]
      \cat{K} \times_{\cat{B}} \cat{E} \ar{dr}[below left]{\subst{P}b} & {} & \ar[dashed]{ll}[above]{\corr{P}{\beta}} \cat{K} \times \cat{E}_{b_0} \ar{dl}{\subst{P}\Delta b_0} \\
      & \cat{E} \ar[phantom]{u}[yshift=0.7em]{\lift{P}\beta}[yshift=-0.3em]{\Rightarrow}
    \end{tikzcd}
  \]
  as follows: $\corr{P}{\beta}$ sends  $(k,e)$ to $(k,\subst{\beta_k}e)$ and $\lift{P}{\beta}$ sends $(k,e)$ to $\overline{\beta_k}e \co \subst{\beta}_ke \to e$.
\end{notation}

\begin{lemma}
  \label{pushforward-colimit}
  Let $P \co \cat{E} \to \cat{B}$ and $Q \co \cat{F} \to \cat{B}$ be Grothendieck fibrations and $V \co (\cat{F},Q) \to (\cat{E},P)$ be a fibered functor over $\cat{E}$  which is itself an isofibration.
  Let a small category $\cat{K}$ and diagram
  \[
    \begin{tikzcd}[row sep=small]
      \cat{K} \ar{dr}[below left]{b} \ar{rr}{d} && \rancob{P}V \ar{dl}{\ranc{P}V} \\
      & \cat{B}
    \end{tikzcd}
  \]
  be given, and suppose that $b \co \cat{K} \to \cat{B}$ admits a colimit that is Van Kampen for $P$ and $Q$.
  Then the colimit of $d \co \cat{K} \to \rancob{P}V$ exists and is preserved by $\ranc{P}V$.
\end{lemma}
\begin{proof}
  By assumption, we have a colimit cocone $\beta \co b \to \Delta b_0$.
  Consider the transpose $d^\dagger \co \cat{K} \times_{\cat{B}} \cat{E} \to \cat{F}$.
  For each $e \in \cat{E}_{b_0}$, the composite
  \[
    \begin{tikzcd}
      \cat{K} \ar{r}{\cat{K} \times e} &[1em] \cat{K} \times \cat{E}_{b_0} \ar{r}{\corr{P}{\beta}} &[1em] \cat{K} \times_{\cat{B}} \cat{E}
    \end{tikzcd}
  \]
  factors through $\cat{K} \times_{\cat{B}} \FibOb{\cat{B}}{\cat{E}}$.
  Because $d$ is a diagram in the cartesian pushforward, its transpose $d^\dagger$ restricts to a functor $\cat{K} \times_{\cat{B}} \FibOb{\cat{B}}{\cat{E}} \to \FibOb{\cat{B}}{\cat{F}}$.
  From the diagram
  \[
    \begin{tikzcd}[row sep=large]
      \cat{K} \ar{dr}[below left]{b} \ar{r}{\corr{P}{\beta}(\cat{K} \times e)} &[5em] \cat{K} \times_{\cat{B}} \FibOb{\cat{B}}{\cat{E}} \ar{r}{d^\dagger} & \FibOb{\cat{B}}{\cat{F}} \ar{dl}{\FibOb{\cat{B}}{Q}} \\
      & \cat{B}
    \end{tikzcd}
  \]
  it follows by \ref{vk-existence} that the family $d^\dagger \circ \corr{P}{\beta} \co \cat{K} \times \cat{E}_{b_0} \to \cat{F}$ admits colimits pointwise in $e \in \cat{E}_{b_0}$, defining a family of colimits $f_0 \co \cat{E}_{b_0} \to \FibOb{\cat{B}}{\cat{F}}$ and colimit cocones $\phi \co d^\dagger \circ \corr{P}{\beta} \to f_0\projr$.
  Since $V$ is a fibred functor, the images of these cocones by $V$ land in $\FibOb{\cat{B}}{\cat{E}}$.
  By \ref{vk-restrict}, then, they are also colimit cocones; since $Vd^\dagger \circ \corr{P}{\beta} = \projr \circ \corr{P}{\beta} \co \cat{K} \times \cat{E}_{b_0} \to \cat{E}$, this means that $Vf_0 \cong \subst{P}b_0 \co \cat{E}_{b_0} \to \cat{E}$.
  Since $V$ is an isofibration, we can assume without loss of generality that $Vf_0 = \subst{P}b_0$ and that $V\phi = \lift{P}{\beta}$.
  In this case, $f_0$ and $\phi$ correspond (using the description of \cref{pushforward-morphisms}) to an object $f_0^\dagger \in \rancob{P}{V}$ and cocone $\phi^{\ddagger} \co d \to \Delta f_0^\dagger$.
  
  It remains to check that $\phi^{\ddagger}$ is a colimit cocone.
  Let $\xi \co d \to \Delta x$ be a cocone under $d$.
  Write $b_x = \ranc{P}V(x) \in \cat{B}$ and $\beta' = \ranc{P}V \circ \xi \co b \to \Delta b_x$; we have a unique morphism $[\beta'] \co b_0 \to b_x$ with $\Delta [\beta'] \circ \beta' = \beta$.
  By \cref{pushforward-morphisms}, $\xi$ corresponds to a transformation $\xi^\ddagger \co d^\dagger \circ \corr{P}{\beta'} \to x^\dagger\projr$ with $V\xi^\ddagger = \lift{P}{\beta'}$.
  By the universal properties of $\phi(-,\subst{[\beta']}e)$ for $e \in \cat{E}_{b_x}$, $\xi^\ddagger$ induces a natural transformation $f_0\subst{[\beta']} \to x^\dagger$, valued in $Q$-cartesian morphisms, which transposes to the desired morphism $f_0^\dagger \to x$ in $\rancob{P}{V}$ over $[\beta']$.
\end{proof}

\subsubsection{Saturation for extension operations}

\begin{definition}
  Given a Grothendieck fibration $P \co \cat{F} \to \cat{E}$, write $\FullCartArr{\cat{F}} \hookrightarrow \Arr{\cat{F}}$ for the full subcategory of $\Arr{\cat{F}}$ consisting of the $P$-cartesian morphisms and $\FullCartArr{P} \co \FullCartArr{\cat{F}} \to \Arr{\cat{E}}$ for the restriction of $P$ to this category.
\end{definition}

\begin{proposition}
  If $P \co \cat{F} \to \cat{E}$ is a Grothendieck fibration, then $\FullCartArr{P} \co \FullCartArr{\cat{F}} \to \Arr{\cat{E}}$ is also a Grothendieck fibration.
\end{proposition}
\begin{proof}
  By the cancellation properties of cartesian morphisms (if $g$ is cartesian and $gf$ is cartesian, then $f$ is cartesian).
\end{proof}

\begin{definition}[Category of extension operations]
  \label{extension-vertical-morphisms}
  Given a Grothendieck fibration $P \co \cat{F} \to \cat{E}$, define $\VertArr{U_P} \co \Ext{P} \to \Arr{\cat{E}}$ to be the cartesian pushforward
  \[
    \begin{tikzcd}
      \FullCartArr{\cat{F}} \ar{d}[left]{\leibpullapp{\dom}(P)} & \rancob{\projl}{\leibpullapp{\dom}(P)} \ar[dashed]{d} & \\
      \Arr{\cat{E}} \times_{\cat{E}} \cat{F} \ar{r}[below]{\projl} & \Arr{\cat{E}} 
    \end{tikzcd}
  \]
  of $\leibpullapp{\dom}(P) \defeq \pair{\FullCartArr{P}}{\dom} \co \FullCartArr{\cat{F}} \to \Arr{\cat{E}} \times_{\cat{E}} \cat{F}$, seen as a fibred functor between the Grothendieck fibrations $\FullCartArr{P} \co \FullCartArr{\cat{F}} \to \Arr{\cat{E}}$ and $\projl \co \Arr{\cat{E}} \times_{\cat{E}} \cat{F} \to \Arr{\cat{E}}$.
\end{definition}

An object of $\Ext{P}$ over a morphism $f \co A \to B$ in $\Arr{\cat{E}}$ thus corresponds to a section $\bm{f} \co \cat{F}_A \to \FullCartArr{\cat{F}}_f$ of the domain projection $\dom \co \FullCartArr{\cat{F}}_f \to \cat{F}_A$, \ie, a functor that extends any $\cat{F}$-structure over $A$ to an $\cat{F}$-structure over $B$ together with a $P$-cartesian morphism over $f$ from the input to the output structure.

\begin{definition}[Double category of extension operations]
  \label{extension-double-category}
  Let $P \co \cat{F} \to \cat{E}$ be a Grothendieck fibration.
  We extend $\VertArr{U_P} \co \Ext{P} \to \Arr{\cat{E}}$ of \cref{extension-vertical-morphisms} to a notion of composable structure $U_P \co \DExt{P} \to \Sq{\cat{E}}$ by taking the vertical identity functor $\idv_{(-)} \co \cat{E} \to \Ext{P}$ to be the transpose of the diagram
    \[
      \begin{tikzcd}
        \cat{F} \ar{dr}[below left]{\pair{\id_{P(-)}}{\Id_{\cat{F}}}} \ar{rr}{\id_{(-)}} & & \FullCartArr{\cat{F}} \ar{dl}{\leibpullapp{\dom}(P)} \\
        & \Arr{\cat{E}} \times_{\cat{E}} \cat{F}
      \end{tikzcd}
    \]
  and vertical composition $\star \co \Ext{P} \times_{\cat{E}} \Ext{P} \to \Ext{P}$ to be the transpose of the diagram
    \[
      \begin{tikzcd}[row sep=large]
        \Ext{P} \times_{\cat{E}} \Ext{P} \times_{\cat{E}} \cat{F} \ar{r}{\Ext{P} \times_{\cat{E}} \epsilon} \ar{dr}[below left]{\VertArr{U_P} \times_{\cat{E}} \VertArr{U_P} \times_{\cat{E}} \cat{F}} &[2em] \Ext{P} \times_{\cat{E}} \FullCartArr{\cat{F}} \ar{rr}{\pair{\epsilon(\Ext{P} \times_{\cat{E}} \dom)}{\projr}} && \FullCartArr{\cat{F}} \times_{\cat{F}} \FullCartArr{\cat{F}} \ar{r}{\circ} &[-1em] \FullCartArr{\cat{F}} \ar{ddll}{\leibpullapp{\dom}(P)} \rlap{.} \\
        & \Arr{\cat{E}} \times_{\cat{E}} \Arr{\cat{E}} \times_{\cat{E}} \cat{F} \ar{dr}[below left]{(\circ) \times_{\cat{E}} \cat{F}}  \\[-1em]
        & & \Arr{\cat{E}} \times_{\cat{E}} \cat{F}
      \end{tikzcd}
    \]
    Conservativity of $\VertArr{U_P}$ follows from the description of morphisms in $\Ext{P}$ in \cref{cartesian-pushforward-morphisms}.
\end{definition}

\begin{proposition}
  \label{extension-dcat-left-connected}
  $U_P \co \DExt{P} \to \Sq{\cat{E}}$ is left-connected.
\end{proposition}
\begin{proof}
  A square with boundary of the form
  \[
    \begin{tikzcd}
      A \ar[verto]{d}[left]{\idv_A} \ar{r}{h} & B \ar[verto]{d}{\bm{g}} \\
      A \ar{r}[below]{k} & C
    \end{tikzcd}
  \]
  in $\DExt{P}$ consists by \cref{cartesian-pushforward-morphisms} of a natural isomorphism $\theta$ such that
  \[
    \begin{tikzcd}[sep=large]
      \cat{F}_{A} \ar[phantom]{dr}{\theta} \ar{r}{\id_{(-)}} & \FullCartArr{\cat{F}}_{\id_A} \ar[phantomcenter]{dr}{\cong} \ar{r}{\dom} & \cat{F}_A \\
      \cat{F}_{B} \ar{u}{\subst{h}} \ar{r}[below]{\bm{g}^\dagger} & \FullCartArr{\cat{F}}_g \ar{u}[description]{\subst{(h,k)}} \ar{r}[below]{\dom} & \cat{F}_B \ar{u}[right]{\subst{h}}
    \end{tikzcd}
    \quad
    =
    \quad
    \begin{tikzcd}[sep=large]
      \cat{F}_A \ar[equals]{r} & \cat{F}_A \\
      \cat{F}_B \ar{u}{\subst{h}} \ar[equals]{r} & \cat{F}_B \ar{u}[right]{\subst{h}} \rlap{.}
    \end{tikzcd}
  \]
  Because any vertical cartesian morphism is an isomorphism, any pair of objects of $\FullCartArr{\cat{F}}_{\id_A}$ are related by a unique domain-fixing isomorphism.
  Thus there is a unique such square provided $gh = k$.
\end{proof}

\begin{proposition}
  \label{extension-dcat-retracts}
  $U_P \co \DExt{P} \to \Sq{\cat{E}}$ admits a codomain retract lifting operator.
\end{proposition}
\begin{proof}
  Suppose we have a codomain retract diagram
  \[
    \begin{tikzcd}
      & A \ar{dl}[above left]{f'} \ar{d}{f} \ar{dr}{f'} \\
      B' \ar{r}[below]{s} & B \ar{r}[below]{r} & B'
    \end{tikzcd}
  \]
  and a vertical morphism over $f$, which is to say a section $\xi \co \cat{F}_A \to \FullCartArr{\cat{F}}_f$ of the domain projection.
  Post-composing with $\subst{(\id_A,s)} \co \FullCartArr{\cat{F}}_f \to \FullCartArr{\cat{F}}_{f'}$ yields a vertical morphism over $f'$.\footnote{Here we abusively write $\subst{(\id_A,s)}$ for the functor which performs the canonical lift along $s$ in the codomain and is the identity in the domain, though we have not assumed that $\subst{\id_A}$ is the identity.}
  This operation can clearly be made functorial in the retract diagram.
\end{proof}

\begin{lemma}
  \label{pullback-van-kampen}
  Let $P \co \cat{E} \to \cat{B}$ be a Grothendieck fibration, $F \co \cat{C} \to \cat{B}$ be a functor, and $c \co \cat{K} \to \cat{C}$ be a diagram.
  If $\psi \co c \to \Delta c_0$ is a colimit cocone such that $F\psi$ is a Van Kampen colimit cocone for $P$, then $\psi$ is Van Kampen for $\subst{F}P$.
\end{lemma}
\begin{proof}
  This is a straightforward consequence of the fact that $\FibOb{\cat{C}}{(\cat{C} \times_{\cat{B}} \cat{E})} \cong \cat{C} \times_{\cat{B}} \FibOb{\cat{B}}{\cat{E}}$.
\end{proof}

\begin{lemma}
  \label{fullcartarr-van-kampen}
  Let $P \co \cat{E} \to \cat{B}$ be a Grothendieck fibration and $d \co \cat{K} \to \Arr{\cat{B}}$ be a diagram.
  If $\delta \co d \to \Delta d_0$ is a colimit cocone such that ${\dom} \circ \delta$ and ${\cod} \circ \delta$ are Van Kampen for $P$, then $\delta$ is Van Kampen for $\FullCartArr{P} \co \FullCartArr{\cat{E}} \to \Arr{\cat{B}}$.
\end{lemma}
\begin{proof}
  Since colimits in arrow categories are computed pointwise, it suffices to check that for any diagram $e \co \cat{K} \to \FullCartArr{\cat{E}}$ over $d$, the induced morphism $\colim ({\dom} \circ e) \to \colim ({\cod} \circ e)$ in $\cat{E}$ is $P$-cartesian.
  This is the case because $\colim ({\dom} \circ e)$ is a colimit in $\FibOb{\cat{B}}{\cat{E}}$, because ${\dom} \circ \delta$ is Van Kampen for $P$.
\end{proof}

By the saturation theorem, we conclude:

\begin{theorem}
  \label{lift-extension-operation}
  \input{lift-extension-operation}
\end{theorem}
\begin{proof}
  By \cref{soa-copointed-coalgebras-functor}.
  \Cref{extension-dcat-left-connected} gives left-connectedness and \cref{extension-dcat-retracts} provides a codomain retract lifting operator.
  For colimits, \cref{pullback-van-kampen,fullcartarr-van-kampen} imply that colimits of $(1+\alpha)$-chains in $\CodClass{\cat{E}}{\class{M}}$ for $\alpha \preceq \kappa$ and cobase changes of maps in $\CodClass{\cat{E}}{\class{M}}$ are Van Kampen for $\FullCartArr{P} \co \FullCartArr{F} \to \Arr{\cat{E}}$ and $\projl \co \Arr{\cat{E}} \times_{\cat{E}} \cat{F} \to \Arr{\cat{E}}$.
  Thus \cref{pushforward-colimit} implies that $\DCodClass{\DExt{P}}{\class{M}}$ is $(\kappa,\class{M})$-cellular.
\end{proof}

Finally, we return to the motivating example of extension of fibrations along trivial cofibrations.
When $(\comonad{L},\monad{R})$ is an \textsc{awfs} on a category $\cat{E}$ with pullbacks, the composite $\cod U_{\monad{R}} \co \Alg{\monad{R}} \to \cat{E}$ is a Grothendieck fibration \cite[Proposition 8]{bourke-garner:16}.
When the \textsc{awfs} is generated by a diagram of maps with tiny codomain, this fibration inherits the Van Kampen colimits of the underlying category:

\begin{lemma}
  \label{awfs-algebras-van-kampen}
  Let $u \co \cat{J} \to \Arr{\cat{E}}$ be a diagram of arrows in a category with pullbacks and let $(\comonad{L},\monad{R})$ be an \textsc{awfs} on $\cat{E}$ cofibrantly generated by $u$.
  Suppose that $\cod u$ is levelwise tiny, in the sense that $\Hom{\cat{E}}{\cod u_i}{-}$ preserves all colimits for all $i \in \cat{J}$.
  Given a diagram $Y \co \cat{K} \to \cat{E}$ and colimit cocone $\upsilon \co Y \to \Delta Y_0$, if $\upsilon$ is Van Kampen for $\cod \co \Arr{\cat{E}} \to \cat{E}$, then $\upsilon$ is also Van Kampen for $\cod U_{\monad{R}} \co \Alg{\monad{R}} \to \cat{E}$.
\end{lemma}
\begin{proof}
  By definition of cofibrantly generated \textsc{awfs}, an $\monad{R}$-algebra structure on $f \in \Arr{\cat{E}}$ can be described as a section of a map
  \[
    \begin{tikzcd}[sep=large]
      \Hom{\cat{E}}{\cod u_{(-)}}{\dom f} \ar{r}{\leibpullapp{\conjugate{u}}(f)} & \Hom{\Arr{\cat{E}}}{u_{(-)}}{f}
    \end{tikzcd}
  \]
  that at stage $i \in \cat{J}$ sends a solution to a lifting problem $u_i \to f$ to the underlying lifting problem; see the proof of \cite[Proposition 16]{bourke-garner:16}.

  Let $\bm{f} \co \cat{K} \to \FibOb{\cat{E}}{(\Alg{\monad{R}})}$ be a diagram over $Y$.
  Then the composite $f \defeq U_{\monad{R}} \bm{f} \co \cat{K} \to \Arr{\cat{E}}$ factors through $\CartArr{\cat{E}}$, so by assumption admits a colimiting cone $\phi \co f \to f_0$ in $\CartArr{\cat{E}}$ over $\upsilon$.
  Since $\cod u$ is levelwise tiny, $\Hom{\cat{E}}{\cod u_i}{\upsilon}$ and $\Hom{\cat{E}}{\cod u_i}{\dom \phi}$ are colimit cocones in $\Set$.
  Because $f$ is a cartesian natural transformation, we have $\Hom{\Arr{\cat{E}}}{u_i}{f_k} \cong \Hom{\cat{E}}{\cod u_i}{Y_k} \times_{\Hom{\cat{E}}{\cod u_k}{Y_0}} \Hom{\Arr{\cat{E}}}{u_i}{f_0}$, and so $\Hom{\Arr{\cat{E}}}{u_i}{\phi}$ is also a colimit cocone by pullback stability of colimits in $\Set$.
  Thus the compatible family of sections for the maps $\leibpullapp{\conjugate{u}}(f_k)$ provided by $\bm{f}$ induces a unique compatible section of $\leibpullapp{\conjugate{u}}(f_0)$, providing a colimiting cocone $\bm{f} \to \Delta \bm{f}_0$ in $\FibOb{\cat{E}}{(\Alg{\monad{R}})}$ over $\upsilon$.
\end{proof}

This lemma applies in particular to the uniform fibrations \textsc{awfs}, and so we have:

\begin{example}
  \label{uniform-fibration-extension}
  Let $(t,I)$ be a finitary uniform fibration configuration in a presheaf category $\cat{E}$.
  Write $P \co \Fib{t}{I} \to \cat{E}$ for the category of monad algebras of the associated \textsc{awfs} with its codomain projection.
  Given lifts of $\leibpushapp{\delta_i}\phi^t \co \Slice{\cat{E}}{\Cof} \to \Arr{\cat{E}}$ through $\Ext{P}$ for $i \in \braces{0,1}$, we can construct a functor $\TCof{t}{I} \to \Ext{P}$ over $\Arr{\cat{E}}$ assigning a uniform fibration extension operation to every trivial cofibration.
\end{example}
\begin{proof}
  By \cref{lift-extension-operation} applied with the $\omega$-backdrop of levelwise complemented monomorphisms and \cref{awfs-algebras-van-kampen}, it suffices to check that colimits of $\alpha$-chains of levelwise complemented monomorphisms and cobase changes of levelwise complemented monomorphisms are Van Kampen in $\cat{E}$.
  This is the case in any presheaf category \cite{lack-sobocinski:06}~\cite[\S3]{shulman:15}.
\end{proof}

The generating case of \cref{uniform-fibration-extension}, namely extension along maps of the form $\leibpushapp{\delta_i}(m)$ where $m$ is a cofibration, can be treated directly in various cubical settings, as in \textcite[Lemma 7.3 and Corollary 7.4]{sattler:17} and \textcite[Proof of Proposition 120]{awodey:23}.

\begin{example}
  \Textcite[\S3]{shulman:19-all} introduces the framework of \emph{notions of fibred structure} to axiomatize algebraic incarnations of classes of maps such as fibrations and trivial fibrations that are closed under base change.
  In particular, the categories $\Alg{\monad{R}}$ and $\Alg{\UnderPtd{\monad{R}}}$ associated to an \textsc{awfs} $(\comonad{L},\monad{R})$ on a category $\cat{E}$ with pullbacks give rise to notions of fibred structure in $\cat{E}$ \cite[Example 3.7]{shulman:19-all}~\cite[Proposition 8]{bourke-garner:16}.

  Given a notion of fibred structure on a category $\cat{E}$, we have in particular a Grothendieck fibration $P \co \cat{F} \to \cat{E}$ from a category of structured maps $\cat{F}$ that projects the codomain of a structured map.
  Shulman's Proposition 3.18(i)$\Rightarrow$(iii) implies that for a \emph{locally representable} notion of fibred structure on $\cat{E}$, any colimit in $\cat{E}$ that is Van Kampen for the codomain fibration $\cod \co \Arr{\cat{E}} \to \cat{E}$ is also Van Kampen for the induced fibration $P$.
  Provided that $\cat{E}$ has enough Van Kampen colimits, extension operations for such structures are thus susceptible to \cref{lift-extension-operation}.
  Indeed, \cref{uniform-fibration-extension} could be factored through a proof that the notion of fibred structure corresponding to uniform fibration monad algebras is locally representable.
\end{example}

%% file: related.tex
\section{Related work and unanswered questions}

\subsection{Coalgebraic cell complexes}

Our work complements that of Athorne on \emph{coalgebraic cell complexes} \cite{athorne:12,athorne:14}.
Athorne shows that under various conditions \cite[\S3.3]{athorne:14} on the category $\cat{E}$ and generating diagram $u \co \cat{J} \to \Arr{\cat{E}}$, the double category $\DCoalg{\comonad{L}}$ of comonad coalgebras for the \textsc{awfs} cofibrantly generated by $u$ is equivalent to a double category of formal cell complexes \cite[Theorem 3.5.1]{athorne:14}.
Compared to Athorne's results, ours are weaker in that we do not give any \emph{characterization} of the double category of coalgebras.
Our \cref{soa-copointed-coalgebras-double-functor} produces a pseudo double functor from $\DCoalg{\UnderPtd{\comonad{L}}}$ into some target, but it does not show that this functor is in some sense uniquely determined, so cannot be used as a universal property to derive an equivalence.

For our less ambitious goal of extending various structures from generators to left maps, our theorems offer a simpler interface, require checking fewer conditions, and are more general.
In particular, our results apply to \emph{any} \textsc{awfs} cofibrantly generated by a diagram that satisfies our smallness conditions in a cocomplete category, though their usefulness depends on the \textsc{awfs} and choice of backdrop.

We note that our backdrops play a role analogous to that of Athorne's class of \emph{typical inclusions} \cite[Definition 3.3.1]{athorne:12} (which is induced by a choice of \emph{typical nerve} \cite[Definition 3.3.2]{athorne:12}).
Whereas Athorne requires the typical inclusions to be monomorphisms, the conditions on a backdrop are always satisfied by the class of all morphisms of a category with the necessary colimits; that said, our concrete applications all use backdrops of monomorphisms.

We hope that future research can unify these two approaches to understanding left maps.
Ideally, improved versions of our results could be used to cleanly prove an equivalence between $\DCoalg{\comonad{L}}$ and a double category of cell complexes such as Athorne's, at least in some cases.
To reach that point, we would need in particular to resolve the following issues involving retract lifting.

\subsection{Retract lifting}

Our \cref{soa-copointed-coalgebras-functor,soa-copointed-coalgebras-double-functor} can assign some structure to every left map of a cofibrantly generated \textsc{awfs}, but they say very little about the assignment except that it is functorial.
In particular, there is no guarantee that the assignment is in some sense unique, nor that it relates to the input data in some way; we only get such guarantees in our \cref{soa-unit-vertical-point-abstract}, which is restricted to free coalgebras.
We extend the assignment from free coalgebras to all copointed endofunctor coalgebras by assuming a retract lifting operator on the target, and we know of no conditions we could put on this operator that would ensure uniqueness.

To see the difficulty, consider what should be the prototypical notion of composable structure: a double category of left maps with its forgetful functor $U_{\UnderPtd{\comonad{L}}} \co \DCoalg{\UnderPtd{\comonad{L}}} \to \Sq{\cat{E}}$.
It lifts retracts in the sense that given a section-retraction pair $f' \overset{\sigma}\to f \overset{\rho}\to f'$ in $\Arr{\cat{E}}$ and a coalgebra structure $\beta \co f \to Lf$ on $f$, we get a coalgebra structure $\beta' \defeq L\rho \circ \beta \circ \sigma \co f' \to Lf'$ on $f'$.
However, the retract diagram itself need not lift: $\sigma$ and $\rho$ need not be morphisms of coalgebras between $(f,\beta)$ and $(f',\beta')$.
In particular, ``having retract lifts'' is apparently a \emph{structure} rather than a \emph{property}; we can at least ask that the choice is functorial (and interacts nicely with composition of left maps, see \cref{compositional-retract-lifting}), but we know of no characteristic property of $(f,\beta')$ that abstractly singles it out among all lifts of $(f,\beta)$ along the retract.

One may also note that while \cref{soa-unit-vertical-point-abstract} deals with free coalgebras and \cref{soa-copointed-coalgebras-functor,soa-copointed-coalgebras-double-functor} deal with copointed endofunctor coalgebras, we have no theorem for mapping out of the intermediate category of \emph{comonad} coalgebras.
In that case, the natural requirement would be that the target notion of composable structure is \emph{closed under retract equalizers} in the sense of Garner \cite[Proposition 44]{garner:07}.
As used in Beck's monadicity theorem, any comonad coalgebra $(f, \beta \co f \to Lf)$ can be written as a retract equalizer of free coalgebras in $\Coalg{\comonad{L}}$, namely the equalizer of the fork
\[
  \begin{tikzcd}
    Lf \ar[yshift=3pt]{r}{L\beta} \ar[yshift=-3pt]{r}[below]{\Sigma_f} & LLf \rlap{.}
  \end{tikzcd}
\]
Given a notion of composable structure $U \co \dcat{A} \to \Sq{\cat{E}}$ and a lift $\bm{L}_{\dcat{A}} \co \Arr{\cat{E}} \to \VertArr{\dcat{A}}$ of the free coalgebra functor through $\dcat{A}$, one would like to extend the mapping to comonad coalgebras by assuming $U$ is closed under retract equalizers and sending $(f, \beta \co f \to Lf)$ to the equalizer of a fork
\[
  \begin{tikzcd}
    \bm{L}_{\dcat{A}}f \ar[yshift=3pt]{r}{\bm{L}_{\dcat{A}}\beta} \ar[yshift=-3pt]{r}[below]{?} & \bm{L}_{\dcat{A}}Lf \rlap{.}
  \end{tikzcd}
\]
However, we do not see how to arrange that $\Sigma_f$ lifts to a morphism $\bm{L}_{\dcat{A}} \to \bm{L}_{\dcat{A}}Lf$, nor alternatively how to adjust the retract equalizer closure condition to compensate for the absent map.

%% file: metatheory.tex
\section{Constructivity}
\label{sec:metatheory}

Our work is motivated by problems in semantics of homotopy type theory \cite{hott-book}.
Homotopical interpretations of type theory typically make use of \textsc{wfs}'s.
Because homotopy type theories are nominally constructive logics, there is a particular interest in building and studying their semantics within constructive metatheories.
However, working constructively requires substantial changes to standard uses of \textsc{wfs}'s in homotopy theory.
The existence of the (mono, epi) \textsc{wfs} in $\Set$, which is the basis for (mono, trivial fibration) \textsc{wfs}'s used in various model structures on presheaf categories, is equivalent to the axiom of choice.
Even the existence of a (mono, $\rlp{\text{mono}}$) \textsc{wfs} on $\Set$ is equivalent to $\Set$ having enough injectives \cite[Proposition 1.6]{adamek-herrlich-rosicky-tholen:02}, which may fail in predicative metatheories \cite{aczel-van-den-berg-granstrom-schuster:13}.
Instead, one can work with the (complemented mono, split epi) \textsc{wfs} generated by the singleton set $\braces{0 \mono 1}$.

The algebraic small object argument is particularly relevant to constructive homotopy theory for several reasons.
First, it is natural in the constructive setting to work with lifting \emph{structures} as opposed to lifting \emph{properties}, which the theory of \textsc{awfs}'s facilitates.
Second, Quillen's argument sometimes requires the axiom of choice where Garner's does not; see for example the discussion in \textcite[\S C]{henry:20} and the use of decidability assumptions to avoid choice in Quillen's argument \textcite[\S 2.2]{gambino-sattler-szumilo:22}.
The quotienting in Garner's argument makes these choices canonical.
Finally, relevant \textsc{awfs}'s such as uniform fibration \textsc{awfs}'s (\cref{sec:uniform-fibrations}) can be generated by a diagram but are apparently not constructively generated by any discrete set of generators.

Two points in our development deserve some comment from a constructive perspective: the uses of transfinite interation and of colimits.

\subsubsection{Ordinals and transfinite iteration}

The algebraic small object argument rests on a transfinite iteration of a certain well-pointed endofunctor.
This iteration can be seen to converge when the domains of the generating arrows are \emph{small} in the sense that mapping out of them commutes with colimits of sufficiently long chains.

In a classical metatheory, smallness requirements are often dealt with by assuming that the category $\cat{E}$ is \emph{locally presentable}, \ie, that there is a regular cardinal $\lambda$ and a set $S$ of $\lambda$-compact objects of $\cat{E}$ such that every object is a $\lambda$-filtered colimit of objects in $S$.
In a classical metatheory, this implies that for every object $A \in \cat{E}$ there is some regular $\kappa$ such that $\Hom{\cat{E}}{A}{-}$ preserves $\kappa$-filtered colimits \cite[Remark 2.15(1)]{adamek-rosicky:94}, in particular colimits of $\kappa$-chains.
Constructively, however, we may not have this implication.

Indeed, much of the classical theory of ordinals (and in turn of locally presentable categories) depends unavoidably on the law of the excluded middle and the axiom of choice.
In practice, one is limited to applying the algebraic small object argument with generators whose domains are $\omega$-compact, as we do with uniform fibrations \textsc{awfs}'s in \cref{sec:uniform-fibrations}.

Our main results, such as \cref{soa-unit-vertical-point,soa-copointed-coalgebras-functor}, are proven constructively but require as input a suitable limit ordinal $\kappa > 0$; the burden is on the user of the theorem to find an ordinal for which the generators are small.
In this way, we can state a single theorem that the constructivist can apply with the unproblematic choice $\kappa = \omega$ and the classical mathematician can apply with larger ordinals.
Behind the scenes, this requires that we define the central transfinite iteration in a constructively acceptable way even for $\kappa$ other than $\omega$, which we do by following \textcite{pitts-steenkamp:22}; see \cref{sec:well-pointed}, where we give the construction, for more on this point.

\textcite{seip:24} has already given a careful constructive treatment of Garner's small object argument for the case where $\kappa = \omega$.
\Textcite{hilhorst-north:24} have also produced a computer formalization of Garner's argument for $\kappa = \omega$ using the \texttt{UniMath} library \cite{unimath} of the Rocq proof assistant, which implements a variant of the Calculus of Inductive Constructions \cite{coquand-huet:88,paulin-mohring:93}.\footnote{\texttt{UniMath} extends Rocq's foundation with the axioms of Voevodsky's Univalent Foundations. This includes axioms of \emph{propositional resizing} for which there are as yet no known constructive models. However, \citeauthor{hilhorst-north:24}'s formalization does not use the full power of Univalent Foundations but only function extensionality and propositional truncation \cite[\S1]{hilhorst-north:24}, which do have constructive justifications.}

\subsubsection{Colimits}

Some constructive metatheories, such as Martin-L\"of type theory, do not assume the existence of general quotients.
The small object arguments are colimit constructions, so superficially would seem to be inapplicable in these contexts.
Even if we only assume $\Set$ has coproducts, however, we can use them to construct some additional colimits, in particular cobase changes along \emph{complemented} monomorphisms and colimits of $\omega$-chains of the same (see \cref{complemented-backdrop}).
As a corollary, the class of \emph{levelwise complemented} monos in any presheaf category $\PSh{\cat{C}}$ forms an $\omega$-backdrop.
Thus, if a category of generators $u \colon \cat{J} \to \Arr{\PSh{\cat{C}}}$ admits a density comonad $\Den{u}$ satisfying enough decidability assumptions, then it generates an \textsc{awfs} even in this weak setting.
In \cref{sec:uniform-fibrations}, we give conditions under which uniform box-filling fibration \textsc{awfs}'s can be constructed with this backdrop.
As with ordinals, parameterizing by a backdrop allows us to prove our results constructively without restricting their applicability for the classical mathematician.

The fact that we do not need general colimits in the metatheory is particularly interesting because these \textsc{awfs}'s can be used to interpret type theories that \emph{do} have colimits in the form of higher inductive types \cite{coquand-huber-mortberg:18,cavallo-harper:19}.
We see this as a higher analogue of setoid model constructions, which also interpret type theories with quotients in type theories without quotients \cite{hofmann:97,altenkirch-boulier-kaposi-tabareau:19}; compare also \citeauthor{gambino-henry-sattler-szumilo:22}'s effective model structure on simplicial objects in a countably lextensive category \cite{gambino-henry-sattler-szumilo:22}.